\documentclass[english, final,letterpaper,11pt]{amsart}

\pdfoutput=1

\usepackage[utf8]{inputenc}
\usepackage[T1]{fontenc}
\usepackage{lmodern}
\usepackage{amsthm, amssymb, amsmath, amsfonts, mathrsfs}
\usepackage[dvipsnames,svgnames,x11names]{xcolor}

\usepackage[draft,defaultcolor=red]{changes}
\usepackage{microtype}
\usepackage{bbm}

\usepackage[inline]{showlabels}
\usepackage{enumerate}

\usepackage[pagebackref,colorlinks=true, urlcolor=blue,
linkcolor=blue,citecolor=blue]{hyperref}
\usepackage{bookmark}

\usepackage{amsmath,amsfonts,amssymb,amsthm}
\usepackage{amsthm, amssymb, amsmath, amsfonts, mathrsfs}

\usepackage{mathrsfs}
\usepackage{MnSymbol}
\usepackage{scalerel} 
\usepackage{color}
\usepackage{accents}

\definecolor{labelkey}{gray}{.8}
\definecolor{refkey}{gray}{.8}

\definecolor{darkred}{rgb}{0.9,0.1,0.1}
\definecolor{darkgreen}{rgb}{0,0.5,0}

\usepackage[margin=1in,footskip=.25in]{geometry}

\newtheorem{theorem}{Theorem}[section]
\newtheorem{lemma}[theorem]{Lemma}
\newtheorem{definition}[theorem]{Definition}
\newtheorem{corollary}[theorem]{Corollary}
\newtheorem{proposition}[theorem]{Proposition}

\newtheorem{Hypothesis}{Hypothesis} 

\theoremstyle{remark}
\newtheorem{remark}[theorem]{Remark}

\renewenvironment{proof}[1][Proof]{ {\itshape \noindent {#1.}} }{$\Box$
\medskip}

\numberwithin{equation}{section}
\numberwithin{figure}{section}
\theoremstyle{plain}

\newcommand{\R}{\mathbb{R}}

\newcommand{\1}{\mathbbm{1}}
\newcommand{\eps}{\varepsilon}

\newcommand{\malD}{\mathcal{D}}
\newcommand{\dd}{\mathrm{d}}
\newcommand{\supp}{\mathrm{supp}}

\newcommand{\Var}{\mathbf{Var}}
\newcommand{\Cov}{\mathbf{Cov}}
\newcommand{\filt}{\mathscr{F}}
\newcommand{\Prob}{\mathbf{P}}
\newcommand{\Expe}{\mathbf{E}}

\newcommand{\kernel}{\mathscr{P}}

\newcommand{\viv}{\hspace{.7pt}\vert\hspace{.7pt}}

\title{Fluctuation exponents of the half-space KPZ  at stationarity}

\author[Yu Gu]{Yu Gu}
\author[Ran Tao]{Ran Tao}
 \address[Yu Gu]{Department of Mathematics, University of Maryland,\newline 4176 Campus Drive, College Park, MD, 20742.}\email{yugull05@gmail.com}
 \address[Ran Tao]{Department of Mathematics, University of Maryland,\newline 4176 Campus Drive, College Park, MD, 20742.}  \email{rantao16@umd.edu}
\begin{document}

\begin{abstract}
We study the half-space KPZ equation with a Neumann boundary condition, starting from stationary Brownian initial data. We derive a variance identity that links the fluctuations of the height function to the transversal fluctuations of a half-space polymer model. Utilizing this identity, we obtain estimates for the polymer endpoints, leading to optimal fluctuation exponents for the height function in both the subcritical and critical regimes, {as well as an optimal upper bound for the fluctuation exponents in
{the extended critical regime}.} 
We also compute the average growth rate as a function of the boundary parameter.

\bigskip

\end{abstract}

\maketitle

\section{Introduction}
\subsection{Main Results}
The Kardar-Parisi-Zhang (KPZ) equation in half-space with a Neumann boundary condition at the origin models stochastic interface growth in contact with a boundary \cite{kardar1985depinning, ito2018fast}. Indexed by the boundary condition parameter $u \in \R$, this equation is the stochastic PDE:
\begin{equation}\label{eq:hsKPZ}
\begin{aligned}
\partial_t H_{u}(t, x) & =\frac{1}{2} \partial_x^2 H_{u}(t, x)+\frac{1}{2}\left(\partial_x H_{u}(t, x)\right)^2+\xi(t, x), \quad t\geq 0, x \geq 0,
\\ \left.\partial_x H_u(t, x)\right|_{x=0} & =u, 
\end{aligned}
\end{equation}
where $\xi$ is a (1+1) space-time white noise on $(t,x)\in \R^2$. Throughout the paper, we assume that the initial condition to the half-space KPZ equation \eqref{eq:hsKPZ} is 
\begin{equation}\label{eq:i.c.}
H_u(0,x)=W_u(x):=W(x)+ux, 
\end{equation}
with $W(\cdot)$ being a standard Brownian motion starting from 0 and independent of $\xi$. We consider the solution to \eqref{eq:hsKPZ} by the Hopf-Cole transform 
\[
H_u(t,x):= \log Z_{u-\frac{1}{2}} (t,x),
\]
where $Z_{\mu}(t,x)$ solves the half-space stochastic heat equation (SHE) with a Robin boundary condition for $\mu \in \R$:
\begin{equation}\label{eq:hsSHE}
\begin{aligned}
\partial_t Z_{\mu}(t, x) & =\frac{1}{2} \partial_x^2 Z_{\mu}(t, x)+Z_{\mu}(t, x)\xi(t, x), \quad t\geq 0, x \geq 0,\\ \left.\partial_x Z_\mu(t, x)\right|_{x=0} & =\mu Z_\mu(t, 0), 
\end{aligned}\end{equation}
with initial condition given by ${Z}_\mu(0, x)=\exp({H_{\mu+\frac{1}{2}}(0, x)})$.

The boundary conditions in \eqref{eq:hsKPZ} or \eqref{eq:hsSHE} are not immediately meaningful, as the solutions to the KPZ equations or the stochastic heat equations are not differentiable. 
We defer precise definitions to Section~\ref{se:prelim}. We also note that the parametrization here follows the convention as in \cite{barraquand2023stationary}. Under this convention, $H_u(t,x)$ is the logarithm of a half-space SHE with Robin boundary parameter $u-\frac{1}{2}$ and starting at initial condition $Z_{u-\frac{1}{2}}(0,x)= \exp(W(x)+ux)$. As it will become clear later, the shift of $-\frac{1}{2}$ introduces a certain symmetry over $u=0$.

{It was shown in \cite{barraquand2023stationary} that, with the initial data in \eqref{eq:i.c.}, the evolution of the height function $H_u$ is at stationarity, namely, for any $t>0$,  
 \begin{equation}\label{eq:station}
 H_u(t,x)-H_u(t,0) \stackrel{\text{law}}{=} H_u(0,x)-H_u(0,0) = W(x)+ux, \text{ as a process in } x \in [0, \infty).
 \end{equation}
 The goal of this paper is to study the statistics of the height function in this stationary setting, including the mean and the variance, for different choices of the boundary conditions.}

Let $(\Omega, \filt, \Prob)$ be a probability space that supports all possible randomness in the environment and the initial   data. We use $\Expe, \Var, \Cov$ to denote the expectation, variance and covariance, respectively, under this probability space. Given two functions   $a(t)$ and $b(t)$, we use notation $a(t) \sim b(t)$ when $\lim_{t\to \infty}a(t)/b(t)=1$, and use notation $a(t) \asymp b(t)$ when $C^{-1}b(t)\leq a(t) \leq Cb(t)$ for some constants $C>0$ independent of $t\geq 1$.

Associated with equations \eqref{eq:hsKPZ} and \eqref{eq:hsSHE}, one can construct a half-space continuum directed random polymer (half-space CDRP) measure $\mathbb{Q}^{u,t}_0$. We defer the precise definitions of $\mathbb{Q}^{u,t}_0$ to Definition \ref{de:cdrp}. One can think of it as a random Gibbs measure 
on the path space $\mathcal{C}([0,t],[0,\infty))$ with Borel $\sigma$-algebra. The path starts from $(t,0)$ and runs backwards in time, and it is restricted to the nonnegative half-space, attracted or repulsed by the boundary 0. 
Along the trajectory, it collects energy $\int_0^t \xi(t-s,X_s)ds$, and the terminal tilt is  given by $\exp(W_u(\cdot))$ at time $t$. In particular, the endpoint of the polymer path induces a measure $\mathbb{Q}^{u,t}_0 (X_t \in \dd x)$ on the half-space $[0,\infty)$ equipped with Borel $\sigma$-algebra $\mathcal{B}([0,\infty))$. Throughout the paper, we use  $\mathbb{E}^{\mathbb{Q}^{u,t}}_{0}$ as the quenched expectation over this polymer measure.

To describe the result, we will use the digamma function $\psi$ 
\begin{equation}\label{eq:digammafunc}
\psi(z)=\partial_z \log \Gamma(z)=-\gamma_0+\sum_{n=0}^{\infty}\left(\frac{1}{n+1}-\frac{1}{n+z}\right), \quad\quad z\in(0,\infty), 
\end{equation}
with $\gamma_0$ being the Euler–Mascheroni constant. Its derivative $\psi'$ (the trigamma function) is a strictly decreasing function on $(0,\infty)$ 
\begin{equation}\label{eq:trigammafunc}
\psi'(z)= \sum_{n=0}^{\infty}\frac{1}{(n+z)^2}.
\end{equation}
In particular, the leading behavior as $z\to0^+$ is
\begin{equation}\label{eq:psiasym}
\psi'(z) \sim \frac{1}{z^2}.
\end{equation}

Our first result below establishes a relation between the variance of  $H_u(t,0)$ and the annealed mean of the endpoint displacement of the polymer paths under $\mathbb{Q}^{u,t}_0$. This identity holds true for arbitrary $u \in \R, t\geq0$.
\begin{theorem}\label{thm:varianceID}
For any $u\in\R, t\geq0$,
\begin{equation}\label{eq:varidkpzhs}
\Var[H_u(t,0)] = 2\Expe \mathbb{E}^{\mathbb{Q}^{u,t}}_{0}[X_t] -ut,
\end{equation}
\end{theorem}

Next, we provide an estimate on $\Expe \mathbb{E}^{\mathbb{Q}^{u,t}}_{0}[X_t]$ under the constraint of $u<0$. The bound in turn leads to  the bounds on the fluctuations of  the height function.

\begin{theorem}\label{thm:var}
Assume  $u<0$.  For any $t\geq 0$, we have 
\begin{equation}\label{eq:intmidupbd}
0 \leq \Expe \mathbb{E}^{\mathbb{Q}^{u,t}}_{0}[X_t]\leq \psi'(2|u|).
\end{equation}

As a result,
\begin{equation}\label{eq:varbdexact}
|u|t \leq \Var [H_u(t,0)] \leq 2\psi'(2|u|)+|u|t.
\end{equation}
In particular, for
 any constants $\alpha\in[0,\frac{1}{3})$ and $c< 0$, with $u=ct^{-\alpha}$,
we have 
\begin{equation}\label{eq:varflucsize}
\Var [H_{ct^{-\alpha}}(t,0)] \sim |c|t^{1-\alpha}, \quad \text{as } t \to \infty.
\end{equation}
For $\alpha=\frac{1}{3}$ and $c<0$, with $u=ct^{-{1}/{3}}$, we have
\begin{equation}\label{eq:varfluccrit}
 \Var [H_{ct^{-{1}/{3}}}(t,0)]  \asymp t^{\frac{2}{3}}.
\end{equation}
\end{theorem}

As a consequence, we obtain the tightness for the annealed polymer endpoint measure. 

\begin{corollary}\label{co:tightnessleft}
For any constants $\alpha\in[0,\frac{1}{3}]$ and $c< 0$, with $u=ct^{-\alpha}$, if we denote the annealed law of the scaled endpoint displacement $\frac{X_t}{t^{2\alpha}}$ on $[0,\infty)$ by $\hat{\Prob}^{\mathbb{Q}}_{u,t}$, (i.e. for any $t>0$ and Borel subset $A\subset [0,\infty)$, we let
$\hat{\Prob}^{\mathbb{Q}}_{u,t}(\frac{X_t}{t^{2\alpha}}\in A)=\Expe \mathbb{E}^{\mathbb{Q}^{u,t}}_{0}[\1_{A}({X_t}/{t^{2\alpha}})]$,)
then the sequence of measures $\{\hat{\Prob}^{\mathbb{Q}}_{u,t}\}_{t> 0}$ is tight.
\end{corollary}

Independent of the above results, we also provide a formula for the mean of $H_u(t,x)$ when $u\leq 0$.
\begin{theorem}\label{thm:mean}
For any $u \leq 0$, $t\geq 0$, and $x\geq0$, 
\begin{equation}\label{e.quafree}
\Expe [H_u(t,x)]=\left(\frac{1}{2}u^2-\frac{1}{24}\right)t+ux.
\end{equation}
\end{theorem}

{The next results concern the cases when $u>0$.}
For this part, we 
rely on the following symmetry result inspired by \cite{barraquand2020halfkpz}. 
\begin{proposition}\label{pr:symmetry}
For any $u>0$ and $t\geq 0$, we have 
\[H_u(t,0) \stackrel{\text{law}}{=}H_{-u}(t,0).
\]
Consequently, \eqref{eq:varbdexact}--\eqref{eq:varfluccrit}
and Theorem~\ref{thm:mean} remain valid when $u>0$ (or correspondingly $c>0$).
\end{proposition}

Combining Theorem~\ref{thm:varianceID} and \eqref{eq:varbdexact}, we can also obtain the order of the polymer endpoint displacement when $u>0$.
\begin{corollary}\label{co:upbdright}
    For any $u>0, t \geq 0$, we have
\begin{equation}\label{eq:intright}
ut \leq \Expe \mathbb{E}^{\mathbb{Q}^{u,t}}_{0}[X_t]\leq ut+ \psi'(2u).
\end{equation}
In particular, for any   $\alpha\in[0,\frac{1}{3}]$ and $c> 0$, with $u=ct^{-\alpha}$,
if we denote the annealed law of the scaled endpoint displacement $\frac{X_t}{t^{1-\alpha}}$ on $[0,\infty)$ by $\tilde{\Prob}^{\mathbb{Q}}_{u,t}$, (i.e. for any $t>0$ and Borel subset $A\subset [0,\infty)$, we let
$\tilde{\Prob}^{\mathbb{Q}}_{u,t}(\frac{X_t}{t^{1-\alpha}}\in A)=\Expe \mathbb{E}^{\mathbb{Q}^{u,t}}_{0}[\1_{A}({X_t}/{t^{1-\alpha}})]$,)
then the sequence of measures $\{\tilde{\Prob}^{\mathbb{Q}}_{u,t}\}_{t> 0}$ is tight.


\end{corollary}



{We further prove the optimal upper bounds for the polymer endpoint displacement, as well as the fluctuation of the height function, with  $u=ct^{-\alpha}$ for any $c\in\R$ and $\alpha>1/3$.} 

\begin{theorem}\label{thm:scpolym}
{Assume that} $u=ct^{-\alpha}$ with any $\alpha>1/3$ and $c \in \R$. For any $t\geq 1$, we have
    \begin{equation}\label{eq:scpolym}
     0 \leq 
    \Expe \mathbb{E}^{\mathbb{Q}^{u,t}}_{0}[X_t]\leq Ct^{2/3},
    \end{equation}
    for some $C>0$ depending only on $c$. 

As a result, for any $t\geq 1$, we  have    \begin{equation}\label{eq:scfe}
    \Var [H_u(t,0)] \leq Ct^{2/3},
    \end{equation}
    for some $C>0$ depending only on $c$. 
\end{theorem}

The tightness of the annealed polymer endpoint measure is the following.
\begin{corollary}\label{co:tightsc}
For any $\alpha>1/3$ and $c \in \R$, with $u=ct^{-\alpha}$, if we denote the annealed law of the scaled endpoint displacement $\frac{X_t}{t^{2/3}}$ on $[0,\infty)$ by $\bar{\Prob}^{\mathbb{Q}}_{u,t}$, (i.e. for any $t>0$ and Borel subset $A\subset [0,\infty)$, 
$\bar{\Prob}^{\mathbb{Q}}_{u,t}\left(\frac{X_t}{t^{2/3}}\in A\right)=\Expe \mathbb{E}^{\mathbb{Q}^{u,t}}_{0}[\1_{A}({X_t}/{t^{2/3}})]$,)
then the sequence of measures $\{\bar{\Prob}^{\mathbb{Q}}_{u,t}\}_{t\geq 1}$ is tight.
\end{corollary}
%

\subsection{Context}
The exploration of a growing interface near an attractive wall arises naturally from the studies of wetting and entropic repulsion phenomena.  In his pioneering work \cite{kardar1985depinning}, Kardar introduced the half-space KPZ models to describe such interface growth and he predicted a ``depinning'' phase transition as the wall attraction weakens.  The underlying concept is as follows: when the wall exerts a strong attractive force, which is also referred to as in the subcritical regime or the bound phase, Kardar conjectured that the half-space directed random polymer is ``pinned'' to the wall due to the force; as this attractive force decreases below a certain threshold, the polymer becomes entropically repelled away from the wall. It is further expected that when ``pinned'', the polymer endpoint has $O(1)$ transversal fluctuation in a window around the wall and the free energy of the polymer fluctuates with an exponent $1/2$. After being ``unpinned'', the polymer exhibits KPZ behaviors: the free energy fluctuates with exponent $1/3$ and the endpoints have transversal fluctuation exponent $2/3$.

The mathematical studies on the KPZ models in half-space started with a series of works on half-space last passage percolation (LPP) models \cite{johansson2000shape, baik2001algebraic, baik01symmetrized, baik01asymptotics}. 
For studies on the half-space LPP models starting from stationary initial conditions, we refer to \cite{betea2020stationary, betea2022half,ferrari2024time}. For models with positive temperatures, significant breakthroughs were only made in recent studies.

The phase transition predicted in \cite{kardar1985depinning} in a positive temperature model was proved through studying the free energy statistics. In \cite{barraquand2023identity}, the authors used a combinatorial identity to connect the free energy of the point-to-line half-space log-gamma (HSLG) polymer and a point-to-point full-space log-gamma polymer model. For the point-to-point HSLG polymer free energy, the phase transition was demonstrated by \cite{imamura2022solvable} using a novel correspondence between the HSLG polymer model and free boundary Schur measures. They also established the asymptotic behavior and the corresponding phase transition for the height function of the half-space KPZ equation at the boundary, starting from the droplet initial condition, by taking the weak noise scaling limit of the point-to-point HSLG polymer free energy \cite{wu2020intermediate, barraquand2023stationary}.

Beyond free energy statistics, progress has also been made in studying the ``depinning'' phase transition through the transversal fluctuations of polymer measures. In two consecutive works—\cite{barraquand2023kpz} for the unbound phase and \cite{das2023half} for the bound phase—the authors derived this transition for the transversal fluctuations in the HSLG polymer model. Their arguments rely on combinatorial identities and the solvability of the HSLG model, utilizing inputs from the HSLG line ensemble \cite{barraquand2023kpz}. Additionally, as applications of the HSLG line ensemble, \cite{das2023half} and \cite{das2024convergence} demonstrated convergence to stationary measures for the point-to-point HSLG polymer in the bound and unbound phases, respectively.

Other related works are \cite{ginsburg2023pinning,  he2022shift, he2024boundary,zhang2024tasep}. In \cite{ginsburg2023pinning}, the author examines discrete point-to-point half-space directed polymer models where the attractive force from the wall is strong enough to increase the free energy macroscopically.  \cite{ginsburg2023pinning} demonstrated that this condition is sufficient to pin the polymer to the wall (resulting in order-one transversal polymer fluctuations) and for the free energy to exhibit diffusive fluctuations and follow an asymptotically Gaussian distribution (with a fluctuation exponent of 1/2). \cite{he2022shift, he2024boundary} study the half-space ASEP and six vertex model and obtain the corresponding phase transitions.
In \cite{zhang2024tasep}, the author solves the half-space TASEP with a
general deterministic initial condition and obtains the  transition probability for the half-space
KPZ fixed point.

 Our work investigates the half-space KPZ equation \eqref{eq:hsKPZ} starting from the stationary Brownian initial condition \eqref{eq:i.c.}, which differs from the setups discussed above. Through a (formal) Feynman-Kac representation, the results are interpreted in terms of the point-to-stationary-measure half-space continuum directed random polymer (CDRP). Previous studies in the physics literature on the half-space KPZ equation with Brownian initial conditions include \cite{barraquand2020halfkpz, barraquand2022half}. Other related works include \cite{de2020delta, krajenbrink2020replica, barraquand2021kardar}. {
 It should be noted that drifted Brownian motions as in \eqref{eq:i.c.} are not the only stationary measures for the half-space KPZ equation; nevertheless, they are conjectured to be the only Gaussian ones \cite[Conjecture 1.5]{barraquand2023stationary}. For other stationary measures of the  half-space KPZ equation, we refer to \cite[Theorem~1.4]{barraquand2023stationary}.
 }
 
 We discuss our approaches and the main contributions in the following sections.
 {In our main results, we rescale the boundary parameter $u$ in windows around the phase transition point $0$. In the following, terminology-wise, we call the regime $u=ct^{-\alpha}$ with $\alpha\in [0,\frac{1}{3})$ and $c<0$ the \textbf{subcritical regime}; we call  $u=ct^{-1/3}$ with $c \neq 0$ the \textbf{critical regime}; we call $u=ct^{-\alpha}$ with $\alpha>1/3$ and $c \in \R$ the \textbf{extended critical regime}. }
 
\subsubsection{Variance identity}
 In Theorem~\ref{thm:varianceID}, we establish a relation between the variance of the height function and the expectation of the polymer endpoint displacement, which serves as the foundation of our analysis. Similar identities have previously appeared in several full-space models at stationarity, such as {log-gamma polymer \cite{timoAOP12},} geometric LPP \cite{timoEJP}, the O'Connell-Yor polymer \cite{seppalainen2010bounds}, the asymmetric simple exclusion process \cite{balazs2010order}, interacting diffusion \cite{landon}, the KPZ fixed point \cite{pimentel2022integration}, and the KPZ equation \cite{gu2023another}, among others. Our derivation of the variance identity is based on Gaussian integration by parts, inspired by \cite{pimentel2022integration, gu2023another}. Such variance identities play a pivotal role in establishing fluctuation exponents: given a family of stationary measures parameterized by a certain variable, one can proceed with a coupling argument involving a small perturbation of this parameter, and by comparing the two processes corresponding to different parameters and leveraging the quadratic form of the average free energy, it is possible to derive the correct fluctuation exponent. This approach has been developed and extensively explored by Sepp\"al\"ainen and collaborators, leading to breakthroughs \cite{balazs2010order,balazs2011fluctuation} and numerous further applications.

With the variance identity \eqref{eq:varidkpzhs} and the quadratic form of the averaging free energy given in \eqref{e.quafree}, it seems plausible to apply the coupling argument for the half-space KPZ equation and prove the fluctuation exponents for all values of $u\in\R$. However, an immediate challenge arises: the family of \emph{Brownian} stationary measures is parameterized by the Neumann boundary condition of the equation, meaning that perturbing the initial data while maintaining stationarity necessitates perturbing the equation itself.
This creates extra technical difficulties for the half-space model compared to the full-space setting{, and we use different comparison arguments to cover various ranges of $u\in\R$.}


\subsubsection{Transversal fluctuation}
With the variance identity, the analysis of height function fluctuations reduces to studying the transversal fluctuations of the polymer endpoint. For the point-to-stationary-measure half-space continuum directed random polymer (CDRP) $\mathbb{Q}_0^{u,t}$, our result \eqref{eq:intmidupbd} provides the conjectured optimal upper bound for the endpoint displacement in the bound phase $u<0$. Specifically, for a polymer path starting at $(t,0)$ and running backward in time,  when  $t$ is sufficiently large, the endpoint reaches equilibrium in the subcritical regime. The term $\psi'(2|u|)$ on the right-hand side of \eqref{eq:intmidupbd} represents the   endpoint position $\Expe \mathbb{E}^{\mathbb{Q}^{u,t}}_{0}[X_t]$ for $t\gg1$. This can also be interpreted as the localization length of the model, which was conjectured by \cite{kardar1985depinning} to diverge quadratically and is confirmed in our analysis (see \eqref{eq:psiasym}).  In essence, the result follows from the explicit stationary measure for the half-space KPZ equation constructed in \cite{barraquand2023stationary}, the Dufresne identity \cite{dufresne1990distribution}, and a stochastic dominance argument, where we compare the polymer path starting at the origin with the one starting from equilibrium.

The explicit calculation for the stationary endpoint position yields $\Expe \mathbb{E}^{\mathbb{Q}^{u,t}}_{0}[X_t]\approx \psi'(2|u|) \sim |u|^{-2}$ in the bound phase, which allows us to extend the analysis beyond the subcritical regime and cover the critical regime. {Using the symmetry result in Proposition~\ref{pr:symmetry}, as well as another comparison argument, we obtain the conjectured optimal upper bounds $t^{2/3}$ for the transversal fluctuations {in the critical and extended critical regimes} $u\sim t^{-\alpha}$ with $\alpha\geq 1/3$, in Corollary~\ref{co:tightnessleft}, Corollary~\ref{co:upbdright}, and Theorem~\ref{thm:scpolym}. }



\subsubsection{Height function statistics}
The variance estimates of the height function follow directly from the variance identity and the bounds obtained for the polymer endpoint.


In \cite{barraquand2020halfkpz}, the authors conjectured that for any fixed $u \neq 0$, $ H_u(t,0)$ would have a large time behavior as
$$
H_u(t,0) \simeq \left(\frac{1}{2}u^2-\frac{1}{24}\right)t+t^{1/2} \chi,
$$
where $\chi$ is an $O(1)$ Gaussian random variable. For $u$ in a scaling window around $0$,  the authors later stated in \cite{barraquand2022half} that for  $u=ct^{-1/3}$ with $c \in \R$, one should expect
$$
\lim _{t \rightarrow \infty} \Prob\left(\frac{H_{ct^{-1 / 3}}(0, t)+\frac{t}{24}}{t^{1 / 3}} \leqslant s\right)=F_c(s),
$$
where $F_c(\cdot)$ admits an explicit expression computed through the Fredholm determinant. 

Our results confirm the predictions regarding the mean and fluctuation exponent of the height function  $H_u(t,0)$ in both the subcritical and critical regimes. Theorem~\ref{thm:var} further illustrates how the boundary parameter $u$ influences the fluctuation scale for any $u \neq 0$ at any finite time $t\geq 0$, which was not immediately evident in earlier predictions. As expected, the method we employ does not yield the exact asymptotic fluctuation distributions.

{For the extended critical regime $u\sim t^{-\alpha}$ with $\alpha>1/3$ (including the case of $u=0$ which formally corresponds to $\alpha=\infty$),  the variance identity combined with the optimal upper bound for the endpoint displacement in \eqref{eq:scpolym} leads to an optimal upper bound of order $t^{2/3}$ for the variance of the height function in \eqref{eq:scfe}. 
Our method does not seem to give the lower bound for the variance of the height function in the {extended critical regime}, which is conjectured to be of order $t^{2/3}$ as well.}
\begin{remark}\label{re:conjconv}
The upper bounds in \eqref{eq:intmidupbd} and \eqref{eq:varbdexact} are expected to be sharp as $t\to\infty$ for any fixed $u< 0$.
This expectation stems from a conjecture, yet to be proven, that when $u<0$, the half-space KPZ equation starting from the droplet initial condition would converge weakly (modulo height shift) to the stationary measure -- a Brownian motion with drift $u$ as $t \to \infty$. The upper bound in \eqref{eq:intmidupbd} corresponds to the annealed mean of endpoint displacement under this conjectured limit. For further discussions of this conjecture, see \cite{le2022steady, barraquand2023stationary}. {We also note that it may be possible to show this result using the recently constructed half-space KPZ line ensembles \cite{das2025half}.}

{However, as shown by \eqref{eq:scpolym} and \eqref{eq:scfe}, the upper bounds in \eqref{eq:intmidupbd} and \eqref{eq:varbdexact} are no longer sharp in the {extended critical regime}.}
\end{remark}


\subsubsection{Organization of the paper}
In Section~\ref{se:prelim}, we introduce some preliminary results and auxiliary lemmas for the half-space SHE. As most results are analogous to full-space SHE, we defer their proofs to Appendix~\ref{ap:auxlemproofs}. In Section~\ref{se:mean}, we compute the average growth rate and prove Theorem~\ref{thm:mean}.  In Section~\ref{se:varID}, we establish the variance identity in Theorem~\ref{thm:varianceID}.  Section~\ref{se:uppbd} contains the proof of the upper bound for the annealed mean of polymer endpoint displacement, leading to Theorem~\ref{thm:var} and Corollary~\ref{co:tightnessleft}. In Section~\ref{se:symmetry}, we provide the proofs of Proposition~\ref{pr:symmetry} and Corollary~\ref{co:upbdright}. {Finally, in Section~\ref{se:sc}, we prove Theorem~\ref{thm:scpolym} and Corollary~\ref{co:tightsc}.}
\subsection{Notations}
We use the following notations and conventions throughout this paper.
\begin{enumerate}[(i)]
\item For two topological spaces $E$ and $F$, we use $\mathcal{C}(E,F)$ to denote the space of continuous functions from $E$ to $F$. We use $\mathcal{C}^\infty$ to denote smooth functions, $\mathcal{C}_b$ to denote the space of continuous and bounded functions, and $\mathcal{C}_c$ to denote the space of continuous and compactly supported functions. 
\item We use   $\mathbf{E}_{\eta}$ as the expectation on any specific noise $\eta$ defined on the probability space $(\Omega,\filt,\Prob)$. The notation $\Expe$ is for the total expectation on $(\Omega,\filt,\Prob)$.  The expectation on Brownian motions independent of $(\Omega,\filt,\Prob)$ will be denoted by $\mathbb{E}_B$. The expectation on the quenched polymer measures (for each fixed $\omega\in \Omega$) is denoted by $\mathbb{E}^{\mathbb{Q}^{\cdot,\cdot}}_\cdot$. We use $\|\cdot\|_p$ to denote the norm of $L^p(\Omega, \filt, \Prob)$ for any  $p >0$.
\item We use $C(\cdot,\cdot,...,\cdot)$ to denote any constant $C$ that depends only on the parameters inside the parentheses. These constants may vary from line to line.
\end{enumerate}

\subsection*{Acknowledgement}
The work was partially supported by the NSF under grant DMS:2203014. We are grateful to Shalin Parekh and Yaozhong Hu for discussions on some proofs of Section~\ref{se:prelim}. We thank  Guillaume Barraquand for helping us understand the symmetry in Proposition~\ref{pr:symmetry} and for providing extensive feedback.
{We also thank Christopher Janjigian, Firas Rassoul-Agha, and Timo Sepp\"al\"ainen for explaining to us the proof of polymer path continuity and on discussions related to \cite{alberts2022green}.}

\section{Preliminaries}\label{se:prelim}

{In this section, we review some basics of the half-space SHE with Robin boundary conditions, including definition, and moment estimates, among others. We also introduce the half-space CDRP. Most of the results presented here are fairly standard, so readers already familiar with the topic may choose to skip this section.}
\subsection{Robin heat kernel}\label{se:mildform}
For $\mu, s, t \in \R, t>s$ and $x, y\geq 0$, the Robin heat kernel on half-line with parameter $\mu$, denoted by $ \kernel_{\mu}^R(t,x\viv s,y) $, is the unique solution to the deterministic heat equation,
\begin{equation}\label{eq:RobHeat}
\begin{aligned}
\partial_t \kernel_{\mu}^R(t,x\viv s,y) &= \frac{1}{2} \partial^2_x \kernel_{\mu}^R(t,x\viv s,y) \quad \text{for }(t,x) \in (s,\infty)\times [0, \infty),\\
\left.\partial_x \kernel_{\mu}^R(t,x\viv s,y)\right|_{x=0}  &= \mu \kernel_{\mu}^R(t,0\viv s,y),\\
\lim_{t\to s} \kernel_{\mu}^R(t,x\viv s,y) &= \delta_y(x) \quad \text{in a weak sense on } L^2([0, \infty)).
\end{aligned}
\end{equation}

An explicit expression of $ \kernel_{\mu}^R(t,x\viv s,y) $ for any $t>s$ is given by 
 \begin{equation}\label{eq:kernelRobin}
\kernel_{\mu}^R(t,x\viv s,y)=p_{t-s}(x-y)+p_{t-s}(x+y)-2 \mu \int_0^{\infty} p_{t-s}(x+y+z) e^{-\mu z} \dd z, 
\end{equation}
where $p_t(x)=\frac{1}{\sqrt{2 \pi t}} e^{-x^2 /(2 t)}$ is the standard heat  kernel. 
A proof can be found in \cite[Section 4]{corwin2018open}. 

The heat kernel $\kernel_{\mu}^R(t,x\viv s,y)$ satisfies the semigroup property and is monotonically decreasing in $\mu$ (see \eqref{eq:probrep2} below). To obtain an upper bound on $\kernel_{\mu}^R$, first,   when $x, y,z \geq 0$, $(x-y)^2 \leq (x+y)^2$, thus $p_{t-s}(x+y) \leq p_{t-s}(x-y)$.
Also, since $(x+y)^2 +z^2\leq (x+y+z)^2$, we have
\[
\int_0^{\infty} p_{t-s}(x+y+z) e^{-\mu z} \dd z \leq   p_{t-s}(x+y)\int_0^{\infty}e^{-\frac{z^2}{2(t-s)}-\mu z} \dd z.
\]
It follows that for any $a,b,\mu_0 \in\R$, there exists a constant $C(a, b, \mu_0)$ such that for any $\mu \geq \mu_0, a\leq s<t\leq b$ and $x,y \in [0, \infty)$, we have
\begin{equation}\label{eq:elem1}
\kernel_{\mu}^R(t,x\viv s,y)  \leq C(a,b, \mu_0)p_{t-s}(x-y) \leq C(a,b, \mu_0)(t-s)^{-1/2}.
\end{equation}

Another useful result is the probabilistic representation of the Robin heat kernel, as it is related to the Feynman-Kac formula of the half-space SHE. We present the representation below, and one can find a proof via It\^o-Tanaka formula   in \cite[Section 2.5]{freidlin1985functional}.

Let $B_t$ be a Brownian motion starting at  $B_0=x\in [0, \infty)$. The process $|B_t|$, defined by the absolute value of $B_t$, is known as a {reflected Brownian motion} (RBM). We use notations $\mathbb{P}_B^x$ and $\mathbb{E}_B^x$ to denote the probability and expectation with respect to $B_t$ only, emphasizing the initial point $x$. For any $\mu \in \R$,  $f \in \mathcal{C}_b (\R, \R)$, 
\begin{equation}\label{eq:probrep1}
\int_0^{\infty} \kernel_{\mu}^R(t,x\viv s,y) f(y) \dd y = \mathbb{E}^x_B[f(|B_{t-s}|)\exp(-\mu L^{0}_{t-s})],
\end{equation}
where \begin{equation*}\label{eq:bmlocaltime}
L^{0}_{t} := \lim _{\varepsilon \rightarrow 0} \frac{1}{2 \varepsilon} \int_0^t \1_{[-\varepsilon, \varepsilon]}(B_s) \dd\langle B\rangle_s=\lim_{\eps \to 0}\frac{1}{2\eps}\text{Leb}(s: |B_s|\leq  \eps, 0\leq s \leq t) 
\end{equation*}
is the {Brownian local time} of $B_t$ at the origin with Leb representing the Lebesgue measure.

(To be more precise, \cite[Section 2.5]{freidlin1985functional} says that
\begin{equation}\label{eq:probrep3}
\int_0^{\infty} \kernel_{\mu}^R(t,x\viv s,y) f(y) \dd y = \mathbb{E}^x_{\Upsilon}[f(\Upsilon_{t-s})\exp(-\mu L^{0,\Upsilon}_{t-s})],
\end{equation}
where $\Upsilon_\cdot$ is the reflected Brownian motion $\Upsilon_t := |B_t|$,  $\mathbb{P}^x_{\Upsilon}$ and $\mathbb{E}^x_{\Upsilon}$ are the probability and expectation with respect to $\Upsilon$,
and $L_t^{0,\Upsilon}$ is the local time of $\Upsilon$  at zero. It is further known by \cite{itodiffusion} that
\begin{equation*}\label{eq:truelocaltime}
    L_t^{0,\Upsilon}=\lim _{\eps \to 0} \frac{1}{2 \varepsilon} \int_0^t \1_{[0, \varepsilon]}\left(\Upsilon_s\right) \dd s.
\end{equation*}
In our case, \eqref{eq:probrep3} is equivalent to \eqref{eq:probrep1}. We will use both of these two representations.)

Consequently, for any $x,y \in [0, \infty)$,
\begin{equation}\label{eq:probrep2}
\kernel_{\mu}^R(t,x\viv s,y) =  \kernel^{N}(t,x\viv s,y) \mathbb{E}_B^x\left[\exp(-\mu L^{0}_{t-s})\mid |B_{t-s}|=y\right] 
\end{equation}
where $\kernel^{N}(t,x\viv s,y) := \kernel_{0}^R(t,x\viv s,y)$ is the transition density of the reflected Brownian motion $|B_t|$. 

Using \eqref{eq:probrep2} and Lemma~\ref{le:bbmax} below, one can check that for any $\mu \in \R$, there exist two positive constants $C_1(\mu,t-s),C_2(\mu,t-s)$ such that 
\begin{equation}\label{eq:RobinComparison}
    C_1(\mu,t-s) \kernel^{N}(t,x\viv s,y) \leq \kernel_{\mu}^R(t,x\viv s,y)  \leq C_2(\mu,t-s) \kernel^{N}(t,x\viv s,y).
\end{equation}
The constants $C_1,C_2$ can be chosen uniformly for $(\mu, t-s)$ in any compact subsets of $\R\times(0,\infty)$.

\subsection{Mild formulation and Feynman-Kac formula} 
For any $\mu \in \R$, we use $Z_\mu(t,x \viv s, \zeta)$ to denote the mild solution to  the half-space SHE  starting from an arbitrary time $s \in \R$ and from a broad class of positive Borel measure $\zeta$:
\begin{equation}\label{eq:hsSHEv2}
\begin{aligned}
\partial_t Z_\mu(t,x \viv s, \zeta) & =\frac{1}{2} \partial_x^2 Z_\mu(t,x \viv s, \zeta)+Z_\mu(t,x \viv s, \zeta)\xi(t, x) \quad \text{for } (t,x)\in (s,\infty)\times [0, \infty),\\ \partial_x Z_\mu(t,x \viv s, \zeta)\big|_{x=0} & =\mu Z_\mu(t,0 \viv s, \zeta), \\
Z_\mu(s, \cdot\viv s, \zeta)& =\zeta(\cdot).
\end{aligned}\end{equation}
We discuss this generalized half-space SHE because some estimates below are easier to obtain when the initial condition is deterministic and either localized (e.g., $\delta$-type initial conditions) or spatially homogeneous (e.g., constant initial conditions). 
To avoid confusion, we use the notation $Z_\mu(t,x \viv s, \zeta)$ to specify the initial time and data. We only use notation $Z_\mu(t,x)$ for the solution of \eqref{eq:hsSHE} when the initial time is $s=0$ and the initial data is fixed as $\zeta(\dd x) = \exp(W(x)+(\mu+\frac{1}{2})x) \dd x$.

We first define the mild solution to \eqref{eq:hsSHEv2} with an arbitrary starting time $s \in \R$ and a (potentially random) Borel measure initial condition.  For any $s<t$, let $(\filt^{\xi, \zeta}_{r})_{r\in[s,t]}$ be the natural filtration  generated by $\zeta$ and $\xi$ on the time interval $[s,t]$.
\begin{definition}\label{de:SHEsol}
Let $\mu, s \in \R$, $\zeta$ be some (potentially random) Borel measure that is independent of the white noise $\xi$ and $\Prob$-almost surely supported on $[0, \infty)$. 
We say that a measurable process $Z_\mu(\cdot,\cdot \viv s,\zeta)$ on $(s,\infty)\times [0,\infty)$ is the mild solution to \eqref{eq:hsSHEv2} with initial condition $ \zeta(\cdot)$, if for all $t\in (s,\infty)$, 
$Z_\mu(t,\cdot\viv s, \zeta)$ is $(\filt^{\xi, \zeta}_{r})_{r\in[s,t]}$-adapted and   satisfies
\begin{equation}\label{eq:mildform}
\begin{aligned}
&Z_\mu(t,x \viv s, \zeta) = \int_0^{\infty} \kernel_{\mu}^R(t,x\viv s,y) \zeta(\dd y)+ \int_s^t\int_0^{\infty}  \kernel_{\mu}^R(t,x\viv r,y) Z_\mu(r,y\viv s, \zeta)\xi(\dd y\dd r), \;\; \forall x\in [0,\infty),
\end{aligned}
\end{equation}
where the stochastic integral is interpreted in the It\^o-Walsh sense.
\end{definition}

The two hypotheses below regarding the initial condition $\zeta $ ensure the existence and uniqueness (under different assumptions) of the mild solution $Z_\mu(\cdot, \cdot\viv s, \zeta)$. All the initial conditions we deal with are included in either case. Throughout the paper, we assume that $\zeta$ is not the zero measure and does not depend on $s$.

\begin{Hypothesis}\label{hy:randomic}
There exists a random variable $f$ taking values in $ \mathcal{C}([0, \infty),(0,\infty))$ and a constant $a>0$ so that
\[\zeta(\dd x) = f(x) \dd x \text{ and } \sup _{x \in [0, \infty)} e^{-a x} \Expe\left[f(x)^2\right]<\infty.\]
\end{Hypothesis}
\begin{Hypothesis}\label{hy:detic}
 $\zeta$ is deterministic, does not depend on  $\mu$, and for any $s<\tau$, 
\[\sup_{s < t \leq \tau}\sup_{x \in [0, \infty)}
\sqrt{t-s}\int_0^{\infty}  \kernel_{\mu}^R(t,x\viv s,y) \zeta(\dd y) <\infty.
\]
\end{Hypothesis}
For any $\mu \in \R$, $\zeta(\dd x)={Z}_\mu(0, x) \dd x =\exp(W(x)+(\mu+\frac{1}{2})x) \dd x$ satisfies Hypothesis~\ref{hy:randomic}. The Dirac-$\delta$ initial condition $\zeta(\cdot)=\delta_y(\cdot)=\delta(\cdot -y)$ at any point $y \in [0,\infty)$ satisfy Hypothesis~\ref{hy:detic} (see \eqref{eq:elem1}). The constant initial condition $\zeta(\dd x) = \dd x$ satisfies both.
By \eqref{eq:RobinComparison}, whenever $\zeta$ satisfies Hypothesis~\ref{hy:detic} for some $\mu \in \R$, it satisfies it for all $\mu \in \R$.
\begin{proposition}\label{pr:hsSHEprop}
(1) 
When the initial data $\zeta$ satisfies Hypothesis~\ref{hy:randomic}, there exists a unique mild solution $Z_\mu(\cdot, \cdot\viv s, \zeta) \in \mathcal{C}([s, \infty)\times [0,\infty), \mathbb{R})$ as defined in Definition~\ref{de:SHEsol} that satisfies, for any terminal time $ \tau >s$,
\begin{equation}\label{eq:uniquenesshyp1}
\sup_{s \leq t \leq \tau}\sup_{x \in [0, \infty)}e^{-ax} \Expe \left[{Z}_{\mu}(t, x\viv s, \zeta)^2\right]<\infty .\end{equation}

(2) 
When the initial data $\zeta$ satisfies Hypothesis~\ref{hy:detic}, there exists a unique mild solution $Z_\mu(\cdot, \cdot\viv s, \zeta) \in \mathcal{C}((s, \infty)\times[0,\infty), \mathbb{R})$ as defined in Definition~\ref{de:SHEsol} that satisfies, for any terminal time $ \tau >s$,
\begin{equation}\label{eq:uniquenesshyp}
\sup_{s < t \leq \tau}\sup_{x \in [0, \infty)}\int_s^t  \int_s^r  \int_0^{\infty}\int_0^{\infty}  \kernel_{\mu}^R(t,x\viv r,y) ^2  \kernel_{\mu}^R(r,y\viv w,z)^2 \mathbf{E}\left[Z_\mu(w, z\viv s, \zeta)^2\right] \dd z \dd y \dd w \dd r<\infty.
\end{equation}
The mild solution admits an explicit chaos series representation as
\begin{equation}\label{eq:chaos1}
\begin{aligned}
Z_\mu(t,x\viv s, \zeta) & =\int_0^{\infty} \kernel_{\mu}^R(t,x\viv s,y) \zeta(\dd y)\\
& +\sum_{k=1}^{\infty} \int_{[0,\infty)^{k+1}}  \int_{\R^k }\zeta(\dd y)\kernel_{\mu}^R\left(r_{1: k}, w_{1: k} | s, y ; t, x\right) \xi(\dd r_1\dd w_1)\cdots \xi(\dd r_k\dd w_k),
\end{aligned}
\end{equation}
where, with convention $r_0=s, r_{k+1}=t, w_0=y, w_{k+1}=x$, we define 
\[\kernel_{\mu}^R\left(r_{1: k}, w_{1: k} | s, y ; t, x\right) = \prod_{j=0}^{k}\kernel_{\mu}^R(r_{j+1},w_{j+1} \viv r_{j},w_{j})\1_{(0,\infty)}({r_{j+1}- r_j}).
\]
\end{proposition}
\begin{proof} \cite[Proposition 4.2]{parekh2019kpz} has proved part (1). To prove part (2), one can follow the same argument as in \cite{bertini1995stochastic}, with the use of the iterative arguments as in the proof of Lemma~\ref{le:posMom} below. By iterative arguments, one can also show that the right-hand side of \eqref{eq:chaos1} is well-defined, $(\filt^{\xi, \zeta}_{r})_{r\in[s,t]}$-adapted, and $\Prob$-almost surely satisfies \eqref{eq:mildform} and \eqref{eq:uniquenesshyp}. Then the uniqueness of the mild solution implies \eqref{eq:chaos1}.
\end{proof}
\begin{remark}
A more general solution theory for the half-space SHE should be able to relax the assumptions on non-random initial conditions. Similar generalizations have been explored for the full-space SHE in \cite{chen2014holder, chen2015moments}. If such relaxations were made, we would only need our initial condition in \eqref{eq:hsSHE} to meet the required hypotheses for $\Prob$-almost surely each realization of $W$. As the current hypotheses are sufficient for our purposes, we do not explore these extensions further.
\end{remark}

We next give an estimate on positive moments for $Z_\mu(t,x\viv s, \zeta)$ with $\zeta$ under Hypothesis~\ref{hy:detic} and $s,t$ in any arbitrary interval $[a,b]\subset \R$. Since the proof is  through a rather standard iteration scheme, we defer it to Appendix~\ref{ap:auxlemproofsposMom}.

\begin{lemma}\label{le:posMom}
Let $a<b, \mu_0 \in \R$ and $\zeta(\cdot)$ satisfy Hypothesis~\ref{hy:detic}. For any   $p \in[1,\infty)$, there exists a constant $C=C(a, b, \mu_0, p,\zeta)$ so that
\begin{equation}\label{eq:positivemoments}
\|Z_\mu(t,x\viv s, \zeta)\|_p \leq C (t-s)^{-1/4}\left(\int_0^{\infty} \kernel_{\mu}^R(t,x\viv s,y)\zeta( \dd y)\right)^{1/2},
\end{equation}
for any $\mu \in  [\mu_0, \infty), a\leq s<t\leq b$ and $x \in [0, \infty)$.
\end{lemma}

It  also helps to study the solution ${Z}_{\mu}(t,x\viv s,\zeta)$ through the Feynman-Kac formula. By the probabilistic representation of parabolic PDEs with Robin boundary condition (see e.g., \cite[Section 2.5]{freidlin1985functional}), it is natural to expect that for any $\mu \in \R$ and deterministic $f \in \mathcal{C}_b([0,\infty),\R)$, the half-space stochastic heat equation \eqref{eq:hsSHEv2} started from $\zeta(\dd x)=f(x)\dd x$ (under Hypothesis~\ref{hy:detic}) formally admits the Feynman–Kac type representation
\begin{equation}\label{eq:formalFK}
\begin{aligned}
&{Z}_{\mu}(t,x\viv s,\zeta) = \mathbb{E}^{x}_{B}\left[f(|B_{t-s}|)\exp(-\mu L^{0}_{t-s}): \exp :\left\{\int_0^{t-s} \xi\left(t-r, |B_r|\right) d r\right\}\right]\\
& = \int_0^{\infty}\kernel^{N}(t,x\viv s,y) \mathbb{E}^{x}_{B}\left[\exp(-\mu L^{0}_{t-s}): \exp :\left\{\int_0^{t-s} \xi\left(t-r, |B_r|\right) d r\right\} \bigg\mid |B_{t-s}|=y\right] f(y)\dd y.
\end{aligned}
\end{equation}
Here $: \exp :$ denotes the Wick-ordered exponential and other notations are the same as in \eqref{eq:probrep1} and \eqref{eq:probrep2}. 
This representation is only formal because the integral of the white noise over a Brownian path makes no sense. Meanwhile, analogous to the classical result \cite{bertini1995stochastic} on full-space, we can rigorously approximate ${Z}_{\mu}(t,x\viv s,\zeta)$ by using the Feynman-Kac  formulas with a smoothed noise. We defer this detailed approximation to Appendix \ref{ap:fk}.

A useful result that can be proved using the Feynman-Kac formula is the  negative moments bound of the solutions.

\begin{lemma}[Negative moments]\label{le:negmom}
Assume that the initial data $\zeta$ is either the constant $\zeta(\dd x)=\dd x$ or the Dirac-$\delta$ initial data $\zeta=\delta_y$ for some $y \in [0,\infty)$.
Let \[
z_\mu(t,x\viv 0,\zeta):= \int_0^{\infty}\kernel_\mu^R(t,x\viv0,y)\zeta(\dd y).
\]
Then for any $\mu \in \R, p\in[0,\infty), \tau>0$, there exists a constant $C = C(\tau,\mu,p)$ such that for any $x\in[0,\infty)$, $t\in (0,\tau]$, and $\zeta$ being constant or Dirac-$\delta$,
\begin{equation}\label{eq:negativemoments}
\Expe\left[Z_\mu(t,x\viv0,\zeta)^{-p}\right] \leq Cz_\mu(t,x\viv 0,\zeta)^{-p}.
\end{equation}
The constant $C$ can be chosen uniformly for $\mu$ in any compact subset of $\R$.
 \end{lemma}
The bound of negative moments seems to be a well-established result, proven for SHEs in various contexts \cite{mueller1991support,muellernualart,flores2014strict,hu2022asymptotics}. Since no proof has been written for the half-space setting, we prove Lemma~\ref{le:negmom} in Appendix~\ref{ap:negmom}. 

\subsection{Green's function}
%
For any fixed $\mu, s \in \R$, when $\zeta(\cdot)=\delta_y(\cdot)=\delta(\cdot -y)$ for some $y \in [0,\infty)$, the unique mild solution $Z_\mu(\cdot,\cdot\viv s, \zeta)$ as described in Proposition~\ref{pr:hsSHEprop} part~(2) is   known as the Green's function, or the fundamental solution, to the equation \eqref{eq:hsSHEv2}. This is a crucial object for the half-space SHE model which can be used  to construct the half-space continuum directed random polymer (CDRP) measure.

In fact, using the chaos expansion \eqref{eq:chaos1}, we can define a five-parameter random field in our common probability space $(\Omega, \filt, \Prob)$ that supports the space-time white noise $\xi$. Let \[\mathcal{D}= \{(\mu, s,y,t,x)\viv \mu,s,t \in \R \text{ with } s<t \text{ and }x,y \in [0,\infty) \} \subset \mathbb{R}^5.\] For $(\mu, s,y,t,x)\in \mathcal{D}$, define (with the same notation and convention as in \eqref{eq:chaos1})
\begin{equation}\label{eq:greensdef}
\begin{aligned}
\mathcal{Z}_\mu(t, x \viv s, y)&:= \kernel_{\mu}^R(t,x\viv s,y) \\&+\sum_{k=1}^{\infty} \int_{[0,\infty)^{k}}  \int_{\R^k }\kernel_{\mu}^R\left(r_{1: k}, w_{1: k} | s, y ; t, x\right) \xi(\dd r_1\dd w_1)\cdots \xi(\dd r_k\dd w_k).
\end{aligned}
\end{equation}
It   follows that, for any fixed $s, \mu \in \R$, $y \in [0,\infty)$, the field $\mathcal{Z}_\mu(\cdot, \cdot \viv s, y)$ is 
the unique mild solution to \eqref{eq:hsSHEv2} satisfying \eqref{eq:mildform} and \eqref{eq:uniquenesshyp} with initial condition $\zeta = \delta_y$. 
\begin{corollary}\label{co:greensPosMom}
Let $a<b, \mu_0 \in \R $ and $p\in[1,\infty)$. For any $\mu \in [\mu_0, \infty)$, $a \leq s < t \leq b$, $x,y \in [0, \infty)$, 
\begin{equation}\label{eq:greensPosMom}
\|\mathcal{Z}_\mu(t,x\viv s,y)\|_p \lesssim (t-s)^{-1/4} \kernel_{\mu}^R(t,x\viv s,y)^{1/2}\lesssim (t-s)^{-1/2}\exp(-(x-y)^2/4(t-s)),
\end{equation}
where $\alpha \lesssim \beta$ denotes $\alpha \leq C \beta$ with some constant $C=C(a, b, \mu_0, p)$.
\end{corollary}
\begin{proof}
Substituting $\zeta$ with $\delta_y$ in Lemma~\ref{le:posMom} and using \eqref{eq:elem1}. Note that {in the statement of Lemma~\ref{le:posMom}, the prefactor $C$ may depend on the initial data $\zeta$, but one can follow the same proof and check that,} with the Dirac-$\delta$ initial conditions, we can relax the bounds in the proof so that the constant $C$ does not depend on $y\in [0,\infty)$. 
\end{proof}

We record the following properties for the field $\mathcal{Z}$.
\begin{proposition}\label{pr:greens}
There exists a modification of the field $\mathcal{Z}_{\mu}(t,x\viv s,y)$ that is jointly continuous in all five variables $(\mu, s,y,t,x)\in \mathcal{D}$, and has the following properties:
\begin{enumerate}[(i).]
\item\label{it:mean} For any $(\mu, s,y,t,x)\in \mathcal{D}$, 
\begin{equation}\label{eq:greenmean}
\Expe\left[\mathcal{Z}_{\mu}(t,x\viv s,y)\right] =  \kernel_{\mu}^R(t,x\viv s,y).\end{equation}
\item\label{it:convolution}
For any fixed $s, \mu \in \R$ and $\zeta(\dd x)$ satisfying Hypothesis~\ref{hy:randomic} (or Hypothesis~\ref{hy:detic}), the random process defined by the convolution formula \begin{equation*}
(t,x)\mapsto \int_0^{\infty} \mathcal{Z}_\mu(t,x\viv s,y) \zeta(\dd y) 
\end{equation*}
is the unique mild solution $Z_\mu(\cdot, \cdot \viv s, \zeta) $ in Proposition~\ref{pr:hsSHEprop} satisfying \eqref{eq:mildform} and \eqref{eq:uniquenesshyp1} (or \eqref{eq:uniquenesshyp}).

\item\label{it:strictpos} (Positivity) For any $(\mu, s,y,t,x)\in \mathcal{D}$, we have $\mathcal{Z}_\mu(t,x\viv s,y)>0$. 
\item\label{it:tstat} (Time stationarity) For any fixed $\mu, t_0 \in \R, x,y \in [0,\infty)$, \[\{\mathcal{Z}_\mu(t+t_0,x\viv s+t_0,y)\}_{s, t \in \R, s<t} \stackrel{\text { law }}{=}\{\mathcal{Z}_\mu(t,x\viv s,y)\}_{s, t \in \R, s<t}.\]
\item\label{it:tindep} For any finite disjoint time intervals $\left\{\left(s_i, t_i\right]\right\}_{i=1}^n$ and any $x_i, y_i \in [0,\infty)$, the random variables $\left\{\mathcal{Z_\mu}\left(t_i, x_i \viv s_i, y_i\right)\right\}_{i=1}^n$ are mutually independent.
\item\label{it:cki} (Chapman-Kolmogorov identity) For any fixed $\mu \in \R$, there exists an event $\Omega_0$ with $\Prob(\Omega_0)=1$ so that for any $\omega \in \Omega_0$, $s,t \in \R \text{ with } s<t, x,y \in [0,\infty)$, and $r \in (s,t)$,
\begin{equation}\label{eq:compositionlaw}
 \mathcal{Z}_\mu(t,x\viv s,y) = \int_0^{\infty}  \mathcal{Z}_\mu(t,x\viv r,w) \mathcal{Z}_\mu(r,w\viv s,y) \dd w.
\end{equation}
\item\label{it:trev} For any fixed $\mu, s,t \in \R, t>s$, 
\begin{equation}\label{eq:trev}
\{\mathcal{Z}_\mu(t,x\viv s,y)\}_{x, y \in [0, \infty)} \stackrel{\text { law }}{=}\{\mathcal{Z}_\mu(t,y\viv s,x)\}_{x, y \in [0, \infty)}.\end{equation}
\end{enumerate}
\end{proposition}
\begin{proof}
The joint continuity of the field and properties \eqref{it:mean}--\eqref{it:cki} closely mirror those of the Green's function field associated with the full-space stochastic heat equations in \cite{alberts2014continuum}. There are two primary differences from the full-space Green's function. First is that we have an additional parameter $\mu$ from the boundary condition in the half-space setting. Since the Robin heat kernel defined in \eqref{eq:RobHeat} is continuous in $\mu$, generalizing the joint continuity property to include this parameter $\mu$ by Kolmogorov continuity
criteria is straightforward.
The second difference is that the field no longer has stationarity in space, but this fact does not impact the properties above. Due to the presence of the boundary, for $x_0\neq 0, x,y \geq -x_0$, we do not expect that $\mathcal{Z}_{\mu}(t,x\viv s,y)$ and $\mathcal{Z}_{\mu}(t,x+x_0\viv s,y+x_0)$ are equal in law.

{For Chapman-Kolmogorov identity \eqref{it:cki},  a detailed proof for the full-space Green's functions can be found in \cite[Lemma 3.12]{alberts2022green}. One can adapt this proof to our half-space setting, except that we do not have a quenched  tail bound analogous to \cite[Corollary 3.10]{alberts2022green}. Alternatively, one can use the Kolmogorov continuity criteria applied to the five variables $(s,t,x,y,r)$ on the right-hand-side of \eqref{eq:compositionlaw} to conclude the proof.}

The proof of property~\eqref{it:trev} is similar to \cite[Appendix B]{dunlap2023fluctuation} and follows from an approximation of Feynman-Kac formulas. We defer it to Appendix~\ref{ap:reversal}.
\end{proof}

For any fixed $\mu\in \R$, we use $Z_{\mu}(\cdot,\cdot)$ to denote the unique mild solution to \eqref{eq:hsSHE} in Proposition~\ref{pr:hsSHEprop}~part~(1). By Proposition~\ref{pr:greens}~\eqref{it:convolution}, when $\mu=u-\frac{1}{2}$, for any $t>0$, $x\in[0,\infty)$,
\begin{equation}\label{eq:sheGreensform}
Z_{u-\frac{1}{2}}(t,x)=
\int_0^{\infty} \mathcal{Z}_{u-\frac{1}{2}}(t,x\viv 0,y)\exp\left(W(y)+uy\right)\dd y.
\end{equation}

We finish this subsection by providing some $L^p(\Omega)$-norm bounds to $Z_{u-\frac{1}{2}}(t,x)$ and $H_u(t,x)$. These bounds follow from Lemma~\ref{le:fk} and Lemma~\ref{le:negmom} and will be very useful throughout the paper. A detailed proof is given in Appendix~\ref{ap:momentsproof}.
\begin{proposition}\label{pr:mombds0}
For any fixed $u \in \R, t\geq 0, p\in[1,\infty)$, there exists a constant $C=C(u,t,p)$ so that the following holds for all $x\in[0,\infty)$:
\begin{equation}\label{eq:posCx}
\|Z_{u-\frac{1}{2}}(t,x)\|_p \leq C\exp(Cx);
\end{equation}
\begin{equation}\label{eq:negCx}
\|Z_{u-\frac{1}{2}}(t,x)^{-1}\|_p \leq C\exp(Cx).
\end{equation}
{It follows that for any fixed $u\in \R$, $\Prob$-almost surely,
\begin{equation}\label{eq:sheposfinite}
0<Z_{u-\frac{1}{2}}(t,x)<\infty , \quad \forall t\geq 0, x\geq 0.\end{equation}
The Hopf-Cole solution to \eqref{eq:hsKPZ} is thus well-defined:}
\begin{equation}\label{eq:kpzGreensform}
H_u(t, x) = \log Z_{u-\frac{1}{2}}(t, x) =  \log \int_0^{\infty} \mathcal{Z}_{u-\frac{1}{2}}(t,x\viv 0,y)\exp(W(y)+uy)\dd y, \end{equation}
and satisfies
\begin{equation}\label{eq:HbdCx} 
\quad \|H_u(t,x)\|_p \leq C(u,t,p)\exp(C(u,t,p)x),
\end{equation}
for any $p \in [1,\infty), t\geq 0, x \geq 0$. All the constants above can be chosen uniformly over any compact subset of $(u,t)\in\R\times {[0,\infty)}$. 
\end{proposition}

\subsection{Half-space CDRP}\label{se:cdrp}
We can now use the Green's function $\mathcal{Z}_{\mu}(t,x\viv s,y)$ to construct the continuum directed random polymer (CDRP) measures in half-space. We then introduce the induced (quenched) endpoint measures on $([0,\infty),\mathcal{B}([0,\infty))) $ from these polymers.

For simplicity, we always set the initial time $s=0$ and fix some $\mu\in\R$. To align with the previous conventions, we use the parameter $u=\mu+\frac{1}{2}$ to index these measures.

For any $t>0$, we equip the path space $\mathcal{C}_{[0,t]}:=\mathcal{C}([0,t],[0,\infty))$ with 
the Borel $\sigma$-algebra $\mathcal{B}(\mathcal{C}_{[0,t]})$. Let $X=(X_r)_{0 \leq r \leq t}$ be the canonical variable on $\mathcal{C}_{[0,t]}$.
A probability measure on this space is uniquely determined by its finite dimensional distributions. For our purposes, we are interested in the continuum directed polymer $X$ on $\mathcal{C}([0,t],[0,\infty))$ that starts from a fixed point, runs backward in time in the environment of $\xi$, and has the terminal condition $\exp \left(W_{u}(\cdot)\right)$ at $r=t$.

For each fixed $u \in \R, t>0$ and $x\in[0,\infty)$, we define the following (quenched) point-to-measure polymer:

\begin{definition}\label{de:cdrp}
For any $\omega\in \Omega$ such that \eqref{eq:sheposfinite} holds, define the (quenched, point-to-measure) half-space continuum directed random polymer (CDRP) measure $\mathbb{Q}^{u,t}_{x}$ as the probability measure on $\mathcal{C}_{[0,t]}$ with finite dimensional marginals given by
\begin{equation}\label{eq:polydef}
\begin{aligned}
&\mathbb{Q}^{u,t}_{x}(X_{t_1}\in \dd x_1, \dots, X_{t_k} \in \dd x_k) \\ &\quad\quad=  \frac{\int_0^{\infty}\prod_{j=0}^k\mathcal{Z}_{u-\frac{1}{2}}(t-t_{j},x_{j}\viv t- t_{j+1},x_{j+1})\exp(W_{u}(x_{k+1})) \dd x_{k+1}}{\int_0^{\infty}\mathcal{Z}_{u-\frac{1}{2}}(t,x\viv 0,y)\exp(W_{u}(y))\dd y}\dd x_1\dots \dd x_k,
\end{aligned}
\end{equation}
for $0=t_0 < t_1<\cdots<t_k < t_{k+1}=t$ and with $x_0=x$. 
\end{definition}
We first address the well-definedness of the half-space CDRP $\mathbb{Q}^{u,t}_{x}$:

By the Chapman-Kolmogorov equation \eqref{eq:compositionlaw}, the family of finite dimensional distributions in \eqref{eq:polydef} is consistent. Thus by the Kolmogorov  consistency theorem, for $\Prob$-almost every realization of $(\xi, W)$, \eqref{eq:polydef} gives a unique probability measure on $([0,\infty)^{[0,t]}, \mathcal{B}([0,\infty))^{[0, t]})$. It remains to show that the measure $\mathbb{Q}^{u,t}_{x}$ is $\Prob$-almost surely 
supported on paths in $\mathcal{C}_{[0,t]}$. 

{To do so, consider a (deterministic) probability measure  $\mathbb{P}^{RB}_{u,t,x}$ on $([0,\infty)^{[0,t]}, \mathcal{B}([0,\infty))^{[0, t]})$ with finite-dimensional marginals }
\begin{align*}&\mathbb{P}^{RB}_{u,t,x}(X_{t_1}\in \dd x_1, \dots, X_{t_k} \in \dd x_k)\\&=\frac{\Expe\left[\mathbb{Q}^{u,t}_{x}(X_{t_1}\in \dd x_1, \dots, X_{t_k} \in \dd x_k)\int_0^{\infty}\mathcal{Z}_{u-\frac{1}{2}}(t,x\viv 0,y)\exp(W_{u}(y))\dd y \right]}{\Expe [\int_0^{\infty}\mathcal{Z}_{u-\frac{1}{2}}(t,x\viv 0,y)\exp(W_{u}(y))\dd y]}. \end{align*} 
By \eqref{eq:greenmean},
$\mathbb{P}^{RB}_{u,t,x}$ is a probability measure on reflected Brownian paths reweighted correspondingly by a local time term as well as the terminal boundary condition. {(To see this, one can compute the expectation with respect to $\xi$ and $W$ explicitly, separately for the numerator and the denominator.)}  Let $Y:[0,\infty)^{[0,t]}\to \R$ be an indicator function such that $Y=1$ if the paths are uniformly continuous when restricted to rational numbers and $Y=0$ if not. By the regularity of the  Brownian paths, $\mathbb{E}^{RB}_{u,t,x}[Y]=1$, which suggests that $\mathbb{Q}^{u,t}_{x}(Y=1)=1$ $\Prob$-almost surely. It thus follows that the polymer path $X_\cdot$ has a continuous modification $\Prob$-almost surely.

We define $\mathbb{Q}^{u,0}_{x}:=\delta_x$ as a measure on $([0,\infty),\mathcal{B}([0,\infty))) $.
We use $\mathbb{E}^{\mathbb{Q}^{u,t}}_{x}$ to denote the expectation under $\mathbb{Q}^{u,t}_{x}$.  
We also define the following.
\begin{definition}\label{def:cdrpendpt}
   Fix $u \in \R$ and $\omega\in \Omega$ such that \eqref{eq:sheposfinite} holds. For any $t>0$, $x\in [0,\infty)$,
   we define 
   the (quenched) density $\rho^{R}_{u}(\cdot \viv t,x)$ on $[0,\infty)$ as 
\begin{equation}\label{eq:hsQuench0} \rho^{R}_{u}(y \viv t,x) := \frac{ \mathcal{Z}_{u-\frac{1}{2}}(t,x\viv 0,y)\exp (W_{u}(y))}{\int_0^{\infty}\mathcal{Z}_{u-\frac{1}{2}}(t,x\viv 0,y')\exp(W_{u}(y')) \dd y'}, \quad \forall y \in [0,\infty).\end{equation}
\end{definition}
It is clear that $\rho^{R}_{u}(\cdot \viv t,x) $ is well-defined for all $t>0, x\in[0,\infty)$ simultaneously on a   probability one event.
It follows that for any fixed $t>0, x\in[0,\infty)$, $\rho^{R}_{u}(\cdot \viv t,x) $ is the (quenched) density of the polymer endpoint on $[0,\infty)$ under $\mathbb{Q}^{u,t}_{x}$.
\begin{proposition}
    For any fixed $u \in \R, t>0, x\in [0,\infty)$, 
    $\Prob$-almost surely, we have 
\[ \mathbb{Q}^{u,t}_{x}(X_t \in A) = \int_A \rho^{R}_{u}(y \viv t,x) \dd y,\quad\quad \text{for all } A\in \mathcal{B}([0,\infty)),\]
and the quenched expectation of the polymer endpoint is
\[
\mathbb{E}^{\mathbb{Q}^{u,t}}_{x}[X_t] = \int_0^\infty y\rho^{R}_{u}(y \viv t,x) \dd y.
\]
\end{proposition}
\begin{proof}
    By the path continuity of $X_\cdot$, the continuity of $Z_{u-\frac{1}{2}}(\cdot,\cdot)$ on $\R_{\geq 0}^2$, and \eqref{eq:polydef}, for any $\varphi \in \mathcal{C}_c([0,\infty),\R)$,
\[
\begin{aligned}
\mathbb{E}^{\mathbb{Q}^{u,t}}_{x}[\varphi(X_t)] &= \lim_{r\to t}\mathbb{E}^{\mathbb{Q}^{u,t}}_{x}[\varphi(X_r)] =  \lim_{r\to t} \frac{1}{{Z}_{u-\frac{1}{2}}(t,x)}\int_0^{\infty} \varphi(z){\mathcal{Z}_{u-\frac{1}{2}}(t,x\viv t-r,z){Z}_{u-\frac{1}{2}}(t-r,z) \dd z}{ }\\
&= \frac{1}{{Z}_{u-\frac{1}{2}}(t,x)}\int_0^{\infty} \varphi(y){\mathcal{Z}_{u-\frac{1}{2}}(t,x\viv 0,y)\exp (W_{u}(y)) \dd y},
\end{aligned}
\]where the interchange of the limit and the integral is justified by the dominated convergence theorem. 
By Urysohn's lemma and the monotone convergence theorem, this extends to any function $\varphi=\1_O$ with $O$ being an open set of $[0,\infty)$. By Dynkin's $\pi-\lambda$ theorem, we can further extend to $\1_A$ with $A\in \mathcal{B}([0,\infty))$.
\end{proof}

\begin{remark}\label{re:whynocoupling}
We define the half-space CDRP $\Prob$-almost surely for each fixed $t>0, x\geq 0$, but do not address its coupling over all  $t>0, x\geq 0$. Additionally, we do not prove some basic properties (Markov, Feller, etc.) for the half-space CDRP, analogous to those proved in the recent comprehensive work \cite{alberts2022green} for the full-space CDRP. One reason for this is that an adaption of \cite{alberts2022green} to half-space setting {still needs} some  technical estimates {similar to those obtained in \cite[Appendix C]{alberts2022green}} on the Robin heat kernels{, yet we were not able to obtain them}.
On the other hand, these properties are not required for our purposes, as we will use an alternative method to establish the stochastic monotonicity needed in Section~\ref{se:uppbd}. 

On a heuristic level, the proof   in Section~\ref{se:uppbd} is based  on  sampling the initial points of the CDRP $\mathbb{Q}^{u,t}_{x}$ from another density on $[0,\infty)$ to construct a measure-to-measure polymer  $\mathbb{Q}^{u,t}_{\tilde{W}}$ for any fixed $\omega \in \Omega$. Without a simultaneous coupling, we do not justify that this construction defines a new measure on $\mathcal{C}[0,t]$ $\Prob$-almost surely.
However, as we will show in Section~\ref{se:uppbd}, we can still apply this idea.
 In particular, we only need the quenched endpoint distribution for this polymer, which is well-defined through SHE mild solutions $\Prob$-almost surely. 
All of our 
$\Prob$-almost sure results are derived from the properties and continuity of mild solutions, making the coupling issue irrelevant for our purposes.

{We do not attempt to couple the CDRP measures for all $u\in\R$ either, since it is unnecessary for our purposes.
}

\end{remark}

\begin{remark}
   By the formal Feynman-Kac representation \eqref{eq:formalFK}, it is natural to interpret the above continuum directed polymer measure as a ``Gibbs measure'' of the reflected Brownian motion in half-space with start-point $x \in [0,\infty)$. The Radon–Nikodym derivative with respect to the law of $\Upsilon_\cdot=|B_\cdot|$ should be expressed as
\[\exp(W(\Upsilon_t)+u\Upsilon_t)  \exp \left(-(u-{1}/{2}) L_{t}^{0,\Upsilon}\right): \exp :\left\{\int_0^{t} \xi\left(t-r,\Upsilon_r\right) d r\right\}.
\]
The above expression is only formal when $\xi$ is the spacetime white noise. In fact, we suspect that the half-space continuum directed polymer measure should be singular with respect to the measure of reflected Brownian motion in half-space, just as the full-space polymer measure studied in \cite{alberts2014continuum}. 
\end{remark}

Finally, we remark that with the definition of $\mathbb{Q}^{u,t}_x$, the solution $Z_{u-\frac{1}{2}}(t,x)$ to the half-space stochastic heat equation \eqref{eq:hsSHE} is the partition function of the half-space CDRP $\mathbb{Q}^{u,t}_x$. Thus, its logarithm, which is the solution $H_u(t,x)$ to the KPZ equation \eqref{eq:hsKPZ}, is the free energy of the polymer.

\section{Average growth rate: proof of Theorem~\ref{thm:mean}}\label{se:mean}

{In this section, we study the average growth speed of $H_u$ and prove Theorem~\ref{thm:mean}. Together with Proposition~\ref{pr:symmetry}, it proves the conjecture of the average growth speed in \cite[(4.18)]{barraquand2020halfkpz}.

The method here is stochastic analysis, essentially an application of It\^o's formula. The only explicit calculation we do relies on a nice conditional Laplace transform formula obtained for exponential functionals of a Brownian motion with a drift \cite{baudoin2011exponential}. Similar calculations in the periodic setting can be found in \cite[Section 2.2.3]{gu2024some}, based on Yor's formula \cite{yor1992some}. An alternative proof for Theorem~\ref{thm:mean}  can be pursued by using the half-space log-gamma polymer and making the claims that led to \cite[(4.18)]{barraquand2020halfkpz} rigorous. Notably, the average growth rate has also been recently established for the open KPZ equation \cite{bar2024}; it would be interesting to investigate if the method presented here can be adapted to that context. This section can be read independently of the rest of the paper.}

{It is well-known that, under appropriate assumptions, the partition function of the point-to-line directed polymer can be related to a positive martingale. Thus, to study the average of its logarithm, one can perform a semimartingale decomposition and only needs to analyze the drift term. This trick has been used extensively in the study of disordered systems, see e.g. \cite{comets1995sherrington} for a spin glass model and \cite[Chapter 5]{comets2017directed} for directed polymers in random environments.}


We begin by presenting the following preliminary result, originally proved by \cite{dufresne1990distribution}. A more detailed discussion on the study of exponential functionals of Brownian motions can be found in the review \cite{matsumoto2005exponential}.
\begin{lemma}
For any $u<0$, the perpetual exponential functional of Brownian motion with drift satisfies
\begin{equation}\label{eq:expoBM} \int_0^{\infty}  \exp(W_{u}(x)) \dd x  \stackrel{\text { law }}{=} \frac{2}{\gamma(-2u)},
\end{equation}
where $\gamma(-2u)$ is a gamma random variable with shape parameter $-2u$.
\end{lemma}

{Now for any $u<0,t\geq 0$, define 
\[
\mathcal{Y}_{u }(t) :=  \int_0^{\infty} Z_{u-\frac{1}{2}}(t,x)\dd x,
\]
which can be viewed as the $\zeta$-to-line partition function of the underlying directed polymer, with the initial data  $\zeta(\dd x)=\exp(W_u(x))\dd x$.} 
$\mathcal{Y}_u(0)$ is the random variable appearing on the left-hand side of \eqref{eq:expoBM}. We know that $0<\mathcal{Y}_u(0)<\infty$ $\Prob$-almost surely.

As described in \eqref{eq:station}, the 
 Brownian motion with drift $u$ is stationary to the half-space KPZ equation with boundary parameter $u$, which implies that for each $t\geq0$,
\begin{equation}\label{eq:shestatdist}
\left\{\frac{{Z}_{u-\frac{1}{2}}(t,y)}{{Z}_{u-\frac{1}{2}}(t,0)}\right\}_{y\in[0,\infty)} \stackrel{\text{law}}{=}\{ \exp \left({W}_{u}(y)\right)\}_{y\in[0,\infty)}, 
\end{equation}
Thus, the random variable $\mathcal{Y}_u(t)/{Z}_{u-\frac{1}{2}}(t,0)$ satisfies
\begin{equation}\label{eq:yoverz}
\frac{\mathcal{Y}_u(t)}{{Z}_{u-\frac{1}{2}}(t,0)} = \int_0^{\infty} Z_{u-\frac{1}{2}}(t,x)[{Z}_{u-\frac{1}{2}}(t,0)]^{-1}\dd x\stackrel{\text{law}}{=}\int_0^{\infty}  \exp(W_{u}(x)) \dd x = \mathcal{Y}_u(0).
\end{equation}
Since $\mathcal{Y}_u(0)$ and ${Z}_{u-\frac{1}{2}}(t,0)$ are both $\Prob$-almost surely finite and positive, $\mathcal{Y}_u(t)$ is also $\Prob$-almost surely finite and positive.

The proof of Theorem~\ref{thm:mean} for $u<0$ relies on the stationarity \eqref{eq:shestatdist} as well as the identity in law \eqref{eq:expoBM}. For $u\geq 0$, $ \exp(W_{u}(y)) \dd y $ is not in $L^1([0,\infty))$. This is the key reason why our proof cannot be extended to $u\geq 0$.

{Now we state a decomposition of $\mathcal{Y}_u(t)$ which follows from the mild formulation of the SHE. Due to the Robin boundary condition, there exists an extra drift term compared to the full space case.} 

\begin{lemma}\label{le:semimarY} Let $\filt^{\xi, W}_{t}$ be the natural filtration generated by $W$ and $\xi$ on the time interval $[0,t]$.
For any $u<0$, $(\mathcal{Y}_u(t), \filt^{\xi, W}_{t})$ is a continuous semimartingale, and  $\mathcal{Y}_u(t)$ admits a decomposition
\begin{equation}\label{eq:semidecom}
\mathcal{Y}_u(t) = \mathcal{Y}_u(0) + M_u(t) + A_u(t), 
\end{equation}
where $M_u(t)$ is a continuous local martingale with \[
M_u(t)= \int_0^t\int_0^{\infty} {Z}_{u-\frac{1}{2}} (s,z)\xi(\dd z\dd s),
\]
and $A_u(t)$ is a process of finite-variation 
\[A_u(t) = - \frac{1}{2} \left(u-\frac{1}{2}\right)\int_0^t  {Z}_{u-\frac{1}{2}} (v,0)\dd v.\]
\end{lemma}
\begin{proof}
For any $t>0$, by the convolution formula \eqref{eq:sheGreensform}, we can use the Green's function $\mathcal{Z}$ to write
\begin{equation}\label{eq:4stofubi}
\begin{aligned}
&\mathcal{Y}_u(t) = \int_0^{\infty} Z_{u-\frac{1}{2}}(t,x)\dd x =  \int_0^{\infty} 
\int_0^{\infty} \mathcal{Z}_{u-\frac{1}{2}} (t,x\viv 0,y) \exp (W_{u}(y))\dd y \dd x\\
&=  \int_0^{\infty} 
\int_0^{\infty} \mathcal{Z}_{u-\frac{1}{2}} (t,x\viv 0,y) \dd x \exp (W_{u}(y))\dd y,
\end{aligned}
\end{equation}
where the last equality follows from Tonelli's theorem. 
By \eqref{eq:mildform}, for any $t>0, y\geq 0$,
\begin{equation}\label{eq:greenintoverx}
\begin{aligned}
& \int_0^{\infty} \mathcal{Z}_{u-\frac{1}{2}} (t,x\viv 0,y)\dd x =  \int_0^{\infty} 
 \kernel_{u-\frac{1}{2}}^R(t,x\viv 0,y) \dd x\\&+ \int_0^{\infty} \int_0^t\int_0^{\infty} \kernel_{u-\frac{1}{2}}^R(t,x\viv s,z)  \mathcal{Z}_{u-\frac{1}{2}} (s,z\viv 0,y)\xi(\dd z\dd s)\dd x.
\end{aligned}
\end{equation}
By \eqref{eq:greensPosMom}, 
\begin{equation}\label{eq:stocfubcon}
\int_0^{\infty} \left(\int_0^t \int_0^{\infty} \kernel_{u-\frac{1}{2}}^R(t,x\viv s,z)^2 {\mathbf{E}_\xi}\left[ 
\mathcal{Z}_{u-\frac{1}{2}} (s,z\viv 0,y)^2\right] \dd z\dd s\right)^{1/2} \dd x<\infty,
\end{equation}
so we can use the stochastic Fubini theorem (\cite[Theorem 4.33]{da2014stochastic}) to switch the order of integration in the last term of \eqref{eq:greenintoverx} and integrate $x \in [0,\infty)$ first.

Applying the fundamental theorem of calculus to the integral $\int_0^{\infty} \kernel_{u-\frac{1}{2}}^R(t,x\viv s,y) \dd x$ and using the definition of the Robin heat kernel in \eqref{eq:RobHeat}, for any fixed $y \geq 0, s\in [0,t]$, we have
\[\begin{aligned}&
\int_0^{\infty}  \kernel_{u-\frac{1}{2}}^R(t,x\viv s,z) \dd x = 1+ \int_s^{t} \int_0^{\infty}\partial_v  \kernel_{u-\frac{1}{2}}^R(v,x\viv s,z)\dd x \dd v  \\&=1+ \int_s^{t} \int_0^{\infty}\frac{1}{2}\partial^2_x\kernel_{u-\frac{1}{2}}^R(v,x\viv s,z)\dd x \dd v =1+ \frac{1}{2} \int_s^{t} \left[\left.-\partial_x\kernel_{u-\frac{1}{2}}^R(v,x\viv s,z)\right|_{x=0}\right] \dd v\\&=1 -\frac{1}{2} \int_s^t \left(u-\frac{1}{2}\right) \kernel_{u-\frac{1}{2}}^R(v,0\viv s,z)\dd v.\end{aligned}
\]
Thus \eqref{eq:greenintoverx} equals to
\[\begin{aligned}
&1 - \frac{1}{2} \left(u-\frac{1}{2}\right)\int_0^t \kernel_{u-\frac{1}{2}}^R(v,0\viv 0,y)\dd v + \int_0^t\int_0^{\infty} \mathcal{Z}_{u-\frac{1}{2}} (s,z\viv 0,y)\xi(\dd z\dd s)\\&
-\frac{1}{2} \left(u-\frac{1}{2}\right) \int_0^t\int_0^{\infty}\left[\int_s^t\kernel_{u-\frac{1}{2}}^R(v,0\viv s,z)\dd v\right] \mathcal{Z}_{u-\frac{1}{2}} (s,z\viv 0,y)\xi(\dd z\dd s).
\end{aligned} \]
Again by \eqref{eq:greensPosMom}, 
\[
\int_0^t \left( \int_0^{\infty} \int_0^v \kernel_{u-\frac{1}{2}}^R(v,0\viv s,z)^2 {\mathbf{E}_\xi}\left[ 
\mathcal{Z}_{u-\frac{1}{2}} (s,z\viv 0,y)^2 \right] \dd s\dd z\right)^{1/2} \dd v<\infty,
\]
so we can use the stochastic Fubini theorem to switch the order of integrations as \[
\begin{aligned}
& \int_0^t\int_0^{\infty}\left[\int_s^t\kernel_{u-\frac{1}{2}}^R(v,0\viv s,z)\dd v\right] \mathcal{Z}_{u-\frac{1}{2}} (s,z\viv 0,y)\xi(\dd z\dd s) \\&= \int_0^t \int_0^{\infty} \int_0^v \kernel_{u-\frac{1}{2}}^R(v,0\viv s,z)
\mathcal{Z}_{u-\frac{1}{2}} (s,z\viv 0,y)\xi(\dd z\dd s) \dd v\\& = \int_0^t  \mathcal{Z}_{u-\frac{1}{2}} (v,0\viv 0,y)\dd v -\int_0^t \kernel_{u-\frac{1}{2}}^R(v,0\viv 0,y)\dd v.\end{aligned}\]
Therefore, for each $y\in[0,\infty)$, \eqref{eq:greenintoverx} equals to
\[ \int_0^{\infty} \mathcal{Z}_{u-\frac{1}{2}} (t,x\viv 0,y)\dd x =1 + \int_0^t\int_0^{\infty} \mathcal{Z}_{u-\frac{1}{2}} (s,z\viv 0,y)\xi(\dd z\dd s)-\frac{1}{2} \left(u-\frac{1}{2}\right) \int_0^t  \mathcal{Z}_{u-\frac{1}{2}} (v,0\viv 0,y)\dd v. \]
Using again the stochastic Fubini theorem and Tonelli's theorem, we have
\[\begin{aligned}
\mathcal{Y}_u(t) &= \int_0^{\infty} \left[1 + \int_0^t\int_0^{\infty} \mathcal{Z}_{u-\frac{1}{2}} (s,z\viv 0,y)\xi(\dd z\dd s)-\frac{1}{2} \left(u-\frac{1}{2}\right) \int_0^t  \mathcal{Z}_{u-\frac{1}{2}} (v,0\viv 0,y)\dd v\right] e^{W_u(y)}\dd y\\
&= \mathcal{Y}_u(0) +\int_0^t\int_0^{\infty} {Z}_{u-\frac{1}{2}} (s,z)\xi(\dd z\dd s) - \frac{1}{2} \left(u-\frac{1}{2}\right)\int_0^t  {Z}_{u-\frac{1}{2}} (v,0)\dd v,
\end{aligned}
\]
which is the decomposition \eqref{eq:semidecom}.
Note that for $\Prob$-almost surely all realization of $W$, by \eqref{eq:greensPosMom},
\begin{equation}\label{eq:qvbd}\begin{aligned}&\mathbf{E}_\xi\int_0^t \int^{\infty}_0 Z_{u-\frac{1}{2}} (s,y)^2 \dd y\dd s\\
&=\int_0^{\infty}  \int_0^{\infty} \int_0^t \int^{\infty}_0 \mathbf{E}_\xi \left[ \mathcal{Z}_{u-\frac{1}{2}} (s,y\viv 0,z_1) \mathcal{Z}_{u-\frac{1}{2}} (s,y\viv 0,z_2)\right]\dd y\dd s e^{W_u(z_1)}e^{W_u(z_2)} \dd z_1 \dd z_2\\
&\leq \int_0^{\infty}  \int_0^{\infty} \int_0^t \int^{\infty}_0\|\mathcal{Z}_{u-\frac{1}{2}} (s,y\viv 0,z_1)\|_2\| \mathcal{Z}_{u-\frac{1}{2}} (s,y\viv 0,z_2)\|_2 \dd y\dd s e^{W_u(z_1)}e^{W_u(z_2)} \dd z_1 \dd z_2\\
&\leq C(t,u) \left[\int_0^{\infty} e^{W_u(z)} \dd z \right]^2,
\end{aligned}
\end{equation}
so for any $t>0$,
\[\int_0^t \int^{\infty}_0 Z_{u-\frac{1}{2}} (s,y)^2 \dd y \dd s<\infty, \quad \Prob\text{-almost surely},
\]
and $M_u(t)$ is a continuous local martingale with quadratic variation \[
\langle M_u \rangle_t=\int_0^t \int^{\infty}_0 Z_{u-\frac{1}{2}} (s,y)^2 \dd y\dd s.\] 
The proof is complete.
\end{proof}
\begin{remark}
One needs to be cautious with the $L^p$ norms after putting the Brownian motion $W$ into the filtration $\filt_t^{\xi,W}$. In fact, by \eqref{eq:expoBM}, $\mathcal{Y}_u(0)=\int_0^{\infty} e^{W_u(x)} \dd x \in L^p(\Omega)$ if and only if $-2u>p$. Consequentially, taking an expectation $\mathbf{E}_W$ on the last line of \eqref{eq:qvbd} only gives a bounded term when $u<-1$, so
we do not claim that $\Expe\langle M_u\rangle_t$ is bounded for all $u<0$. 

On the other hand, if we consider each fixed realization of the initial data,  we can think of $M_u(t)$ as a local martingale with respect to $\filt_t^{\xi}$ -- the filtration generated by $\xi$ only. The bound in \eqref{eq:qvbd} shows that $M_u(t)$ is a square-integrable martingale w.r.t. $\filt_t^{\xi}$ for $\Prob$-almost surely all realization of $W$. 

A similar issue also occurs when we apply the stochastic Fubini theorem in the above proof. In order to overcome it, we used \eqref{eq:4stofubi} to separate the randomness from $\xi$ and from $W$. The stochastic Fubini theorem is applied to processes that are adapted to $\filt_t^{\xi}$, as shown in \eqref{eq:stocfubcon}.  \end{remark}

{The following lemma provides the semimartingale decomposition for $\log \mathcal{Y}_u(t)$, using which we compute $\Expe \log \mathcal{Y}_u(t)$ which further leads to $\Expe H_u(t,0)$. The crucial part is that the drift term will be written as an additive functional of the process $\{Z_{u-\frac12}(s,\cdot)/\mathcal{Y}_u(s)\}_{s\geq0}$, which is at stationarity by \eqref{eq:station}.}
\begin{lemma}\label{co:semimardec}
For any $u<0$ and $t\geq 0$, we have
\begin{equation}\label{eq:meandecom}
\Expe[H_u(t,0)] = \Expe\left[ - \frac{1}{2} \left(u-\frac{1}{2}\right) \frac{1}{\mathcal{Y}_u(0)}-  \frac{1}{2}\frac{\int_{0}^{\infty}\exp(2W_u(y)) \dd y }{\mathcal{Y}_u(0)^{2}} \right]t.
\end{equation}
\end{lemma}
\begin{proof}
Recall that as defined in \eqref{eq:hsKPZ}, $H_u(t,0)=\log Z_{u-\frac{1}{2}}(t,0)$. Taking the logarithm and then the expectation on both sides of \eqref{eq:yoverz}, we have
\[
\Expe[H_u(t,0)] =  \Expe[\log \mathcal{Y}_u(t) - \log \mathcal{Y}_u(0)].
\]
By Lemma~\ref{le:semimarY} and It\^o's formula \cite[Theorem 4.32]{da2014stochastic}, $\log \mathcal{Y}_u(t)$ can be decomposed as
\[
\begin{aligned}
&\log \mathcal{Y}_u(t) - \log \mathcal{Y}_u(0) = \int_0^t \mathcal{Y}_u(s)^{-1} \dd M_u(s) +  \int_0^t \mathcal{Y}_u(s)^{-1} \dd A_u(s) -\frac{1}{2} \int_0^t \mathcal{Y}_u(s)^{-2} \dd \langle M_u\rangle_s\\
&=\int_0^t \int_0^{\infty} \frac{Z_{u-\frac{1}{2}}(s,y)}{\mathcal{Y}_u(s)} \xi(\dd y\dd s) - \frac{1}{2} \left(u-\frac{1}{2}\right)\int_0^t \frac{Z_{u-\frac{1}{2}}(s,0)}{\mathcal{Y}_u(s)} \dd s - \frac{1}{2}  \int_0^t \int_0^{\infty} \frac{Z_{u-\frac{1}{2}}(s,y)^2}{\mathcal{Y}_u(s)^{2} } \dd y \dd s.
\end{aligned}
\]
By \eqref{eq:shestatdist}, for any $u<0$ and $s\geq 0$, we have
\[\Expe \frac{Z_{u-\frac{1}{2}}(s,0)}{\mathcal{Y}_u(s)} = \Expe \left[{\int^{\infty}_0 Z_{u-\frac{1}{2}} (s,y) Z_{u-\frac{1}{2}} (s,0)^{-1}\dd y}\right]^{-1} =  \Expe \left[\int^{\infty}_0 \exp(W_u(y))\dd y\right]^{-1} = \Expe\frac{1}{\mathcal{Y}_u(0)},
\]
and
\[
\begin{aligned}&\Expe \int^{\infty}_0 \frac{Z_{u-\frac{1}{2}} (s,y)^2}{\mathcal{Y}_u(s)^2} \dd y = \Expe \frac{ \int^{\infty}_0Z_{u-\frac{1}{2}} (s,y)^2 \dd y }{\left[ \int^{\infty}_0 Z_{u-\frac{1}{2}} (s,y') \dd y'\right]^2 }=  \Expe \frac{ \int^{\infty}_0Z_{u-\frac{1}{2}} (s,y)^2Z_{u-\frac{1}{2}} (s,0)^{-2} \dd y }{\left[ \int^{\infty}_0 Z_{u-\frac{1}{2}} (s,y')Z_{u-\frac{1}{2}} (s,0)^{-1} \dd y'\right]^2 }  \\&= \Expe \frac{ \int^{\infty}_0 \exp(W_u(y))^2 \dd y }{\left[ \int^{\infty}_0 \exp(W_u(y'))\dd y'\right]^2 } =  \Expe\left[ {\mathcal{Y}_u(0)^{-2}}{ \int^{\infty}_0 \exp(2W_u(y))\dd y }\right].
\end{aligned}
\] Here we used the stationarity \eqref{eq:shestatdist}, so both expectations are time independent, and by \eqref{eq:expoBM} and \eqref{eq:2ndmeandec} below, they are both finite. Thus the term $ \int_0^t \int_0^{\infty} {Z_{u-\frac{1}{2}}(s,y)}{\mathcal{Y}_u(s)^{-1}} \xi(\dd y\dd s)$ is a square-integrable martingale, and \eqref{eq:meandecom} follows.
\end{proof}

To complete the proof of Theorem~\ref{thm:mean}, for any $u<0$, it only remains to compute the expectation on the right hand side of \eqref{eq:meandecom}.

\begin{proof}[Proof of Theorem~\ref{thm:mean}]
It follows from \eqref{eq:expoBM} that for any fixed $u<0$,
\[
\Expe\left[\frac{1}{\mathcal{Y}_u(0)} \right]= \frac{1}{2} \Expe [\gamma(-2u) ]= -u.
\]Define
\[
\mathcal{J}=\int_0^{\infty} \mathrm{e}^{2W(y)+2u y} \dd y, \quad \text{and} \quad \mathcal{K}=\int_0^{\infty} \mathrm{e}^{W(y)+u y} \dd y=\mathcal{Y}_u(0),
\]
which are finite $\Prob$-almost surely since $u<0$.
It remains to compute \[ \Expe\left[ \frac{ \int^{\infty}_0 \exp(2W_u(y))\dd y }{\mathcal{Y}_u(0)^{2}}\right]=
\Expe \frac{ \int^{\infty}_0 \exp(W_u(y))^2 \dd y }{\left[ \int^{\infty}_0 \exp(W_u(y'))\dd y'\right]^2 } =\Expe\left[\frac{\mathcal{J}}{\mathcal{K}^2} \right].\] Note that the same Brownian motion appears in both the numerator and the denominator, which makes the calculation more involved.
By \cite[Example 2.4]{baudoin2011exponential}, for any $u<0$, the conditional Laplace transform takes the form
\[
\Expe\left(\left.\mathrm{e}^{-\frac{1}{2}\lambda ^2 \mathcal{J}} \right| \mathcal{K}=s\right) \Prob\left(\mathcal{K} \in \dd s\right)=\frac{\lambda ^{-2u+1}}{2 \Gamma(-2u)} \frac{\mathrm{e}^{-\lambda  \coth\left(\frac{\lambda s}{2}\right)}}{\left(\sinh \left(\frac{\lambda s}{2}\right)\right)^{-2u+1}} \dd s, \quad \forall s>0.
\]
By differentiating both sides in $\lambda ^2$ and taking the limit $\lambda \to 0$, we have
\[\begin{aligned}
&\Expe\left( \left. -\frac{1}{2}\mathcal{J}\right| \mathcal{K}=s\right) \Prob\left(\mathcal{K} \in \dd s\right)=
\lim_{\lambda \to 0}
\frac{\dd}{\dd \lambda ^2}\Expe\left( \left.\mathrm{e}^{-\frac{1}{2}\lambda ^2 \mathcal{J}} \right| \mathcal{K}=s\right) \Prob\left(\mathcal{K} \in \dd s\right)\\&=\lim_{\lambda \to 0} \frac{\dd}{\dd \lambda ^2} \left(\frac{\lambda ^{-2u+1}}{2 \Gamma(-2u)} \frac{\mathrm{e}^{-\lambda  \coth\left(\frac{\lambda  s}{2}\right)}}{\left(\sinh \left(\frac{\lambda  s}{2}\right)\right)^{-2u+1}}\right) \dd s =\lim_{\lambda \to 0}\frac{1}{2\lambda } \left[\frac{\lambda ^{-2u+1}}{2 \Gamma(-2u)} \frac{e^{-\lambda  \coth\left(\frac{\lambda  s}{2}\right)}}{\left(\sinh \left(\frac{\lambda  s}{2}\right)\right)^{-2u+1}}\right]\Psi(s\viv u,\lambda)
\dd s.
\end{aligned}\]
with
\[\Psi(s\viv u,\lambda) = \left[ \frac{-2u+1}{\lambda }- \coth\left(\frac{\lambda  s}{2}\right)+\frac{{\lambda  s}}{2\left(\sinh \left(\frac{\lambda  s}{2}\right)\right)^2} +\frac{(2u-1)s}{2}\coth \left(\frac{\lambda  s}{2}\right)\right].
\]
Taylor series expressions give that
\[
\coth (x)=x^{-1}+\frac{x}{3}-\frac{x^3}{45}+O(x^5), \quad 0<|x|<\pi,
\]
and
\[
\sinh(x)^{-2}=\operatorname{csch} (x)^2=\left(x^{-1}-\frac{x}{6}+O(x^3)\right)^2=x^{-2}-\frac{1}{3}+O(x^2), \quad 0<|x|<\pi.
\]
Thus
\[\lim_{\lambda \to0} e^{-\lambda  \coth\left(\frac{\lambda  s}{2}\right)} =e^{-\frac{2}{s}}, \quad \lim_{\lambda \to0} \frac{\lambda ^{-2u+1}}{\left(\sinh \left(\frac{\lambda  s}{2}\right)\right)^{-2u+1}} = \left(\frac{2}{s}\right)^{-2u+1},
\]
and
\[\begin{aligned}
&\Psi(s\viv u,\lambda)=(-2u+1)\frac{1}{\lambda }-\frac{2}{s}\frac{1}{\lambda }-\frac{1}{3}\frac{\lambda  s}{2}+\frac{\lambda  s}{2}\left(\left(\frac{\lambda  s}{2}\right)^{-2}-\frac{1}{3}\right)+(2u-1)\frac{s}{2}\left(\frac{2}{s}\frac{1}{\lambda } +\frac{1}{3}\frac{\lambda  s}{2}\right)+O(\lambda ^3)\\
&= -\frac{1}{3}{\lambda  s} + \frac{2u-1}{12}\lambda s^2 +O(\lambda ^3).
\end{aligned}
\]
These results imply that
\[\lim_{\lambda \to 0}\frac{1}{2\lambda } \left[\frac{\lambda ^{-2u+1}}{2 \Gamma(-2u)} \frac{e^{-\lambda  \coth\left(\frac{\lambda  s}{2}\right)}}{\left(\sinh \left(\frac{\lambda  s}{2}\right)\right)^{-2u+1}}\right]\Psi(s\viv u,\lambda)
= \frac{e^{-\frac{2}{s}}}{4\Gamma(-2u)}\left(\frac{2}{s}\right)^{-2u+1}\left(-\frac{1}{3}{s} +\frac{2u-1}{12}s^2 \right),\]
and
\begin{equation}\label{eq:2ndmeandec}
\begin{aligned}&
\Expe\left[\frac{\mathcal{J}}{\mathcal{K}^2} \right] = \int_0^{\infty}\frac{-2}{s^2}\Expe\left( \left. -\frac{1}{2}\mathcal{J}\right| \mathcal{K}=s\right) \Prob\left(\mathcal{K} \in \dd s\right) = \int_0^{\infty}\frac{-e^{-\frac{2}{s}}}{2s^2\Gamma(-2u)}\frac{2^{-2u+1}}{s^{-2u+1}}\left(-\frac{s}{3}  +\frac{2u-1}{12}s^2 \right)\dd s\\
&= \frac{1}{2^{2u}\Gamma(-2u)}\int_0^{\infty} \frac{1}{3} {e^{-\frac{2}{s}}}s^{2u-2}- \frac{2u-1}{12}e^{-\frac{2}{s}} s^{2u-1} \dd s =\frac{1}{2^{2u}\Gamma(-2u)}\left(\frac{1}{3}\frac{\Gamma(-2u+1)}{2^{-2u+1}} -\frac{2u-1}{12}\frac{\Gamma(-2u)}{2^{-2u}}  \right)\\
&= -\frac{u}{3} -\frac{2u-1}{12} = -\frac{u}{2}+\frac{1}{12}.
\end{aligned}\end{equation}
Now by \eqref{eq:meandecom}, we have
\[
{\Expe[H_u(t,0)]}  = \left(-\frac{1}{2} \left(u-\frac{1}{2}\right)\Expe\left[\frac{1}{\mathcal{Y}_u(0)} \right]- \frac{1}{2}\Expe\left[\frac{\mathcal{J}}{\mathcal{K} ^2} \right] \right)t=   \left[\frac{1}{2} \left(u-\frac{1}{2}\right)u -\frac{1}{2}\left(-\frac{u}{2}+\frac{1}{12}\right) \right]t = \left(\frac{u^2}{2}-\frac{1}{24}\right) t.
\]

We next extend to $u=0$. For any fixed $t>0$, by Proposition~\ref{pr:mombds0}, $H_u(t,0)$ is uniformly integrable for $u$ in any compact subset of  $\R$. Thus the convergence $\Expe \left[H_{u}(t,0)\right] \to \Expe \left[H_{0}(t,0)\right] $ as $u\to 0$ follows from the convergence in probability \[Z_{u-\frac{1}{2}}(t,0) \stackrel{\text{prob}}{\to} Z_{-\frac{1}{2}}(t,0) \quad \text{as } u\to 0,\]  together with the continuous mapping theorem. To show the above convergence in probability, we note that by the Minkowski and H\"older's inequalities,
\[
\begin{aligned}
\|Z_{u-\frac{1}{2}}(t,0)-Z_{-\frac{1}{2}}(t,0)\|_2 &\leq \int_0^{\infty} \|\mathcal{Z}_{u-\frac{1}{2}}(t,0\viv 0, y)-\mathcal{Z}_{-\frac{1}{2}}(t,0\viv 0,y)\|_4\|\exp(W(y)+uy) \|_4 \dd y
\\&+\int_0^{\infty} \|\mathcal{Z}_{-\frac{1}{2}}(t,0\viv 0, y)\|_4\|\exp(W(y)+uy)-\exp(W(y)+y) \|_4 \dd y.
\end{aligned}
\]
By the $\Prob$-almost surely continuity of the Green's function $\mathcal{Z}_\cdot(\cdot,\cdot\viv \cdot,\cdot)$ in Proposition~\ref{pr:greens}, the uniform moment bounds in \eqref{eq:greensPosMom}, and the dominated convergence theorem, we have $\|Z_{u-\frac{1}{2}}(t,0)-Z_{-\frac{1}{2}}(t,0)\|_2 \to 0$ as $u\to 0$.

For any $x>0$, the result follows from \eqref{eq:station}.
\end{proof}

\section{Variance identity: proof of Theorem~\ref{thm:varianceID}}\label{se:varID}
{To prove the identity \eqref{eq:intmidupbd}, which relates the variance of the height function to the endpoint displacement of the directed polymer, we utilize an integration-by-parts formula in the Gaussian space generated by the Brownian motion $W(x)$.  Although the model involves two sources of Gaussianity, we only rely on the one from the initial data. This method heavily relies on the fact that the Brownian motion with drift $W(x)+ux$ is the stationary measure of the increment process $H_u(t,x)-H_u(t,0)$ (see \eqref{eq:station}).} 

{In the following argument, we consider the four height increments along different directions: $A=H_u(t,0)-H_u(0,0)$, $B=H_u(t,x)-H_u(t,0)$, $C=H_u(t,x)-H_u(0,x)$ and $D=H_u(0,x)-H_u(0,0)$. The quantity of interest is $A$, while $B$ has Gaussian distribution due to the choice of initial data, $C$ shares the same distribution as the height increment in the full space case when $x\gg1$, and $D =W(x)+ux$ is the Gaussian random variable through which we perform the integration by parts. Similar arguments appeared in various previous works, see \cite{timoEJP,gu2023another}.} 

\begin{proof}[Proof of Theorem~\ref{thm:varianceID}] 
Fix $u\in\R, t>0$.
Define the height increment of \eqref{eq:hsKPZ} in the time interval $[0,t]$ for any $x \in [0, \infty)$ as \[\mathcal{H}_u(t,x) := H_u(t,x) - H_u(0,x) = \log Z_{u-\frac{1}{2}}(t,x) - W_{u}(x) .\] 
Then we have
\[
H_u(t,x) - H_u(t,0) = \mathcal{H}_u(t,x)  - \mathcal{H}_u(t,0) + W_u(x),
\]
which implies that
\[
\Var [H_u(t,x) - H_u(t,0)] = \Var [\mathcal{H}_u(t,x)  - \mathcal{H}_u(t,0) + W_u(x)].
\]
By \eqref{eq:station}, for any $t\geq 0$, $\Var [H_u(t,x) - H_u(t,0)]=\Var [W_u(x)] = x$. Thus we have an identity
\[
\begin{aligned}
&0=\Var [\mathcal{H}_u(t,x)-\mathcal{H}_u(t,0) ] +  2\Cov [\mathcal{H}_u(t,x)  - \mathcal{H}_u(t,0), W_u(x)]\\
&= \Var [\mathcal{H}_u(t,x) ]-2\Cov[\mathcal{H}_u(t,x), \mathcal{H}_u(t,0)] +\Var [\mathcal{H}_u(t,0)] +  2\Cov [\mathcal{H}_u(t,x)  - \mathcal{H}_u(t,0), W_u(x)]\\
& = \Var [\mathcal{H}_u(t,0)] +  \left(\Var [\mathcal{H}_u(t,x) ] + 2\Cov [\mathcal{H}_u(t,x), W_u(x)]\right)\\&-2\Cov[\mathcal{H}_u(t,x), \mathcal{H}_u(t,0)] - 2\Cov [\mathcal{H}_u(t,0), W_u(x)].
\end{aligned}
\]
Since $H_u(t,0)=\mathcal{H}_u(t,0)$, we have $\Var[H_u(t,0)] =\Var [\mathcal{H}_u(t,0)]$. The above equation can be rewritten as
\begin{align*}
    \Var[H_u(t,0)] &= 2\Cov [{H}_u(t,0) , W_u(x)]  + 2\Cov[\mathcal{H}_u(t,x), \mathcal{H}_u(t,0)]  \\&- (\Var[\mathcal{H}_u(t,x)]  + 2\Cov [\mathcal{H}_u(t,x), W_u(x)]).
\end{align*}
{Recalling that $\mathbb{E}^{\mathbb{Q}^{u,t}}_{x}[X_t]$ is the (quenched) expectation of the half-space CDRP endpoint. The endpoint has a quenched density $\rho^{R}_{u}(\cdot \viv t,x)$ as defined in \eqref{eq:hsQuench0}.}
The proof is complete once we establish the following results:
\begin{equation}\label{eq:polymlimit}
\Cov [{H}_u(t,0) , W_u(x)] \to \Expe \int_0^{\infty} y \rho^{R}_{u}(y \viv t,0) \dd y 
 {= \Expe \mathbb{E}^{\mathbb{Q}^{u,t}}_{0}[X_t] }, \quad \text{as }x \to \infty;
\end{equation}
\begin{equation}\label{eq:heatlimit}
\Var[\mathcal{H}_u(t,x)]  + 2\Cov [\mathcal{H}_u(t,x), W_u(x)]  \to ut, \quad \text{as } x \to \infty;
\end{equation}
and
\begin{equation}\label{eq:covlimit}
\Cov[\mathcal{H}_u(t,x), \mathcal{H}_u(t,0)]  \to 0, \quad \text{as } x \to \infty.
\end{equation}
We will prove equations \eqref{eq:polymlimit} and \eqref{eq:heatlimit} in the following subsections. The proof of \eqref{eq:covlimit} is very similar to the proof of the analogous result for full-space KPZ equation in \cite{gu2023another}, so we put it in Appendix \ref{ap:covlimit}. 
\end{proof}
\begin{remark}
{If using the above definitions for the four height increments $A,B,C,D$, the computation above is equivalent to taking the variance on both sides of the identity $B=D+C-A$. The value of interest is $\Var[A]$. With stationarity \eqref{eq:station}, $\Var[B]=\Var[D]$. We perform the integration-by-parts to prove the limit of $\Cov[A,D]$ as $x\to\infty$ in \eqref{eq:polymlimit}. For the terms $\Var[C]+2\Cov[C,D]$, note that they only involve the height function $H_u(t,\cdot)$ evaluated at point $x$ away from the boundary. Thus we use its full-space counterpart's value to approximate it when $x\gg 1$ in \eqref{eq:heatlimit}. The remaining term $\Cov[A,C]$ does not involve any KPZ behavior. It vanishes when $x\gg 1$ at any finite time $t$, which is stated in \eqref{eq:covlimit}. }
\end{remark}

\subsection{Proof of \eqref{eq:polymlimit}}
We prove \eqref{eq:polymlimit} through an integration-by-parts in the Gaussian space generated by the Brownian motion $W(x)$. The following is analogous to the full-space result \cite[Lemma 2.4]{gu2023another}, but a few details are different in the half-space setting. As discussed above after Proposition~\ref{pr:greens}, due to   the boundary, the Green's function $ \mathcal{Z}_\mu(t,x\viv 0,y)$ is not spatially stationary.

Let $\mathscr{W}$ be the spatial white noise associated with the Brownian motion $W(\cdot)$. Let $\malD$ be the Malliavin derivative operator on the Gaussian probability space generated by $\mathscr{W}$. 

For each $t,x,z\geq0$ and each realization of $\xi$, we apply an integration-by-parts on the Gaussian space generated by $\mathscr{W}$. Further taking an expectation over $\xi$, we obtain
\[
\begin{aligned}
&\Cov[H_u(t,z), W_u(x)]=\Expe[H_u(t,z)W(x)]\\ &= \Expe[H_u(t,z)\int_0^{\infty} \mathbbm{1}_{[0,x]}(r)\mathscr{W}(r)\dd r] = \Expe\langle\malD H_u(t,z), \1_{[0,x]}\rangle.  
\end{aligned}
\]
{The integral above is the standard Wiener integral since $\mathbbm{1}_{[0,x]}(r)$ is a deterministic function,} and the bracket denotes the inner product in $L^2(\R)$. By \eqref{eq:kpzGreensform}, for any $r\geq 0$, \[
\begin{aligned}
\malD_r H_u(t,z) &= \malD_r \log \int_0^{\infty}\mathcal{Z}_{u-\frac{1}{2}}(t,z\viv 0,y)\exp (W_{u}(y)) \dd y = \frac{\malD_r [\int_0^{\infty}\mathcal{Z}_{u-\frac{1}{2}}(t,z\viv 0,y)\exp (W_{u}(y)) \dd y]}{\int_0^{\infty}\mathcal{Z}_{u-\frac{1}{2}}(t,z\viv 0,y)\exp (W_{u}(y)) \dd y} \\&= \frac{\int_0^{\infty}\mathcal{Z}_{u-\frac{1}{2}}(t,z\viv 0,y)\exp (W_{u}(y))\1_{[0,y]}(r) \dd y}{\int_0^{\infty}\mathcal{Z}_{u-\frac{1}{2}}(t,z\viv 0,y)\exp (W_{u}(y)) \dd y}= \int_r^{\infty}\rho^{R}_{u}(y \viv t,z) \dd y.
\end{aligned}
\]
This implies that 
\[
\begin{aligned}
&\Expe\langle\malD H_u(t,z), \1_{[0,x]}\rangle = \Expe \int_0^x \int_r^{\infty}\rho^{R}_{u}(y \viv t,z)\dd y \dd r = \Expe \int_0^{\infty} \min{(x,y)}\rho^{R}_{u}(y \viv t,z)\dd y\\
&= \Expe \int_0^{x} y \rho^{R}_{u}(y \viv t,z) \dd y + x \Expe  \int_x^{\infty} \rho^{R}_{u}(y \viv t,z)\dd y.
\end{aligned}
\]
Setting $z=0$, we have
\begin{equation}\label{eq:covlimitIBP}
\Cov[H_u(t,0), W_u(x)] = \Expe \int_0^{x} y \rho^{R}_{u}(y \viv t,0) \dd y + x \Expe  \int_x^{\infty} \rho^{R}_{u}(y \viv t,0) \dd y.
\end{equation}
By H\"older's inequality, \eqref{eq:greensPosMom}, \eqref{eq:expBMbd} and \eqref{eq:negCx}, for any $t>0, y \in [0,\infty)$, we have
\begin{equation}\label{eq:quedenmom}
\begin{aligned}
&\Expe \rho^{R}_{u}(y \viv t,0) =\Expe \frac{ \mathcal{Z}_{u-\frac{1}{2}}(t,0\viv 0,y)\exp (W_{u}(y))}{\int_0^{\infty}\mathcal{Z}_{u-\frac{1}{2}}(t,0\viv 0,y')\exp (W_{u}(y')) \dd y'}\\
&\leq \|\mathcal{Z}_{u-\frac{1}{2}}(t,0\viv 0,y)\|_4 \|\exp (W_{u}(y))\|_4 \left\|Z_{u-\frac{1}{2}}(t,0)^{-1}\right\|_2 \leq C \exp(-y^2/C),
\end{aligned}
\end{equation}
for some constant $C=C(t,u)$. Taking the limit $x\to \infty$ on both sides of \eqref{eq:covlimitIBP} gives  \eqref{eq:polymlimit}.

\subsection{Proof of \eqref{eq:heatlimit}}
The idea is that, for large $x\gg1$ and fixed $t>0$, the value $H_u(t,x)$ almost does not feel the boundary effects, hence the result should be close to the full space case.
We proceed in two steps. First, we show that \eqref{eq:heatlimit} holds with $H_u(t,x)$ replaced by  $H^{fs}_u(t,x)$, where ${H}^{fs}_u(t,x)$ solves the KPZ equation on full-space starting from a two-sided Brownian motion with drift $u$. In the second step, we show that the error induced by the approximation vanishes as $x\to\infty$, namely,   ${\mathbf{V}}\mathbf{ar}[{\mathcal{H}}^{fs}_u(t,x)]  + 2\mathbf{Cov} [\mathcal{H}^{fs}_u(t,x), {W}^{fs}_u(x)]$ can be used to approximate the left-hand side of \eqref{eq:heatlimit} when $x \to \infty$. 

Let us begin with  a few definitions. For $(t,x) \in [0, \infty)\times \R$ and $u \in \R$, let $ Z^{fs}_{u}(t, x)$ be the solution to the (1+1)-dimensional stochastic heat equation on the full space 
\begin{equation}\label{eq:fsSHE}
\begin{aligned}
\partial_t Z^{fs}_{u}(t, x) &=\frac{1}{2} \partial_x^2 Z^{fs}_{u}(t, x)+Z^{fs}_{u}(t, x){\xi}(t, x),\\
Z^{fs}_u(0,x) &= \exp({W}^{fs}_{u}(x)),
\end{aligned}
\end{equation}
where ${\xi}$ is the same space-time white noise on $(t,x)\in\R^2$ as in \eqref{eq:hsSHE} and ${W}^{fs}_{u}(x) := {W}^{fs}(x)+u x$ with $W^{fs}$ being a two-sided Brownian motion on $\R$. For simplicity, we assume that $W^{fs}(\cdot)$ is defined on our probability space $(\Omega, \filt, \Prob)$ and it coincides with $W(\cdot)$ on $[0,\infty)$ for each $\omega \in \Omega$. 

The solution to the KPZ equation on the full space is given by the Hopf-Cole transform 
\begin{equation}\label{eq:fsKPZ}
{H}^{fs}_u(t,x):= \log Z^{fs}_{u} (t,x).
\end{equation}
In particular, ${H}^{fs}_u(0,x)={W}^{fs}_u(x)$. For $x \in [0,\infty)$, $H^{fs}_u(0,x) = {H}_u(0,x) = W_u(x)$. Note that the boundary parameter $u$ appears in   the half-space SHE \eqref{eq:hsSHE} itself, while in \eqref{eq:fsSHE}, $u$ is only parametrizing the initial condition. Thus one needs to be careful to include a shift term $1/2$ when referring to a proper half-space SHE. We have $Z^{fs}_u(0,x) = {Z}_{u-\frac{1}{2}}(0,x) = \exp(W_u(x))$ when $x \in [0,\infty)$.

It is a classical result (see  \cite{bertini1997stochastic}) that for each $u \in \R$, \eqref{eq:fsSHE} has a unique mild solution satisfying
\begin{equation}\label{eq:mildformfsSHE}
Z^{fs}_{u}(t, x) = \int_\R p_{t}(x-y)\exp(W^{fs}_u(y))\dd y + \int_0^t\int_\R p_{t-s}(x-y) Z^{fs}_u(s,y)\xi(\dd y\dd s),
\end{equation}
and for any $\tau>0$, $\sup_{0 \leq t \leq \tau}\sup_{x \in \R}e^{-a|x|} \Expe \left[{Z}_{u}^{fs}(t, x)^2\right]<\infty$ {for some $a>0$.} It is also well-known (see \cite{funaki2015kpz}) that when $H_u^{fs}(\cdot,\cdot)$ starts from the initial data $W^{fs}_u(\cdot)$, the increment process $\{H_u^{fs}(t,\cdot)-H_u^{fs}(t,0)\}_{t\geq0}$ is stationary in time.

We use $\mathcal{Z}^{fs}(t,x\viv s,y) $ to denote the Green's function of the (1+1)-dimensional full-space SHE, of which we refer to \cite{alberts2022green} for further properties. In particular, there is a convolution formula
\begin{equation}\label{eq:fsSHEconvolution}
Z^{fs}_u(t,x) = \int_{-\infty}^{\infty}  \mathcal{Z}^{fs}(t,x\viv 0,y) \exp(W^{fs}_u(y))\dd y.
\end{equation}

Now define $\mathcal{H}^{fs}_u(t,x):= H^{fs}_u(t,x) - H^{fs}_u(0,x)$. We first prove the following result, which is almost \eqref{eq:heatlimit} but with $\mathcal{H}_u(t,x) $ substituted by $\mathcal{H}_u^{fs}(t,x)$. 
\begin{proposition}\label{prop:wslimit}
For any $u \in \R, t>0$,
\[
\mathbf{Var}[\mathcal{H}^{fs}_u(t,x)]  + 2\mathbf{Cov} [\mathcal{H}^{fs}_u(t,x), W^{fs}_u(x)]  \to ut, \quad \text{as } x \to \infty.
\]
\end{proposition}
\begin{proof}
Using that $W^{fs}_u(\cdot)$ is the stationary measure of the increment process $\{H^{fs}_u(t,\cdot)-H^{fs}_u(t,0)\}_{t\geq0}$, and the fact that for any fixed  $t>0$, $\{\mathcal{H}^{fs}_u(t,x)\}_{x\in\R}$ is stationary, following the argument for \cite[Lemma 2.3]{gu2023another}, we derive that
\[
\begin{aligned}
& \mathbf{Var}[\mathcal{H}^{fs}_u(t,x)]  + 2\mathbf{Cov} [\mathcal{H}^{fs}_u(t,x), W^{fs}_u(x)] = \mathbf{Var}[\mathcal{H}^{fs}_u(t,0)]  + 2\mathbf{Cov} [\mathcal{H}^{fs}_u(t,x), W^{fs}_u(x)] \\
&= \mathbf{Cov} [\mathcal{H}^{fs}_u(t,x)+ \mathcal{H}^{fs}_u(t,0), W^{fs}_u(x)] + \mathbf{Cov} [\mathcal{H}^{fs}_u(t,x), \mathcal{H}^{fs}_u(t,0)].
\end{aligned}
\]
For the second term on the right-hand side, by adapting the proof in \cite[Appendix A]{gu2023another} to include the drift term, we can show \[\mathbf{Cov} [\mathcal{H}^{fs}_u(t,x), \mathcal{H}^{fs}_u(t,0)] \to 0 \text{ as } x \to \infty.\]
It remains to analyze the term $\mathbf{Cov} [\mathcal{H}^{fs}_u(t,x)+ \mathcal{H}^{fs}_u(t,0), W^{fs}_u(x)]$. 

Let $\rho^{fs}_{u}(y\viv t,x)$ be the quenched density of the (1+1)-dimensional continuum directed polymer model with boundary $\exp(W^{fs}_u(\cdot))$, starting from $(t,x)$ and running backward to time zero:
\[
\rho^{fs}_{u}(y \viv t,x):=\frac{\mathcal{Z}^{fs}(t,x\viv 0,y)\exp(W^{fs}_u(y))}{\int_\R\mathcal{Z}^{fs}(t,x\viv 0,y')\exp(W^{fs}_u(y')) \dd y'}.
\]
Using the same integration-by-parts method as in the proof of \eqref{eq:polymlimit} and \cite[Lemma 2.5]{gu2023another}, for any $t>0$ and $x,z \geq 0$, we have
\[ 
\begin{aligned}
\mathbf{Cov} [{H}^{fs}_u(t,z), W^{fs}_u(x)] &=\Expe \int_{0}^{\infty}\rho^{fs}_{u}(y \viv t,z)\min(x,y)\dd y
 \\ &=\Expe \int_{-z-ut}^{\infty}\rho^{fs}_{0}(y \viv t,0)\min(x,y+z+ut)\dd y.
\end{aligned}
\]
Here, for the second ``='', we used the fact that 
\[
\{\rho^{fs}_{u}(y+z+ut \viv t,z)\}_{y\in\R}\stackrel{\text{law}}{=}\{\rho^{fs}_{0}(y \viv t,0)\}_{y\in\R}, 
\]
which follows from the spatial stationarity of $\mathcal{Z}^{fs}$ and the identity in law:
\[\left\{\mathcal{Z}^{fs}(t, u t+y\viv 0,0) e^{u y+\frac{1}{2} u^2 t}\right\}_{t>0, y \in \mathbb{R}} \stackrel{\text { law }}{=}\left\{\mathcal{Z}^{fs}(t, y\viv 0, 0)\right\}_{t>0, y \in \mathbb{R}}.
\]
This relation was proved in \cite[Proposition 2.7]{gu2023another} and arises from the shear invariance of the space-time white noise ${\xi}$. It follows that for fixed $u\in \R, t>0$ and {any $x \geq0$}, 
\[
\begin{aligned}
& \mathbf{Cov} [\mathcal{H}^{fs}_u(t,x)+ \mathcal{H}^{fs}_u(t,0), W^{fs}_u(x)] \\
&= \mathbf{Cov} [{H}^{fs}_u(t,x) - W^{fs}_u(x) + {H}^{fs}_u(t,0), W^{fs}_u(x)]
\\
&= -x + \Expe\int_{-x-ut}^{\infty}\rho^{fs}_{0}(y\viv t,0)\min(x,y+x+ut)\dd y+\Expe \int_{-ut}^{\infty} \rho^{fs}_0(y\viv t,0)\min(x,y+ut)\dd y \\
&= -x+\Expe \int_{-x-ut}^{x-ut}\rho^{fs}_{0}(y\viv t,0)(y+x+ut)\dd y+ 2x  \Expe \int_{x-ut}^{\infty}\rho^{fs}_{0}(y\viv t,0) \dd y.\\
&= \Expe \int_{-x-ut}^{x-ut}(y+ut) \rho^{fs}_{0}(y\viv t,0)\dd y  + x  \Expe \int_{x-ut}^{\infty}\rho^{fs}_{0}(y\viv t,0) \dd y - x \Expe \int_{-\infty}^{-x-ut}\rho^{fs}_{0}(y\viv t,0) \dd y.
\end{aligned}
\]
Since $\Expe \rho^{fs}_0(t,0;y)$ is an even probability density on $y\in \R$, \[
\begin{aligned}
\Expe \int_{-x-ut}^{x-ut}(y+ut)\rho^{fs}_{0}(y\viv t,0)\dd y \to ut, \quad \text{ when } x\to \infty.\end{aligned}
\]
By \cite[Lemma 2.6]{gu2023another}, for any $t>0$, $\Expe \rho^{fs}_{0}(y\viv t,0) \leq C \exp(-y^2/C)$ for some positive constant $C=C(t)$. Thus we also have that
\[ x  \Expe \int_{x-ut}^{\infty}\rho^{fs}_{0}(y\viv t,0) \dd y - x \Expe \int_{-\infty}^{-x-ut}\rho^{fs}_{0}(y\viv t,0) \dd y \to 0, \quad \text{ when } x\to \infty.
\]
The proof is complete.
\end{proof}

The next proposition is to justify the substitution. It shows that when $x \to \infty$, the statistics of $\mathcal{H}_u(t,x)$ can be approximated by the statistics of $\mathcal{H}^{fs}_u(t,x)$.
\begin{proposition}\label{pr:substitution}
For any $u\in\R, t>0$, we have
\begin{equation}\label{eq:varsubstitution}
\left|\mathbf{Var}[\mathcal{H}^{fs}_u(t,x)]   - \Var[{\mathcal{H}}_u(t,x)] \right| \to 0, \quad \text{as } x\to \infty;
 \end{equation}
 and 
\begin{equation}\label{eq:covsubstitution}
\left|\mathbf{Cov}  [\mathcal{H}^{fs}_u(t,x), W^{fs}_u(x)]  - \Cov [{\mathcal{H}}_u(t,x), {W}_u(x)]\right| \to 0, \quad \text{as } x\to \infty.
\end{equation}
\end{proposition}
\begin{proof}
First, we claim that for any $u\in \R, t>0$, there exists a positive constant $C=C(u,t)$ such that 
\begin{equation}\label{eq:approxfshs0} \| \mathcal{H}^{fs}_u(t,x) - \mathcal{H}_u(t,x)\|_2=\| {H}^{fs}_u(t,x) - {H}_u(t,x)\|_2
 \leq C\exp(-x^2/C), \quad\quad x\geq 0.\end{equation}
By a straightforward computation, we have
\[
\begin{aligned}
&\mathbf{Var}[\mathcal{H}^{fs}_u(t,x)] {-\mathbf{Var}[\mathcal{H}_u(t,x)]} =  \Expe\left( [\mathcal{H}^{fs}_u(t,x) - \mathcal{H}_u(t,x)] [\mathcal{H}^{fs}_u(t,x) + \mathcal{H}_u(t,x)]\right) \\&-  \Expe [\mathcal{H}^{fs}_u(t,x) - \mathcal{H}_u(t,x)]  \Expe[\mathcal{H}^{fs}_u(t,x) + \mathcal{H}_u(t,x)],
\end{aligned}
\]
and it follows that 
\[\left|\mathbf{Var}[\mathcal{H}^{fs}_u(t,x)]   - \Var[{\mathcal{H}}_u(t,x)] \right| \leq 2 \| \mathcal{H}^{fs}_u(t,x) - \mathcal{H}_u(t,x)\|_2 \| \mathcal{H}^{fs}_u(t,x) + \mathcal{H}_u(t,x)\|_2.\]
When $x \geq 0$, \[
\| \mathcal{H}^{fs}_u(t,x) + \mathcal{H}_u(t,x)\|_2 \leq  \|H^{fs}_u(t,x)\|_2 +  \|H^{fs}_u(0,x)\|_2 +  \|{{H}}_u(t,x)\|_2 +  \|{{H}}_u(0,x)\|_2,
\]
with $H^{fs}_u(0,x) = {{H}}_u(0,x) = W_u(x)$. By \eqref{eq:HbdCx}, $\|{H}_u(t,x)\|_2 \leq C_1\exp(C_1x)$ for some positive constant $C_1=C_1(u,t)$. By a similar argument for the full-space KPZ equation, there exists another positive constant $C_2=C_2(u,t)$ so that $\|H^{fs}_u(t,x)\|_2 \leq C_2\exp(C_2x)$. Then the convergence in \eqref{eq:varsubstitution} follows from these bounds and \eqref{eq:approxfshs0}.

Similarly, for the covariance, we have
\[
\begin{aligned}
&\left|\mathbf{Cov}  [\mathcal{H}^{fs}_u(t,x), W^{fs}_u(x)]  - \Cov [{\mathcal{H}}_u(t,x), {W}_u(x)]\right| = |\Expe[\mathcal{H}^{fs}_u(t,x) W^{fs}(x)] - \Expe[{\mathcal{H}}_u(t,x) {W}(x)] |\\
&=  \left|\Expe\left[\left(\mathcal{H}^{fs}_u(t,x) -\mathcal{H}_u(t,x)\right)W (x)\right]\right| \leq \|\mathcal{H}^{fs}_u(t,x) -\mathcal{H}_u(t,x)\|_2 \|W (x)\|_2.
\end{aligned}
\]
Thus \eqref{eq:covsubstitution} also follows from \eqref{eq:approxfshs0}.

To complete the proof of the proposition, it remains to prove \eqref{eq:approxfshs0}. For any positive real numbers $\alpha, \beta>0$, we have the elementary inequality \[ |\log \alpha - \log \beta|  
\leq \frac{|\alpha-\beta|}{\min(\alpha,\beta)} \leq \frac{|\alpha-\beta|}{\alpha}+\frac{|\alpha-\beta|}{\beta}.\]
Therefore, for any $t>0, x \geq 0$
\[
\begin{aligned}
&\|H^{fs}_u(t,x) - {H}_{u}(t,x) \|_2 = \| \log Z^{fs}_u(t,x) -\log {Z}_{u-\frac{1}{2}}(t,x) \|_2\\&\leq \left\| {|Z^{fs}_u(t,x) - {Z}_{u-\frac{1}{2}}(t,x)|}Z^{fs}_u(t,x)^{-1} + {|Z^{fs}_u(t,x) - {Z}_{u-\frac{1}{2}}(t,x)|}{ {Z}_{u-\frac{1}{2}}(t,x)^{-1}} \right\|_2\\
 &\leq \|Z^{fs}_u(t,x) - {Z}_{u-\frac{1}{2}}(t,x)\|_{4} \|Z^{fs}_u(t,x)^{-1}\|_{4}+\|Z^{fs}_u(t,x) - {Z}_{u-\frac{1}{2}}(t,x)\|_{4} \|{Z}_{u-\frac{1}{2}}(t,x)^{-1}\|_{4}.
\end{aligned}
\]
By \eqref{eq:negCx} and a similar result for the full-space SHE, there exists a positive constant $C_3=C_3(u,t)$ such that for any $x\geq 0$,
\[\|{Z}_{u-\frac{1}{2}}(t,x)^{-1}\|_{4} \leq C_3\exp(C_3x) \quad \text{and}\quad \|{Z}^{fs}_{u}(t,x)^{-1}\|_{4} \leq C_3\exp(C_3x).
\]
Thus, the proof of \eqref{eq:approxfshs0} will be complete if we can show there exists a positive constant $C_4=C_4(u,t)$ such that
\begin{equation}\label{eq:approxFSshe}
\|Z^{fs}_u(t,x) - {Z}_{u-\frac{1}{2}}(t,x)\|_{4} \leq C_4\exp(-x^2/C_4), \quad \forall x \geq 0.
\end{equation}
Applying Lemma~\ref{le:sheinftydif} below, we complete the proof.
\end{proof}
%
%


The following lemma shows that, for fixed $t$ and large $x$, the solution to the SHE with Robin boundary condition does not ``feel'' the boundary.
\begin{lemma}\label{le:sheinftydif}
For any $\tau>0$, $u \in \R$, $p \geq 2$, there exists some positive constant $C=C(\tau,u,p)$ such that
\[
\|Z^{fs}_u(t,x) - {Z}_{u-\frac{1}{2}}(t,x)\|_p \leq C\exp(-x^2/t), \quad \text{for all } t \in (0,\tau] \text{ and } x \geq 0.\]
\end{lemma}
\begin{proof}
We use $\alpha \lesssim \beta$ to denote $\alpha \leq C \beta$ with any constant $C=C(\tau, u,p)$.


In the proof, we use the chaos expansion for the half-space SHE \eqref{eq:hsSHE} and its full-space counterpart \eqref{eq:fsSHE}. For any $x\geq 0$, we write
\[
\begin{aligned}
Z^{fs}_u(t,x) - {Z}_{u-\frac{1}{2}}(t,x) = \sum_{k=0}^{\infty} z^{fs}_k(t,x) -  \sum_{k=0}^{\infty} z_k(t,x),
\end{aligned}
\]
with
\[\begin{aligned}
z_0(t,x) &:
= \int_0^{\infty} \kernel_{u-\frac{1}{2}}^R(t,x\viv 0,y)\exp({W}_{u}(y)) \dd y,\\
z^{fs}_0(t,x) &:= \int_{-\infty}^{\infty}  p_t(x-y)\exp(W^{fs}_{u}(y)) \dd y,\end{aligned}
\]
and for any $n\geq 0$,
\[\begin{aligned}
z_{n+1}(t,x) &:= \int_0^t\int_0^{\infty}  \kernel_{u-\frac{1}{2}}^R(t,x\viv s,w) z_n(s,w) \xi(\dd w \dd s),\\
z^{fs}_{n+1}(t,x) &:= \int_0^t\int_{-\infty}^{\infty} p_{t-s}(x-w) z^{fs}_n(s,w) \xi(\dd w \dd s).
\end{aligned}
\]
The above expansion of $Z_{u-\frac{1}{2}}$ can be obtained through the chaos expansion of the Greens function \eqref{eq:greensdef} together with the convolution formula \eqref{eq:sheGreensform} and the stochastic Fubini theorem. 
Similarly, the expansion of $Z^{fs}_u(t,x)$ follows from the convolution formula \eqref{eq:fsSHEconvolution} together with the chaos expansion of the Green's function $\mathcal{Z}^{fs}$, as proved in \cite{alberts2014intermediate}.

We shall use the following inequalities later to estimate the difference between the integrals involving the Robin heat kernel and the ones involving the standard heat kernel.
For any $t>0$ and $x,y\geq 0$, we can bound
$p_{t}(x+y) \leq p_{t}(x)e^{-\frac{y^2}{2t}}$. Also, for any $\tau >0$, there exists some constant $C=C(\tau,u)$ so that for any $0\leq s<t \leq \tau$ and $x,y\geq 0$,
\[
\int_0^{\infty} p_{t-s}(x+y+z) e^{-(u-\frac{1}{2})z} \dd z \leq   p_{t-s}(x+y)\int_0^{\infty}e^{-\frac{z^2}{2(t-s)}} e^{-(u-\frac{1}{2})z} \dd z \leq C p_{t-s}(x+y) \sqrt{t-s}.
\]
By the explicit expression \eqref{eq:kernelRobin} and the above estimates, we estimate the difference between $z_0$ and $z^{fs}_0$ as  
\begin{equation}\label{eq:zeroterm}
\begin{aligned}
&\|{z}_0(t,x)  - z^{fs}_0(t,x)  \|_p
=\left\|\int_0^{\infty}\kernel_{u-\frac{1}{2}}^R(t,x\viv 0,y)\exp({W}_{u}(y)) \dd y -  \int_{-\infty}^{\infty}  p_{t}(x-y)\exp(W^{fs}_u(y))\dd y\right\|_p \\
&=\left\|\int_0^{\infty} \left(p_{t}(x+y)-(2u-1) \int_0^{\infty} p_{t}(x+y+z) e^{-(u-\frac{1}{2})z} \dd z\right)e^{{W}_{u}(y)} \dd y-  \int_{-\infty}^0 p_{t}(x-y)e^{W^{fs}_u(y)}\dd y\right\|_p\\
& \leq \int_0^{\infty} \left|p_{t}(x+y)-(2u-1) \int_0^{\infty} p_{t}(x+y+z) e^{-(u-\frac{1}{2})z} \dd z\right|\|e^{{W}_{u}(y)}\|_p \dd y  + \int_{-\infty}^0 p_{t}(x-y)\|e^{W^{fs}_u(y)}\|_p\dd y\\
&\leq p_{t}(x)(1+ |2u-1|C \sqrt{t})\int_0^{\infty}  e^{-\frac{y^2}{2t}}\| \exp({W}_{u}(y))\|_p\dd y  + p_{t}(x) \int_{0}^{\infty} e^{-\frac{y^2}{2t}}\| \exp(W^{fs}_{u}(-y))\|_p\dd y \\
& \lesssim p_t(x)\sqrt{t}.
\end{aligned}
\end{equation}
Now for any $n \geq 0$,
\[
\begin{aligned}
&z_{n+1}(t,x)  - z^{fs}_{n+1}(t,x) \\&=\int_0^t\int_0^{\infty}  \kernel_{u-\frac{1}{2}}^R(t,x\viv s,w) z_n(s,w) \xi( \dd w \dd s)- \int_0^t\int_{-\infty}^{\infty} p_{t-s}(x-w) z^{fs}_{n}(s,w) \xi( \dd w\dd s)\\
&=I_n(t,x)+J_n(t,x),
\end{aligned}
\]
with
\[\begin{aligned}
I_n(t,x)&:=  \int_0^t \int_0^{\infty} p_{t-s}(x-w) [z_n(s,w) -z^{fs}_n(s,w)] \xi( \dd w \dd s),\\
\text{and}\quad  J_n(t,x)&:=  \int_0^t \int_0^{\infty} \left(p_{t-s}(x+w)-(2u-1) \int_0^{\infty} p_{t-s}(x+w+z) e^{-(u-\frac{1}{2})z} \dd z\right) z_n(s,w) \xi( \dd w \dd s)\\
&-  \int_0^t\int_{-\infty}^0p_{t-s}(x-w) z^{fs}_{n}(s,w)\xi( \dd w \dd s).
\end{aligned}
\]
While $I_n$ is not a martingale in $t$, the process
\[
I_n(v):=  \int_0^v \int_0^{\infty} p_{t-s}(x-w) [z_n(s,w) -z^{fs}_n(s,w)]\xi( \dd w \dd s)
\] is a martingale in $v$ so we could apply the Burkholder-Davis-Gundy inequality to obtain\[
\begin{aligned}
\|I_n(t,x)\|^2_p & \lesssim   \int_0^t \int_0^{\infty}p_{t-s}(x-w)^2  \|z_n(s,w) -z^{fs}_n(s,w)\|^2_{p} \dd s \dd w.
\end{aligned}
\]
Similarly, we can convert the two terms in $J_n$ into martingales and apply the Burkholder-Davis-Gundy inequality:
\[\begin{aligned}
&\|J_n(t,x)\|_p^2 \lesssim  \left\| \int_0^t\int_{-\infty}^0p_{t-s}(x-w)^2z^{fs}_n(s,w)^2 \dd s \dd w\right\|_{p/2} \\&+\left\|\int_0^t \int_0^{\infty} \left(p_{t-s}(x+w)-(2u-1) \int_0^{\infty} p_{t-s}(x+w+z) e^{-(u-\frac{1}{2})z} \dd z\right)^2 z_n(s,w)^2\dd s \dd w \right\|_{p/2}\\& \leq \int_0^t\int_0^{\infty}p_{t-s}(x+w)^2\left\|z^{fs}_n(s,-w)\right\|_{p}^2 \dd s \dd w
\\&+ \int_0^t \int_0^{\infty} \left(p_{t-s}(x+w)-(2u-1) \int_0^{\infty} p_{t-s}(x+w+z) e^{-(u-\frac{1}{2})z} \dd z\right)^2 \left\|z_n(s,w)\right\|_{p}^2\dd s \dd w \\
& \lesssim \int_0^t \int_0^{\infty} p_{t-s}(x+w)^2 \left\|z_n(s,w)\right\|_{p}^2\dd s \dd w + \int_0^t\int_0^{\infty}p_{t-s}(x+w)^2\left\|z^{fs}_n(s,-w)\right\|_{p}^2 \dd s \dd w.
\end{aligned}
\]
Thus we have achieved an intermediate result
\begin{equation}\label{eq:interme}
\begin{aligned}
&\|z_{n+1}(t,x)  - z^{fs}_{n+1}(t,x) \|_p^2  \leq \left( \|I_n(t,x)\|_p+\|J_n(t,x)\|_p\right)^2\\
& \lesssim \int_0^t \int_0^{\infty}p_{t-s}(x-w)^2   \|z_{n}(s,w) -z^{fs}_n(s,w)\|^2_{p} \dd s \dd w\\
&+  \int_0^t \int_0^{\infty}p_{t-s}(x+w)^2 \|z_n(s,w)\|_p^2\dd s \dd w +  \int_0^t\int_0^{\infty}p_{t-s}(x+w)^2\|z^{fs}_n(s,-w)\|_p^2\dd s \dd w.
\end{aligned} 
\end{equation}
We first estimate the last two terms in the right-hand side of \eqref{eq:interme}.
By a similar iteration as in the proof of  Lemma~\ref{le:posMom}, for any $s\in[0,\tau], w\geq 0$, there exists a constant $C=C(\tau,p,u)$ such that {(with the convention of $(n/2)!=\Gamma(\tfrac{n}{2}+1)$)},
\[
\|z_n(s,w)\|_p^2 \leq \frac{(Cs)^{(n-1)/2}}{(n/2)!}  e^{Cw} \int_0^{\infty} \kernel^R_{u-\frac{1}{2}}(s,w\viv 0,y)  \left\| \exp(W_u(y)) \right\|_p\dd y \leq \frac{(Cs)^{(n-1)/2}}{(n/2)!}  e^{2Cw} .
\]
It then follows that
\[
\begin{aligned}
 &\int_0^t\int_0^{\infty}p_{t-s}(x+w)^2 \|z_n(s,w)\|_p^2\dd s \dd w  \leq \int_0^t\int_0^{\infty}p_{t-s}(x+w)^2  \frac{(Cs)^{(n-1)/2}}{(n/2)!} e^{2Cw}\dd s \dd w\\
 & \leq  \frac{(Ct)^{(n-1)/2}}{(n/2)!} \int_0^t p_{t-s}(x)^2\int_0^{\infty}e^{-\frac{w^2}{2(t-s)}}e^{2Cw}\dd w \dd s \lesssim  \frac{(Ct)^{(n-1)/2}}{(n/2)!}\int_0^tp_{t-s}(x)^2\sqrt{t-s}\dd s \\& \lesssim  \frac{(Ct)^{n/2}}{(n/2)!} \left[\sqrt{t}p_t(x)\right]^2.
 \end{aligned}
\]
A similar result can be obtained for $z^{fs}_n$. We record that, for $0<t\leq \tau$ and $x \geq 0$,
\[
 \int_0^t\int_0^{\infty}p_{t-s}(x+w)^2 \|z^{fs}_n(s,-w)\|_p^2\dd s \dd w \lesssim \frac{({C}t)^{n/2}}{(n/2)!} \left[\sqrt{t}p_t(x)\right]^2,
\]
for some positive constant ${C}= C(\tau,p,u)$. With these results, we could further bound \eqref{eq:interme} by
\[
\|z_{n+1}(t,x)  - z^{fs}_{n+1}(t,x) \|_p^2  \lesssim \frac{(Ct)^{n/2}}{(n/2)!} \left[\sqrt{t}p_t(x)\right]^2+\int_0^t\int_0^{\infty}p_{t-s}(x-w)^2   \|z_{n}(s,w) -z^{fs}_n(s,w)\|^2_{p} \dd s \dd w,
\]
with some constant $C=C(\tau,p,u)$. Using the estimate \eqref{eq:zeroterm}, we iterate the above inequality to derive 
\[
\|z_{n+1}(t,x)  - z^{fs}_{n+1}(t,x) \|_p^2 \leq C (n+2) \frac{(Ct)^{n/2}}{(n/2)!} \left[\sqrt{t}p_t(x)\right]^2,
\]
for some $C=C(\tau,p,u)$. Thus for any $t\in(0,\tau]$,
\[
\begin{aligned}
\|Z^{fs}_u(t,x) - {Z}_{u-\frac{1}{2}}(t,x) \|_p &\leq \sum_{k=0}^{\infty} \|z_{k}(t,x)  - z^{fs}_{k}(t,x) \|_p \leq \sum_{k=0}^{\infty} \left( (k+2) C\frac{(Ct)^{k/2}}{(k/2)!} \right)^{1/2}\sqrt{t}p_t(x).
\end{aligned}
\]
The proof is complete as the sum of the infinite series is bounded. 
\end{proof}

\section{Endpoint displacement: proof of Theorem~\ref{thm:var}}\label{se:uppbd}
In this section, we  prove \eqref{eq:intmidupbd}, which provides an upper bound on the first moment of the polymer endpoint displacement {in the bounded phase (when the boundary parameter $u<0$)}. Combining with the variance identity, this leads to the estimates on the fluctuations of the height function. 
The main contribution here is to leverage stochastic dominance to derive an upper bound of the annealed mean of the CDRP endpoint displacement in the bounded phase (when $u<0$). The upper bound is expected to be sharp since  it is conjectured to be the limit of the annealed mean when $t\to \infty$ (see Remark~\ref{re:conjconv}).

{The quantity of interest $\Expe\mathbb{E}^{\mathbb{Q}^{u,t}}_{0}[X_t]$ is the annealed average of the polymer endpoint. By our choice of the stationary initial data for KPZ, the polymer path $\{X_s\}_{s\in[0,t]}$, starting from $(t,0)$ and running backward in time, has a stationary terminal condition. Stochastic dominance ensures that the displacement of the endpoint in this case can be bounded from above by the endpoint of a different polymer path, one that starts from stationarity.  In other words, we use the polymer measure with both endpoints sampled from stationarity as a comparison. The upper bound in $\Expe\mathbb{E}^{\mathbb{Q}^{u,t}}_{0}[X_t] \leq \psi'(-2u)$ is derived from this stationarity-to-stationarity polymer measure.}


%

For any  $u\in \R, t>0, x\in [0,\infty)$, the half-space CDRP $\mathbb{Q}^{u,t}_{x}$ defined in Definition~\ref{de:cdrp} is a point-to-measure polymer with initial point $x$ and terminal data $\exp \left(W_{u}(x)\right)\dd x$. When $u< 0$, the terminal data is $\Prob$-almost surely in $L^1([0,\infty))$, thus we can also think of $\mathbb{Q}^{u,t}_{x}$ as a polymer with the normalized terminal data
\begin{equation}\label{eq:density1}
\frac{\exp \left(W_{u}(x)\right)\dd x}{\int_0^{\infty}\exp \left(W_{u}(x')\right)\dd x'}.
\end{equation}

Now let $\tilde{W}$ be another standard Brownian motion defined on  the probability space $(\Omega,\filt,\Prob)$, independent of $\xi$ and $W$. When $u<0$, for $\Prob$-almost surely all realization of $\tilde{W}$, we have $\int_0^{\infty}\exp \left(\tilde{W}_{u}(x)\right)\dd x<\infty$ and one can sample the initial point of the polymer $\mathbb{Q}^{u,t}_{x}$ from the density
\begin{equation}\label{eq:resample}
\frac{\exp \left(\tilde{W}_{u}(x)\right)\dd x}{\int_0^{\infty}\exp \left(\tilde{W}_{u}(x')\right)\dd x'}
\end{equation}
to construct another polymer $\mathbb{{Q}}^{u,t}_{\tilde{W}_u}$. By \eqref{eq:station} (and thus \eqref{eq:shestatdist}), the polymer $\mathbb{{Q}}^{u,t}_{\tilde{W}_u}$ has the special property that the start-point and end-point distributions of the paths are both stationary but independent of each other as well as the random environment. For $\Prob$-almost sure every realization of $\xi,W$ and $\tilde{W}$, the (quenched) endpoint density of this polymer is
\begin{equation}\label{eq:densitysam2x}
\begin{aligned}
\tilde{\rho}^{R}_{u}(y\viv t) := \frac{\int_0^{\infty}\mathcal{Z}_{u-\frac{1}{2}}(t,x\viv 0,y)\exp(\tilde{W}_{u}(x))\dd x\exp (W_{u}(y))}{\int_0^{\infty}\int_0^{\infty}\mathcal{Z}_{u-\frac{1}{2}}(t,x'\viv 0,y')\exp(\tilde{W}_{u}(x')) \dd x'\exp(W_{u}(y')) \dd y' }, \quad \forall y \in [0,\infty).
\end{aligned}\end{equation}

As discussed in Remark~\ref{re:whynocoupling}, since we do not show the simultaneous coupling of CDRP $\mathbb{Q}^{u,t}_{x}$ for all $x \in [0,\infty)$, it is not immediately clear that such a construction by sampling the start-point gives a  measure-to-measure CDRP $\mathbb{{Q}}^{u,t}_{\tilde{W}_u}$ on an event of  probability one. For our purpose, we do not need to check this. Throughout the proof, we only need the endpoint density formula \eqref{eq:densitysam2x} to be $\Prob$-almost surely well-defined, yet it is intuitive to understand this density from the ``polymer measure'' $\mathbb{{Q}}^{u,t}_{\tilde{W}_u}$. 

We verify that the endpoint density \eqref{eq:densitysam2x} is $\Prob$-almost surely well-defined.
By \eqref{eq:shestatdist}, the denominator in \eqref{eq:densitysam2x} multiplied by $Z_{u-\frac{1}{2}}(t,0)^{-1}$ is
\begin{equation}\label{eq:asfinite}\begin{aligned}
& Z_{u-\frac{1}{2}}(t,0)^{-1}\int_0^{\infty} {Z}_{u-\frac{1}{2}}(t,x)\exp(\tilde{W}_{u}(x)) \dd x \stackrel{\text { law }}{=} \int_0^{\infty} \exp(W_u(x)+\tilde{W}_{u}(x)) \dd x. \end{aligned}
\end{equation}
The right-hand side is a $\Prob$-almost surely positive and finite random variable when $u<0$. Since $Z_{u-\frac{1}{2}}(t,0)$ is also $\Prob$-almost surely positive and finite {by \eqref{eq:sheposfinite}}, the denominator in \eqref{eq:densitysam2x} is $\Prob$-almost surely positive and finite. The density $\tilde{\rho}^{R}_{u}(y\viv t) $ is thus well-defined.

Let us first briefly explain how to prove \eqref{eq:intmidupbd} through the density \eqref{eq:densitysam2x}. As we shall see below in Corollary~\ref{co:endptcompdom}, the polymer measures with different starting points $\{\mathbb{Q}_x^{u,t}\}_{x\in[0,\infty)}$ satisfies a stochastic monotonicity inherited from their path continuity and the planar structure of half-space. Intuitively, for any fixed environment noise $\xi$ and $W$, two independent polymer paths $X^1, X^2$ starting from points $x_1, x_2$ with $x_1<x_2$ should satisfy $\rho^1\leq_{st}\rho^2$, where $\leq_{st}$ refers to stochastic dominance and $\rho^i$ is the endpoint density. {Without an established strong Markov property for the half-space CDRP, we take a detour to prove this stochastic monotonicity through a Karlin-McGregor based argument. In particular, we analyze some determinants formed by four copies of Green's functions, working with a smoothed noise   to utilize the Feynman-Kac formula before passing to the limit. This argument gives  $\Prob$-almost sure comparison between products of Green's functions, which can be further used to prove the relation $\rho^1\leq_{st}\rho^2$.} 

We then compare the quenched densities $\rho_u^R(\cdot\viv t,0)$ and $\tilde{\rho}_u^R(\cdot\viv t)$ for $\Prob$-almost surely all $\omega \in \Omega$. We know that $\rho_u^R(\cdot\viv t,0)$ is the endpoint density of polymer paths starting from $x=0$, while $\tilde{\rho}_u^R(\cdot\viv t)$ is the endpoint density of polymer paths with starting points sampled from the density \eqref{eq:resample} on $[0,\infty)$. It shall follow that 
$\rho_u^R(\cdot\viv t,0)\leq_{st}\tilde{\rho}_u^R(\cdot\viv t)$ $\Prob$-almost surely, and as a corollary,
\[\Expe\mathbb{E}^{\mathbb{Q}^{u,t}}_{0}[X_t]=
\Expe \int_0^{\infty}y\rho_u^R(y\viv t,0) \dd y\leq \Expe\int_0^{\infty}y\tilde{\rho}_u^R(y\viv t)  \dd y = \Expe\mathbb{E}^{\mathbb{Q}^{u,t}}_{\tilde{W}_u}[X_t].\] On the other hand, since $\mathbb{{Q}}^{u,t}_{\tilde{W}_u}$ has the initial point sampled from the stationary density, we can  compute exactly $\Expe\mathbb{E}^{\mathbb{Q}^{u,t}}_{\tilde{W}_u}(X_t)$ for any $t> 0$. These two facts combined  together lead to the upper bound in \eqref{eq:intmidupbd}. Since $\exp \left(W_{u}(x)\right)\dd x$ is not in $L^1([0,\infty))$ when $u\geq 0$, the proof   in this section does not extend to $u \geq 0$.

{We now prove the aforementioned results in sequence. We begin by showing that some specific determinants consisting of four copies of Green's functions are $\Prob$-almost surely nonnegative. This argument is an adaption of  \cite[Proposition 5.5]{oconnell} to the half-space setting}.
\begin{lemma}\label{le:ptstdom} 
Fix $\mu \in \R$ and $t>0$. For any  $0\leq x_1<x_2$ and $0\leq y_1<y_2$, we have
\begin{equation}\label{eq:determinant}
    \det[\mathcal{Z}_{\mu}(t,x_i\viv 0, y_j) ]_{i,j \in \{1,2\}} \geq 0 \quad \Prob\text{-almost surely}.
\end{equation}
\end{lemma}
\begin{proof}
    We will use an argument based on the Karlin-McGregor formula. We first prove the inequality \eqref{eq:determinant} for the Green's functions of half-space SHE with a smoothed noise, using a  Feynman-Kac representation. We then pass to the limit to obtain the same result for the  white noise   using Lemma~\ref{le:fk}.

Let $\xi_{\delta}(t,x) = \int_\R  p_\delta(x-y)\xi(t, \dd y)$ so that it is a Gaussian noise that is white in time and smooth in space, with the covariance function
\[
Q_{\delta}(t_1,t_2,x_1,x_2) = \Expe \xi_\delta(t_1,x_1)\xi_\delta(t_2,x_2)=\delta_0(t_1-t_2) p_{2\delta}(x_1-x_2).
\]
For any $\delta>0$, $x,y \in [0,\infty)$, define \[
    \mathcal{Z}_\mu^{\delta}(t,x\viv 0, y) := \kernel^{N}(t,x\viv 0,y)
\mathbb{E}^x_B\left[\exp\left(-\mu L^{0}_{t} +
\int_{0}^{t}\xi_{\delta}(t-r,|B_r|)\dd r- \frac{t}{2}p_{2\delta}(0)\right)\bigg||B_{t}|=y\right].\]
{As in \eqref{eq:probrep3}, we rewrite the right-hand side in terms of the ``reflected Brownian bridges''. Let $\Upsilon_t$ be the reflected Brownian motion with $\Upsilon_t := |B_t|$ starting from $\Upsilon_0=|B_0|=x$. For any $t>0, y \geq 0$, there exists a unique probability measure $\mathbb{P}^{x,y}_{\Upsilon}$ on $\mathcal{C}([0,t],[0,\infty))$ such that $\mathbb{P}^{x,y}_{\Upsilon}$ equals to the conditional probability distribution  $\mathbb{P}^{x}_{\Upsilon}(\cdot \viv \Upsilon_t=y)$. We call the canonical process associated with $\mathbb{P}^{x,y}_{\Upsilon}$  the ``reflected Brownian bridge'', 
starting from $\Upsilon_0 =x$ and ending at $\Upsilon_t =y$.} 
The above display equals to
\[ 
\kernel^{N}(t,x\viv 0,y)
\mathbb{E}^{x,y}_{\Upsilon}\left[\exp\left(-\mu L^{0,{\Upsilon}}_{t} +
\int_{0}^{t}\xi_{\delta}(t-r,\Upsilon_r)\dd r- \frac{t}{2}p_{2\delta}(0)\right)\right],
\]
where 
we used $L^{0,{\Upsilon}}_{t}$ to denote the local time  at zero:
\begin{equation}\label{eq:ltrBM}
L^{0,{\Upsilon}}_{t} := \lim _{\eps \to 0} \frac{1}{2 \varepsilon} \int_0^t \1_{[0, \varepsilon]}\left(\Upsilon_s\right) \dd s.
\end{equation}
For {discussions on }reflected Brownian bridges, we refer to {\cite{{fpyMarkovianBridge}} in which the authors discussed general Markovian bridges,} and also \cite{pitman2001distribution, aldous2006two} and the references therein.

To ease notations, we define
\[
F^\delta_{t}(\Upsilon) := \exp\left(-\mu L^{0,{\Upsilon}}_{t} +
\int_{0}^{t}\xi_{\delta}(t-r,\Upsilon_r)\dd r- \frac{t}{2}p_{2\delta}(0)\right).
\]
Then $F^\delta_{t}(\cdot)$ is an $\Prob$-almost surely continuous, strictly positive and multiplicative functional of $\Upsilon$. 

Next we prove that for any $\delta>0$, $0\leq x_1<x_2$ and $0\leq y_1<y_2$,   the following identity holds $\Prob$-almost surely:
\begin{equation}\label{eq:KMid}
    \det[\mathcal{Z}^\delta_{\mu}(t,x_i\viv 0, y_j) ]_{i,j \in \{1,2\}} = \det[\kernel^N(t,x_i\viv 0, y_j)]_{i,j \in \{1,2\}} \mathbb{E}^{\mathbf{x},\mathbf{y}}_{\Xi^{1},\Xi^{2}} [F^\delta_{t}(\Xi^{1})F^\delta_{t}(\Xi^{2})],
\end{equation}
where the expectation $\mathbb{E}^{\mathbf{x},\mathbf{y}}_{\Xi^{1},\Xi^{2}}$ is taken over the  two-dimensional reflected Brownian bridges $(\Xi^{1}, \Xi^{2})$ started from $\mathbf{x}=(x_1,x_2)$, ending at $\mathbf{y}=(y_1,y_2)$ at time $t$, and   $\Xi^{1}, \Xi^{2}$ being non-intersecting on $[0,t]$. 

To prove \eqref{eq:KMid}, by the continuity of the Green's function, it suffices to show that for arbitrary functions $\varphi_1,\varphi_2 \in \mathcal{C}_c([0,\infty),\R)$ with coordinate-wise ordered supports (which means that for any $y_1\in \supp~\varphi_1$ and $y_2 \in \supp~\varphi_2$, we have $y_1< y_2$),  there is the identity
\begin{equation}\label{eq:KMidvar}
\begin{aligned}
     &\int_0^\infty \int_0^\infty \varphi_1(y_1) \varphi_2(y_2) \det[\mathcal{Z}_{\mu}(t,x_i\viv 0, y_j) ]_{i,j \in \{1,2\}} \dd y_1 \dd y_2 \\
      = & \int_0^\infty \int_0^\infty \varphi_1(y_1) \varphi_2(y_2) 
 \det[\kernel^N(t,x_i\viv 0, y_j)]_{i,j \in \{1,2\}} \mathbb{E}^{\mathbf{x},\mathbf{y}}_{\Xi^{1},\Xi^{2}} [F^\delta_{t}(\Xi^{1})F^\delta_{t}(\Xi^{2})] \dd y_1 \dd y_2.
\end{aligned}
\end{equation}

By Karlin-McGregor formula \cite{KMid}, the determinant $\det[\kernel^N(t,x_i\viv 0, y_j)]_{i,j \in \{1,2\}} $ is the transition density at time $t$ of a two-dimensional reflected Brownian motion $(\Upsilon^{1}, \Upsilon^{2})$  in the domain $\Lambda=\{\mathbf{y}=(y_1,y_2)\in \R_{\geq 0}^2: y_1\leq y_2\}$ killed when $\Upsilon^{1}$ and $ \Upsilon^{2}$ first intersect. It then follows that the right-hand side of \eqref{eq:KMidvar} equals to
\[ \mathbb{E}^{\mathbf{x}}_{\Upsilon^{1},\Upsilon^{2}} [\varphi_1({\Upsilon}_t^{1})\varphi_2({\Upsilon}_t^{2}) F^\delta_{t}(\Upsilon^{1})F^\delta_{t}(\Upsilon^{2}) \1(\tau > t)], \]
where the expectation $\mathbb{E}^{\mathbf{x}}_{\Upsilon^{1},\Upsilon^{2}}$ is taken over the two-dimensional reflected Brownian motions $(\Upsilon^{1}, \Upsilon^{2})$ in $\R_{\geq 0}^2$ started at $\mathbf{x}=(x_1,x_2)$,
and $\tau$ is the stopping time taking values in $[0,t]\cup \{\infty\}$ defined by \[
\tau=\inf \left\{r \in[0, t]: \Upsilon_r^{1}=\Upsilon_r^{2}\right\}.\]

On the other hand, by expanding the determinant $\det[\mathcal{Z}_{\mu}(t,x_i\viv 0, y_j) ]_{i,j \in \{1,2\}}$,  the left-hand side of \eqref{eq:KMidvar} equals to
\begin{align*}
    & \mathbb{E}^{\mathbf{x}}_{\Upsilon^{1},\Upsilon^{2}} [\varphi_1({\Upsilon}_t^{1})\varphi_2({\Upsilon}_t^{2}) F^\delta_{t}(\Upsilon^{1})F^\delta_{t}(\Upsilon^{2})]
     -  \mathbb{E}^{\mathbf{x}}_{\Upsilon^{1},\Upsilon^{2}} [\varphi_2({\Upsilon}_t^{1})\varphi_1({\Upsilon}_t^{2}) F^\delta_{t}(\Upsilon^{1})F^\delta_{t}(\Upsilon^{2})].
\end{align*}
Recall that the reflected Brownian motions $\Upsilon^{1}, \Upsilon^{2}$ started from ordered initial points $x_1<x_2$, and the support of $\varphi_1,\varphi_2$ are also ordered. If $\tau >t$, the paths of $\Upsilon^{1}, \Upsilon^{2}$ do not cross each other on $[0,t]$. Since the half-space domain is planar, the path continuity of the reflected Brownian motions forces that
\[
\mathbb{E}^{\mathbf{x}}_{\Upsilon^{1},\Upsilon^{2}} [\varphi_2({\Upsilon}_t^{1}) \varphi_1({\Upsilon}_t^{2})F^\delta_{t}(\Upsilon^{1})F^\delta_{t}(\Upsilon^{2})\1{(\tau >t)}]=0.
\]
Combining the above results, the identity \eqref{eq:KMidvar} is now equivalent to
\begin{equation}\label{eq:finalkm}
\begin{aligned}
  &\mathbb{E}^{\mathbf{x}}_{\Upsilon^{1},\Upsilon^{2}} [\varphi_1({\Upsilon}_t^{1})\varphi_2({\Upsilon}_t^{2}) F^\delta_{t}(\Upsilon^{1})F^\delta_{t}(\Upsilon^{2})\1{(\tau \leq t)}]\\  
&-\mathbb{E}^{\mathbf{x}}_{\Upsilon^{1},\Upsilon^{2}} [\varphi_2({\Upsilon}_t^{1})\varphi_1({\Upsilon}_t^{2}) F^\delta_{t}(\Upsilon^{1})F^\delta_{t}(\Upsilon^{2})\1{(\tau \leq t)}]=0.
\end{aligned}
\end{equation}

We prove \eqref{eq:finalkm} by using a path-switching argument. Define a new two-dimensional reflected Brownian motion on $r\in[0,t]$ by \[
(\tilde{\Upsilon}^{1}_r, \tilde{\Upsilon}^{2}_r) = \begin{cases}({\Upsilon}^{1}_r,  {\Upsilon}^{2}_r)  & r \leq \tau \\ ( {\Upsilon}^{2}_r,  {\Upsilon}^{1}_r) & r \geq \tau. \end{cases}
\]
The paths of $\tilde{\Upsilon}^{1}, \tilde{\Upsilon}^{2}$ intersect on $[0,t]$ if and only if the paths of ${\Upsilon}^{1}, {\Upsilon}^{2}$ intersect on $[0,t]$.
Since the functional $F^\delta_{t}(\cdot)$ is multiplicative, for each realization of $({\Upsilon}^{1},  {\Upsilon}^{2})$, \[F^\delta_{t}(\tilde{\Upsilon}^{1})F^\delta_{t}(\tilde{\Upsilon}^{2}) = F^\delta_{t}({\Upsilon}^{1})F^\delta_{t}({\Upsilon}^{2}).\] On the event of $\{\tau\leq t\}$, $\varphi_1(\tilde{\Upsilon}_t^{1})\varphi_2(\tilde{\Upsilon}_t^{2}) =\varphi_1({\Upsilon}_t^{2})\varphi_2({\Upsilon}_t^{1})$. It follows that
\[
\begin{aligned}
&\mathbb{E}^{\mathbf{x}}_{\tilde{\Upsilon}^{1},\tilde{\Upsilon}^{2}} [\varphi_1(\tilde{\Upsilon}_t^{1})\varphi_2(\tilde{\Upsilon}_t^{2}) F^\delta_{t}(\tilde{\Upsilon}^{1})F^\delta_{t}(\tilde{\Upsilon}^{2})\1{(\tau \leq t)}]\\
& =\mathbb{E}^{\mathbf{x}}_{{\Upsilon}^{1},{\Upsilon}^{2}} [\varphi_1({\Upsilon}_t^{2})\varphi_2({\Upsilon}_t^{1}) F^\delta_{t}({\Upsilon}^{1})F^\delta_{t}({\Upsilon}^{2})\1{(\tau \leq t)}].
\end{aligned}
\]
Moreover, as the reflected Brownian motion is strong Markov, $(\tilde{\Upsilon}^{1}, \tilde{\Upsilon}^{2})$ has the same law as $({\Upsilon}^{1},  {\Upsilon}^{2})$, which implies that
\[\begin{aligned}
   & \mathbb{E}^{\mathbf{x}}_{\Upsilon^{1},\Upsilon^{2}} [\varphi_1({\Upsilon}_t^{1})\varphi_2({\Upsilon}_t^{2}) F^\delta_{t}(\Upsilon^{1})F^\delta_{t}(\Upsilon^{2})\1{(\tau \leq t)}]\\
&= \mathbb{E}^{\mathbf{x}}_{\tilde{\Upsilon}^{1},\tilde{\Upsilon}^{2}} [\varphi_1(\tilde{\Upsilon}_t^{1})\varphi_2(\tilde{\Upsilon}_t^{2}) F^\delta_{t}(\tilde{\Upsilon}^{1})F^\delta_{t}(\tilde{\Upsilon}^{2})\1{(\tau \leq t)}].
\end{aligned}
\]
The above two identities prove
\eqref{eq:finalkm}.

Now we have proved \eqref{eq:KMidvar} for any $\varphi_1,\varphi_2 \in \mathcal{C}_c([0,\infty),\R)$ with coordinate-wise ordered supports. The identity \eqref{eq:KMid} follows. When $0\leq x_1<x_2$ and $0\leq y_1<y_2$, $\det[\kernel^N(t,x_i\viv 0, y_j)]_{i,j \in \{1,2\}} $ is strictly positive. Thus for any $\delta>0$, $0\leq x_1<x_2$ and $0\leq y_1<y_2$, we have
\[\det[\mathcal{Z}^\delta_{\mu}(t,x_i\viv 0, y_j) ]_{i,j \in \{1,2\}}>0  \quad \Prob\text{-almost surely}.\]
By Lemma~\ref{le:fk}~\eqref{it:approx} (with a mollification in space only), for any $t>0, x,y \in [0,\infty)$, $p \in [1,\infty)$, \[\mathcal{Z}^\delta_{\mu}(t,x\viv 0, y) \to \mathcal{Z}_{\mu}(t,x\viv 0, y)\quad \text{ in } L^p(\Omega) \quad \text{as } \delta \to 0.\]
The result \eqref{eq:determinant} is thus proved.
\end{proof}

With Lemma~\ref{le:ptstdom}, we are ready to  prove the stochastic monotonicity.
\begin{corollary}\label{co:endptcompdom}
For any $u\in \R$, $t>0$, $0\leq x_1<x_2$,  there exists an event $\Omega_0$ with $\Prob(\Omega_0)=1$ such that for all $\omega \in \Omega_0$ and $a>0$,
\begin{equation}\label{eq:comp2dens}
\mathbb{Q}^{u,t}_{x_1}(X_t \in [0,a]) = \int_0^a \rho^R_u(y \viv t,x_1) \dd y \geq  \int_0^a \rho^R_u(y \viv t,x_2) \dd y =\mathbb{Q}^{u,t}_{x_2}(X_t \in [0,a]).
\end{equation}
As a result, for any fixed $u<0$, $t>0$,
\begin{equation}\label{eq:endptcompdom}\mathbb{E}^{\mathbb{Q}^{u,t}}_{0}[X_t]=
\int_0^{\infty} y \rho^R_u(y \viv t,0) \dd y \leq
 \int_0^{\infty}y \tilde{\rho}^{R}_{u}(y\viv t) \dd y, \quad \Prob\text{-almost surely},
\end{equation}
where $\tilde{\rho}^{R}_{u}(\cdot \viv t)$ is as defined in \eqref{eq:resample}.
\end{corollary}
\begin{proof}
By Lemma~\ref{le:ptstdom}, for any fixed $u\in \R$, $t>0$, $0\leq x_1<x_2$, with $\Prob$ probability one, 
 \[
\mathcal{Z}_{u-\frac{1}{2}}(t,x_1\viv 0, y_1)\mathcal{Z}_{u-\frac{1}{2}}(t,x_2\viv 0, y_2)\geq \mathcal{Z}_{u-\frac{1}{2}}(t,x_1\viv 0, y_2)\mathcal{Z}_{u-\frac{1}{2}}(t,x_2\viv 0, y_1),
\]
for all $0\leq y_1<y_2$ simultaneously (one can first consider rational $y_1,y_2$ then use the continuity of $\mathcal{Z}$). By \eqref{eq:sheposfinite}, $\{\mathcal{Z}_{u-\frac{1}{2}}(t,x_i\viv 0, y)\}_{i=1,2}$ are $\Prob$-almost surely integrable with respect to the density \eqref{eq:density1}. Therefore, there exists an event $\Omega_0$ with $\Prob(\Omega_0)=1$ such that for any $\omega \in \Omega_0$, we can integrate over any rectangles $(y_1,y_2)\in [0,a)\times [a,\infty)$ with $a>0$ to obtain
\begin{align*}
&\int_0^a\mathcal{Z}_{u-\frac{1}{2}}(t,x_1\viv 0, y_1) \exp (W_{u}(y_1))\dd y_1 \int_a^\infty \mathcal{Z}_{u-\frac{1}{2}}(t,x_2\viv 0, y_2) \exp (W_{u}(y_2))\dd y_2\\
&\geq \int_a^\infty\mathcal{Z}_{u-\frac{1}{2}}(t,x_1\viv 0, y_2) \exp (W_{u}(y_2))\dd y_2 \int_0^a \mathcal{Z}_{u-\frac{1}{2}}(t,x_2\viv 0, y_1) \exp (W_{u}(y_1))\dd y_1,
\end{align*}
and both sides are positive and finite. Add 
\[
\int_0^a\mathcal{Z}_{u-\frac{1}{2}}(t,x_1\viv 0, y) \exp (W_{u}(y))\dd y
\int_0^a\mathcal{Z}_{u-\frac{1}{2}}(t,x_2\viv 0, y) \exp (W_{u}(y))\dd y
\]
to both sides and then divide by the product $Z_{u-\frac{1}{2}}(t,x_1)Z_{u-\frac{1}{2}}(t,x_2)$, we obtain  \eqref{eq:comp2dens}.

To prove \eqref{eq:endptcompdom}, by \eqref{eq:comp2dens}, for any fixed $x>0$,
\[\begin{aligned}&\int_0^{\infty} y \rho^R_u(y \viv t,0) \dd y =\frac{\int_0^{\infty}y\mathcal{Z}_{u-\frac{1}{2}}(t,0\viv 0,y)\exp (W_{u}(y))\dd y}{\int_0^{\infty}\mathcal{Z}_{u-\frac{1}{2}}(t,0\viv 0,y')\exp(W_{u}(y')) \dd y'} \\& \leq \int_0^{\infty} y \rho^R_u(y \viv t,x) \dd y = \frac{\int_0^{\infty}y\mathcal{Z}_{u-\frac{1}{2}}(t,x\viv 0,y)\exp (W_{u}(y))\dd y}{\int_0^{\infty}\mathcal{Z}_{u-\frac{1}{2}}(t,x\viv 0,y')\exp(W_{u}(y')) \dd y'} \quad \Prob\text{-almost surely.}
\end{aligned}\]
Multiplying the positive denominators on both sides, this is equivalent to the inequality
\begin{equation}\label{eq:compwoCoup}\begin{aligned}
&{\int_0^{\infty}y\mathcal{Z}_{u-\frac{1}{2}}(t,0\viv 0,y)\exp (W_{u}(y))\dd y}{\int_0^{\infty}\mathcal{Z}_{u-\frac{1}{2}}(t,x\viv 0,y')\exp(W_{u}(y')) \dd y'}\\
& \leq {\int_0^{\infty}\mathcal{Z}_{u-\frac{1}{2}}(t,0\viv 0,y')\exp(W_{u}(y')) \dd y'}{\int_0^{\infty}y\mathcal{Z}_{u-\frac{1}{2}}(t,x\viv 0,y)\exp (W_{u}(y))\dd y}.
 \end{aligned}\end{equation}
 By Proposition~\ref{pr:greens}~\eqref{it:convolution}, both sides are continuous in $x\in[0,\infty)$ $\Prob$-almost surely as the integrals are mild solutions to SHE. It follows that \eqref{eq:compwoCoup} holds for all $x\in[0,\infty)$ simultaneously. 
 
When $u<0$, by \eqref{eq:asfinite}, the left-hand side of \eqref{eq:compwoCoup} is integrable in $x$ with respect to the density \eqref{eq:resample} $\Prob$-almost surely. Similar integrability holds for the right-hand side.
Thus with  probability one, we have an inequality 
\[\begin{aligned}
&{\int_0^{\infty}y\mathcal{Z}_{u-\frac{1}{2}}(t,0\viv 0,y)\exp (W_{u}(y))\dd y}{\int_0^{\infty} \int_0^{\infty}\mathcal{Z}_{u-\frac{1}{2}}(t,x\viv 0,y')\exp(W_{u}(y')) \dd y'} \exp(\tilde{W}_u(x))\dd x
\\
& \leq {\int_0^{\infty}\mathcal{Z}_{u-\frac{1}{2}}(t,0\viv 0,y')\exp(W_{u}(y')) \dd y'}{\int_0^{\infty}\int_0^{\infty}y\mathcal{Z}_{u-\frac{1}{2}}(t,x\viv 0,y)\exp (W_{u}(y))\dd y}\exp(\tilde{W}_u(x))\dd x,
 \end{aligned}
\] with all the integrals being positive and finite. This is equivalent to
\[\frac{\int_0^{\infty}y\mathcal{Z}_{u-\frac{1}{2}}(t,0\viv 0,y)\exp (W_{u}(y))\dd y}{\int_0^{\infty}\mathcal{Z}_{u-\frac{1}{2}}(t,0\viv 0,y')\exp(W_{u}(y')) \dd y'} \leq \frac{\int_0^{\infty} \int_0^{\infty}y\mathcal{Z}_{u-\frac{1}{2}}(t,x\viv 0,y)\exp (W_{u}(y))\exp(\tilde{W}_u(x)) \dd y \dd x}{\int_0^{\infty} \int_0^{\infty}\mathcal{Z}_{u-\frac{1}{2}}(t,x'\viv 0,y')\exp(W_{u}(y'))\exp(\tilde{W}_u(x')) \dd y' \dd x'},
\]
and the proof is complete.
\end{proof}

\begin{remark}
We would like to point out that in \cite{alberts2022green}, the authors proved a strict inequality `$>$' for \eqref{eq:determinant} after coupling the full space CDRP. Another proof using a different argument has been conducted in \cite{hang2020continuity}. We do not pursue it here.

Since we only need an inequality with `$\leq$', an alternative approach is to prove the strong Markov property for polymer measures $\mathbb{Q}_x^{u,t}$ and then use the coupling method. For continuum time models, proving the strong Markov property requires both the Markov and the Feller properties of the associated semigroup \cite[Theorem 6.17]{le2016brownian}, which we believe can also be done for this model. 
\end{remark}

To complete the proof of Theorem~\ref{thm:var} and Corollary~\ref{co:tightnessleft}, it remains to compute the annealed mean of the stationary measure-to-measure polymer's endpoint $\Expe\mathbb{E}^{\mathbb{Q}^{u,t}}_{\tilde{W}_u}(X_t)$.

\begin{lemma}\label{le:compendpt} For any $u<0$ and $t>0$, we have
\begin{equation}\label{eq:twosideddensity}
\Expe \int_0^{\infty}y \tilde{\rho}^{R}_{u}(y\viv t) \dd y = \psi'(2|u|).
\end{equation}
The functions $\psi$ and $\psi'$ are the digamma and trigamma functions defined in \eqref{eq:digammafunc}--\eqref{eq:trigammafunc}.
\end{lemma}
\begin{proof}
By definition \eqref{eq:densitysam2x},\[
\begin{aligned}
\Expe \int_0^{\infty} y \tilde{\rho}^{R}_{u}(y\viv t) \dd y&=\Expe \frac{\int_0^{\infty}y\int_0^{\infty}\mathcal{Z}_{u-\frac{1}{2}}(t,x\viv 0,y)e^{\tilde{W}_{u}(x)}\dd xe^{W_{u}(y)}\dd y}{\int_0^{\infty}\int_0^{\infty}\mathcal{Z}_{u-\frac{1}{2}}(t,x'\viv 0,y')e^{\tilde{W}_{u}(x')} \dd x'e^{W_{u}(y')} \dd y' }\\
& = \Expe \frac{\int_0^{\infty}y\int_0^{\infty}\mathcal{Z}_{u-\frac{1}{2}}(t,y\viv 0,x)e^{\tilde{W}_{u}(x)}\dd xe^{W_{u}(y)}\dd y}{\int_0^{\infty}\int_0^{\infty}\mathcal{Z}_{u-\frac{1}{2}}(t,y'\viv 0,x')e^{\tilde{W}_{u}(x')} \dd x'e^{W_{u}(y')} \dd y' },\end{aligned}\]
where the second equality follows from the time reversal of the field $\mathcal{Z}$ in \eqref{eq:trev}. Since $W$ and $\tilde{W}$ are independent and identical, we can interchange them under the expectation, and the above equals to \[
\Expe \frac{\int_0^{\infty}y Z_{u-\frac{1}{2}}(t,y) e^{\tilde{W}_{u}(y)}\dd y}{\int_0^{\infty}Z_{u-\frac{1}{2}}(t,y') e^{\tilde{W}_{u}(y')} \dd y'}= \Expe \frac{\int_0^{\infty}y Z_{u-\frac{1}{2}}(t,y)Z_{u-\frac{1}{2}}(t,0)^{-1} e^{\tilde{W}_{u}(y)}\dd y}{\int_0^{\infty}Z_{u-\frac{1}{2}}(t,y')Z_{u-\frac{1}{2}}(t,0)^{-1} e^{\tilde{W}_{u}(y')} \dd y' } = \Expe \frac{\int_0^{\infty}y  e^{W_u(y)}e^{\tilde{W}_{u}(y)} \dd y}{\int_0^{\infty}  e^{W_u(y')} e^{\tilde{W}_{u}(y')}\dd y' },
\] where we used the invariance \eqref{eq:shestatdist} and our assumption that $W,\tilde{W}, \xi$ are all independent. The last expression is time independent and can be computed explicitly. In fact, since $W(y)+\tilde{W}(y)$ has the same distribution as $W(2y)$, the above expectation equals to
\begin{equation}\label{eq:midptmean}\begin{aligned}
& \Expe \frac{\int_0^{\infty}y \exp(W(2y)+2uy) \dd y}{\int_0^{\infty} \exp(W(2y')+2uy') \dd y' } = \frac{1}{2} \Expe \frac{\int_0^{\infty}y \exp(W(y)+uy) \dd y}{\int_0^{\infty} \exp(W(y')+uy') \dd y' } = \frac{1}{2} \Expe \frac{\dd}{\dd v} \log \int_0^{\infty} e^{W(y)+vy} \dd y\bigg\mid_{v=u}\\
 & = \frac{1}{2} \frac{\dd}{\dd v} \Expe \log \int_0^{\infty} e^{W(y)+vy} \dd y\bigg\mid_{v=u} = \frac{1}{2} \frac{\dd}{\dd v} \left(-\psi(-2v)+\log{2} \right)\bigg\mid_{v=u} = \psi'(-2u).
\end{aligned}\end{equation}
 In the second from last equality, we have used \eqref{eq:expoBM} together with the fact that $\Expe \left[\log(1/\gamma(\theta))\right]=-\psi(\theta)$ for any $\theta>0$. To check the interchange of differentiation and expectation, we first use Fatou's lemma to derive that $\Expe \frac{\int_0^{\infty}y \exp(W(y)+uy) \dd y}{\int_0^{\infty} \exp(W(y')+uy') \dd y' } \leq 2\psi'(-2u)$, then the interchange can be justified by applying the dominated convergence theorem and the mean value theorem.
\end{proof}
\begin{remark}
With the convergence conjecture mentioned in Remark~\ref{re:conjconv}, \eqref{eq:twosideddensity} should also equal to the annealed mean of the midpoint position of long half-space polymers (i.e. as $t \to \infty$) with both endpoints fixed
near the wall at times $0$ and $2t$ in the bound phase $u<0$. For a discussion on such midpoint distributions, see \cite[Section IV B]{barraquand2021kardar}.
\end{remark}





\section{Symmetry: proof of Proposition~\ref{pr:symmetry}}\label{se:symmetry}
The goal of this section is to prove the symmetry identity in Proposition~\ref{pr:symmetry}. This symmetry would allow us to study the height function statistics and polymer endpoint displacement for all values of $u>0$, using the results we obtained in the case of $u<0$. It follows from results in \cite{barraquand2020half} regarding the half-space Macdonald processes. We learned this symmetry from \cite[Claim 4.9]{barraquand2020halfkpz}, where the authors  stated a proof of convergence heuristically. This convergence was later rigorously proved in \cite{barraquand2023stationary}, so we now summarize these results to provide a proof. 

For this section, we use uppercase letters (e.g. $S, T,X,Y$) for continuum variables and lowercase letters (e.g. $r,s,t,x,y,w$) for discrete variables. We use $\mathbf{Q}$   to denote the set of rational numbers, and $\mathbf{Z}$ to denote the set of integers.

We use the inhomogeneous half-space log-gamma (HSLG) polymer models as defined in \cite[Definition 2.1]{barraquand2023stationary}. Following the notations there, let $\alpha_{\circ}, \alpha_1, \alpha_2, \ldots$ be real parameters such that $\alpha_i+\alpha_{\circ}>0$ for all $i \geq 1$ and $\alpha_i+\alpha_j>0$ for all $i \neq j \geq 1$. Let $\left(\varpi_{i, j}\right)_{i \geqslant j}$ be a family of independent random variables such that for $i>j, \varpi_{i, j} \sim 1/\gamma\left(\alpha_i+\alpha_j\right)$ and $\varpi_{i, i} \sim 1/\gamma\left(\alpha_{\circ}+\alpha_i\right)$. The partition function of the half-space log-gamma polymer is defined as \begin{equation}\label{eq:hsloggam}
z(n, m)=\sum_{\pi:(1,1) \rightarrow(n, m)} \prod_{(i, j) \in \pi} \varpi_{i, j},
\end{equation}
where the sum is over all up-right paths from $(1,1)$ to $(n, m)$ in the octant $\left\{(i, j) \in \mathbf{Z}_{>0}^2: i \geqslant j\right\}$.

Derived from a symmetry result of the half-space Macdonald process, the partition functions $z(m,m)$ for any $m\geq 1$ as defined in \eqref{eq:hsloggam} will not change their laws when the parameters $\alpha_{\circ}$ and $\alpha_1$ are interchanged. 
This is a generalization of \cite[Proposition 8.1]{barraquand2020half} for inhomogeneous HSLG polymer.

\begin{proposition}\label{pr:loggammasymmetry}
For any $m \geq 1$, the law of the partition function $z(m,m)$ with parameters $\alpha_{\circ},\alpha_1,\alpha_2, ..., \alpha_m$ equals to the law of the partition function $\tilde{z}(m,m)$ with parameters $\alpha_{1},\alpha_{\circ},\alpha_2, ..., \alpha_m$.
\end{proposition}
\begin{proof}
The law of the partition function $z(m,m)$ with parameters $\alpha_{\circ},\alpha_1, ..., \alpha_m$ equals to the limiting law of  $(1-q)^{2m-1} q^{-\lambda_1}$ as $q\to 1$, where $\lambda_1$ is distributed according to the q-Whittaker measure \cite{o2014geometric, barraquand2020half} with parameters $\{(q^{\alpha_1}, ..., q^{\alpha_m}),  q^{\alpha_{\circ}}\}$. One can think of $(q^{\alpha_1}, ..., q^{\alpha_m})$ as the bulk parameters, and $q^{\alpha_{\circ}}$ as the diagonal parameter. By \cite[Proposition 2.6]{barraquand2020half},  $\lambda_1$ has the same distribution as $\pi_1$, where $\pi_1$ is distributed to the q-Whittaker measure with parameters $\{(q^{\alpha_1}, ..., q^{\alpha_m} , q^{\alpha_{\circ}}),0 \}$, where $(q^{\alpha_1}, ..., q^{\alpha_m} , q^{\alpha_{\circ}})$ is the bulk parameter and $0$ is the diagonal parameter. Since the q-Whittaker measure is invariant under permutation of its bulk parameters (see \cite[Lemma 2.7]{barraquand2023stationary} for the proof and further applications of this symmetry), $\pi_1$ equals in law to $\tilde{\pi}_1$, which is defined to be distributed as the q-Whittaker measure with parameters $\{(q^{\alpha_{\circ}},q^{\alpha_2}, ..., q^{\alpha_m},  q^{\alpha_1}),0\}$. Apply \cite[Proposition 2.6]{barraquand2020half} again backward, we have $\tilde{\pi}_1$ has the same distribution as $\tilde{\lambda}_1$, which is distributed according to the q-Whittaker measure with parameters $\{(q^{\alpha_{\circ}},q^{\alpha_2}, ..., q^{\alpha_m}),  q^{\alpha_1}\}$. Thus we have \[(1-q)^{2m-1} q^{-\lambda_1} \stackrel{\text{law}}{=}  (1-q)^{2m-1} q^{-\tilde{\lambda}_1}.\] Taking the limit $q \to 1$ as in \cite[Proposition 8.1]{barraquand2020half}, we have the partition function $z(m,m)$ with parameters $\alpha_{\circ},\alpha_1, ..., \alpha_m$ equals in law to the partition function $\tilde{z}(m,m)$ with parameters $\alpha_1,\alpha_{\circ},\alpha_2, ..., \alpha_m$.
\end{proof}

From now on, we assume $u-v>0$. We set $\alpha_\circ=u$ and $\alpha_1=-v$, and for all $i\geq 2$, $\alpha_i=\alpha$ for some $\alpha>\min{(0,-u,v)}$. We use $z_{u,-v}^{\alpha}(n,m):= z(n,m)$ to emphasize the dependence on parameters $u,v$ and $\alpha$. 
We define the ratio \[z_{u,-v}^{\text{stat},\alpha}(m,m) :=\frac{z_{u,-v}^{\alpha}(m,m)}{\varpi_{1,1}}.\]
\begin{corollary}\label{cr:ratioequalinlaw}
   {Assume $u>v$.} For any fixed $m\geq 1$, \begin{equation}\label{eq:ratioidinlaw}
z_{u,-v}^{\text{stat},\alpha}(m,m) \stackrel{\text{law}}{=} z_{-v,u}^{\text{stat},\alpha}(m,m).
\end{equation}
\end{corollary}
\begin{proof}
Proposition~\ref{pr:loggammasymmetry} says that for any $m \geq 1$,
$z_{u,-v}^{\alpha}(m,m) \stackrel{\text{law}}{=} z_{-v,u}^{\alpha}(m,m)$. Since $\left(\varpi_{i, j}\right)_{i \geqslant j}$ are independent and $\varpi_{1,1}$ appears in the product in all of the terms in \eqref{eq:hsloggam},   the ratio $z_{u,-v}^{\text{stat},\alpha}(m,m)$ is independent of ${\varpi_{1,1}}$ and we have
$
\log z_{u,-v}^{\text{stat},\alpha}(m,m)  + \log \varpi_{1,1} \stackrel{\text{law}}{=} \log z_{-v,u}^{\text{stat},\alpha}(m,m)+ \log \varpi_{1,1}$. Since {$\alpha_0+\alpha_1=u-v>0$} and the characteristic function of a log-gamma random variable is always nonzero, the independence implies \eqref{eq:ratioidinlaw}.
\end{proof}

We next take the weak noise scaling limit to obtain the result for the half-space KPZ equation, based on  \cite[Theorem 5.4]{barraquand2023stationary} for the above half-space HSLG models. 
\begin{proposition}\label{pr:weakDisorderLimit}
{Assume $u>v$}. For any fixed $T,X\geq 0$, $u>v$, with $\alpha=\alpha^{(n)}=\frac{1}{2}+\sqrt{n}$, we have
\begin{equation}\label{eq:weakDisLim}
(\sqrt{n})^{n T+\sqrt{n}X} {z_{u, -v}^{\text{stat},\alpha^{(n)}} \left(\frac{n T}{2}+\sqrt{n}X+1, \frac{n T}{2}+1\right)}\stackrel{\text{law}}{\to} \int_0^{\infty}\mathcal{Z}_{u-\frac{1}{2}}(T,X\viv 0,Y)\exp (W_{v}(Y)) \dd Y,\end{equation}
as $n\to \infty$, where
$z_{u, -v}^{\text{stat},\alpha^{(n)}} \left(\cdot , \cdot \right)$ are defined by linear interpolation when $nT/2 \notin \mathbf{Z}_{\geq 0}$ or $\sqrt{n}X \notin \mathbf{Z}_{\geq 0}$.
\end{proposition}
\begin{proof}
The proof follows closely  the proof of \cite[Theorem 1.4]{barraquand2023stationary}, except that we need to verify  the convergence of a different initial condition.
Following the convention there, we let $\overline{\varpi}:=\mathbb{E}\left[\varpi_{3,2}\right] = (2\alpha-1)^{-1}$ be the mean of a generic log-gamma polymer bulk weight. Since we are dealing with ``one-row partition functions'' (see \cite{barraquand2023stationary}), we define
\[\widetilde{z}^{\text{1r}}_{u, -v ; \alpha}(t, y):=\left(\frac{1}{2 \overline{\varpi}}\right)^{2 t+y} {z_{u, -v}^{\text {stat}, \alpha}(t+y+1, t+1)},\]
and{
\[
\widetilde{z}^{\text{1r}}_{-v; \alpha}(x):=\widetilde{z}^{\text{1r}}_{u, -v ; \alpha}(0, x+1),
\]
where we drop the subscript $u$ as these weights only depend on $-v$ and $\alpha$.}
We have that for any $y \in \mathbf{Z}_{\geq 0}$, and $ t \in \mathbf{Z}_{\geq 1}$, 
\begin{equation}\label{eq:1rcompos}
\widetilde{z}^{\text{1r}}_{u, -v ; \alpha}(t, y)=\sum_{x=0}^{t+y-1}
\widetilde{z}^{\text{1r}}_{-v ; \alpha}(x) \left(\frac{1}{2 \overline{\varpi}}\right)^{2t+y-x-1}
{z}_{u ; \alpha}(x+2,2 ; t+y+1, t+1),  
\end{equation}
where $z_{u ; \alpha}\left(a, b ; a^{\prime}, b^{\prime}\right)$ denotes the log-gamma partition function for up-right paths starting at $(a, b)$ and ending at $\left(a^{\prime}, b^{\prime}\right)$ as a generalization of the definition \eqref{eq:hsloggam}. This partition function only depends on parameters $u$ and $\alpha$. Similar to \cite{barraquand2023stationary}, by assigning proper boundary weights, one can match the log-gamma polymer paths to reflected  symmetric simple random walks.
Consequently, as a process in $x,y \in \mathbf{Z}_{\geq 0}$ and $ t \in \mathbf{Z}_{\geq 1}$, 
\begin{equation}\label{eq:modparinlaw}
    {z}_{u ; \alpha}(x+2,2 ; t+y+1, t+1)\stackrel{\text { law }}{=} \frac{1}{2} \cdot\left(\frac{\omega(2 t+y-2, y)}{\overline{\varpi} 2^{\1_{y=0}}}\right) \cdot 2^{\1_{y=0}} z_{\boldsymbol{\mathcal{X}}, \boldsymbol{\omega}, \beta}(x, x ; 2 t+y-2, y),
\end{equation}
where $z_{\boldsymbol{\mathcal{X}}, \boldsymbol{\omega}, \beta}(s, x ; t,y)$ is the ``modified polymer partition function'' as defined in \cite[(5.13)]{barraquand2023stationary} for reflected simple symmetric random walks, with the boundary random variables $\boldsymbol{\mathcal { X }}=(\mathcal{X}(r))_{r\in \mathbf{Z}_{\geq 0}}$ being i.i.d., the bulk random variables $ \boldsymbol{\omega} = (\omega(r, w))_{r\in \mathbf{Z}_{\geq 0}, w \in \mathbf{Z}_{\geq 1}}$ being i.i.d., and the distributions and the parameter $\beta>0$ are   so that 
\[
\begin{aligned}
1+\beta \omega(r, w) & \stackrel{\text{law}}{=} \frac{\varpi_{3,2}}{\overline{\varpi}} \sim(2 \alpha-1) \mathrm{Gamma}^{-1}(2 \alpha), \\
\mathcal{X}(r) & \stackrel{\text{law}}{=} \frac{\varpi_{2,2}}{2 \overline{\varpi}} \sim \frac{2 \alpha-1}{2} \mathrm{Gamma}^{-1}(\alpha+u).
\end{aligned}
\]
Then \eqref{eq:1rcompos} and \eqref{eq:modparinlaw} imply that as a process in $t \in \mathbf{Z}_{\geq 1}$ and $y \in \mathbf{Z}_{\geq 0}$,
\begin{equation}\label{eq:1ridinlaw}
\widetilde{z}^{\text{1r}}_{u,-v ; \alpha}(t, y) \stackrel{\text{law}}{=} \frac{2^{\1_{y=0}}}{2} \cdot\left(\1_{y=0} \mathcal{X}(2 t+y-2)+\1_{y>0}(1+\beta \omega(2 t+y-2, y))\right) \cdot z^{\diagup}_{\boldsymbol{\mathcal{X}}, \boldsymbol{\omega}, \beta, z^{\diagup\text{1r}}}(2 t+y-2, y),
\end{equation}
where $z^{\diagup}_{\boldsymbol{\mathcal{X}}, \boldsymbol{\omega}, \beta, z^{\diagup\text{1r}}}(\cdot,\cdot)$ is the partition function as defined in \cite[(5.19)]{barraquand2023stationary} with the initial data $z^{\diagup\text{1r}}(\cdot) = \widetilde{z}^{\text{1r}}_{-v ; \alpha}(\cdot)$ independent of the boundary and bulk weights $\boldsymbol{\mathcal { X }}, \boldsymbol{\omega}$.
For any $S\geq 0, Y \geq 0$, define $z^{\diagup}_{\boldsymbol{\mathcal{X}}, \boldsymbol{\omega}, \beta, z^{\diagup\text{1r}}}(S,Y)$ by linear interpolation when $S \notin \mathbf{Z}_{\geq 0}$ or $Y \notin \mathbf{Z}_{\geq 0}$ or $S+Y$ is not even. 
Note that the only difference between the field $z^{\diagup}_{\boldsymbol{\mathcal{X}}, \boldsymbol{\omega}, \beta, z^{\diagup\text{1r}}}(\cdot,\cdot)$ here and the field $z^{\diagup}_{\boldsymbol{\mathcal{X}}, \boldsymbol{\omega}, \beta, z^{\diagup}}(\cdot,\cdot)$ in \cite[(6.3)]{barraquand2023stationary} is that the initial condition differs. The boundary and the bulk random variables $\boldsymbol{\mathcal{X}},\boldsymbol{\omega}$ are the same. Let $\alpha=\alpha^{(n)} = \sqrt{n}+1/2$ and $\beta=\beta^{(n)}=n^{-1/4}/\sqrt{2}$. In \cite[Section 6]{barraquand2023stationary}, the authors have verified the conditions on the same $\boldsymbol{\mathcal{X}},\boldsymbol{\omega}$ to apply \cite[Theorem 5.4]{barraquand2023stationary}. As long as we can  verify the conditions on the initial data $z^{\diagup\text{1r}}(\cdot)$ here, \cite[Theorem 5.4]{barraquand2023stationary} implies the convergence in law
\begin{equation}\label{eq:1rconv}
   \frac{2^{\1_{X=0}}}{2} z^{\diagup}_{\boldsymbol{\mathcal{X}}, \boldsymbol{\omega}, \beta, z^{\diagup\text{1r}}}(nT+\sqrt{n}X-2,\sqrt{n}X)  \stackrel{\text{law}}{\to} \int_0^{\infty}\mathcal{Z}_{u-\frac{1}{2}}(T,X\viv 0,Y)\exp (W_{v}(Y)) \dd Y,  \quad \text{as } n\to \infty, 
\end{equation}
for any fixed $T,X \geq 0$. The convergence \eqref{eq:weakDisLim} then follows from \eqref{eq:1ridinlaw} and \eqref{eq:1rconv}, together with the fact that $ \mathcal{X}(r)\to 1$ and $1+\beta \omega(r,w) \to 1$ in probability as $n\to \infty$ (the dependence on $n$ was kept implicit through $\alpha^{(n)}$ and $\beta^{(n)}$).

To verify the conditions on the initial data, with $\alpha=\alpha^{(n)}=\sqrt{n}+1/2$, we use Donsker's theorem to check convergence in law 
\[
\widetilde{z}^{\text{1r}}_{-v ; \alpha}(\sqrt{n}X)  \stackrel{\text{law}}{\to} \exp(W(X)+vX) = \exp(W_v(X))  \quad \text{as a process in } X \in [0,\infty).
\]
Moreover, for any $n \in \mathbf{Z}_{\geq 1}$ and $X \geq 0$ such that $\sqrt{n}X\in \mathbf{Z}_{\geq 0}$, we have that $\mathbb{E}[ \widetilde{z}^{\text{1r}}_{-v ; \alpha}(\sqrt{n}X) ^2]=[n M_2(-v)]^{{\sqrt{n} X+1}}$ with $M_k(- v):=\mathbb{E}[(\varpi^{(n)})^k]$ where $\varpi^{(n)} \sim \operatorname{Gamma}^{-1}(\alpha^{(n)}-v)$. By \cite[(1.1)]{barraquand2023stationary}, \[n M_2 (-v)=1+(2+2 v) n^{-1 / 2}+O(n^{-1}),\]
which implies the uniform   bound 
\[
\sup_{n\in \mathbf{Z}_{\geq 1}}\sup_{X\geq 0}
e^{-aX}\mathbb{E}[ \widetilde{z}^{\text{1r}}_{-v ; \alpha}(\sqrt{n}X) ^2] <\infty,
\]
for some $a>0$ that only depends on $v$. The conditions imposed on the initial data in \cite[Theorem 5.4]{barraquand2023stationary} are thus verified for our $z^{\diagup}_{\boldsymbol{\mathcal{X}}, \boldsymbol{\omega}, \beta, z^{\diagup\text{1r}}}(\cdot,\cdot)$.
\end{proof}

Combining the results \eqref{eq:ratioidinlaw} and \eqref{eq:weakDisLim}, we prove the identity in law for the partition function on the boundary in the continuum setting under the condition $u>v$. This identity in law was previously claimed in \cite[Claim 4.9]{barraquand2020halfkpz} with a slightly different convention on the parameters. For discussion on general $u,v \in \R$, see Remark~\ref{re:idinlaw} below.
\begin{proposition}\label{pr:symm1}
    For any $T\geq 0$ and $u,v \in \R$ with $u>v$,  we have \begin{equation}\label{eq:covIDinlaw}
\int_0^{\infty}\mathcal{Z}_{u-\frac{1}{2}}(T,0\viv 0,Y)\exp (W_{v}(Y)) \dd Y \stackrel{\text{law}}{=}\int_0^{\infty}\mathcal{Z}_{-v-\frac{1}{2}}(T,0\viv 0,Y)\exp (W_{-u}(Y)) \dd Y.
\end{equation}
\end{proposition}
\begin{proof}
    Since Corollary~\ref{cr:ratioequalinlaw} is not generalized to joint law, we need to be careful to avoid using any linear interpolations for the identity in law. For any fixed $T\in \mathbf{Q}_{\geq 0}$, there exists a series of integers $\{n_k \in \mathbf{Z}_{\geq 0}\}_{k\geq 1}$ such that $n_k\to \infty$ as $k \to \infty$ and $\frac{n_kT}{2} \in \mathbf{Z}_{\geq 0}$. Then Corollary~\ref{cr:ratioequalinlaw} implies that for each $k\geq 1$,
\[(\sqrt{n_k})^{n_k T} {z_{u, -v}^{\text{stat},\alpha^{(n_k)}} \left(\frac{n_k T}{2}+1, \frac{n_k T}{2}+1\right)}\stackrel{\text{law}}{=}(\sqrt{n_k})^{n_k T} {z_{-v, u}^{\text{stat},\alpha^{(n_k)}} \left(\frac{n_k T}{2}+1, \frac{n_k T}{2}+1\right)}.
\]
By taking the limit $k\to \infty$, Proposition~\ref{pr:weakDisorderLimit} implies that for any fixed $T\in \mathbf{Q}_{\geq 0}$, \eqref{eq:covIDinlaw} holds.

Using the continuity of mild solutions, we extend \eqref{eq:covIDinlaw} to any $T\in [0,\infty)$.
\end{proof}


\begin{remark}
By sending $v\to -\infty$ in Proposition~\ref{pr:symm1}, one can formally recover \cite[Theorem 1.1]{shalinID}.
\end{remark}

To complete the proof of Proposition~\ref{pr:symmetry}, it remains to pass the limit $v\to u$.

\begin{proof}[Proof of Proposition~\ref{pr:symmetry}]
By the Minkowsiki and H\"older's inequalities, the $\Prob$-almost surely continuity of the Green's function $\mathcal{Z}_\cdot(\cdot,\cdot\viv \cdot,\cdot)$ (Proposition~\ref{pr:greens}), the uniform moment bounds in \eqref{eq:greensPosMom} and the dominated convergence theorem, as $v\to u$,
we have 
\[\left\|\int_0^{\infty}\mathcal{Z}_{u-\frac{1}{2}}(T,0\viv 0,Y)\exp (W_{v}(Y)) \dd Y -\int_0^{\infty}\mathcal{Z}_{u-\frac{1}{2}}(T,0\viv 0,Y)\exp (W_{u}(Y)) \dd Y \right\|_2\to 0,\]
and
\[\left\|\int_0^{\infty}\mathcal{Z}_{-v-\frac{1}{2}}(T,0\viv 0,Y)\exp (W_{-u}(Y)) \dd Y -\int_0^{\infty}\mathcal{Z}_{-u-\frac{1}{2}}(T,0\viv 0,Y)\exp (W_{-u}(Y)) \dd Y \right\|_2\to 0.\]
The identity \eqref{eq:covIDinlaw} can thus be extended to $v=u$, and we have $Z_{u-\frac{1}{2}}(T,0) \stackrel{\text{law}}{=}Z_{-u-\frac{1}{2}}(T,0)$ for any $T\geq 0$ and $u \in \R$. Taking a logarithm on both sides and applying the continuous mapping theorem then completes the proof.
\end{proof}
\begin{remark}\label{re:idinlaw}
By using analytic continuation, one should be able to extend Corollary~\ref{cr:ratioequalinlaw} to the more general condition with $u,v\in \R$ and $\alpha>\min(0,-u,v)$. In other words, the identity in law still holds after dropping the assumption $u-v>0$. Heuristically, this is clear as the ratio $z_{u,-v}^{\text{stat},\alpha}(m,m)$ does not involve $\varpi_{1,1}$ and thus does not depend on the value of $u-v$. This extension has been explained in the proof of \cite[Claim 4.9]{barraquand2020half}, but a verification for the analytic continuation has not yet been proved.
\end{remark}

As a last result, we prove Corollary~\ref{co:upbdright}.

\begin{proof}[Proof of Corollary~\ref{co:upbdright}]
From \eqref{eq:varidkpzhs} in Theorem~\ref{thm:varianceID}, we know that for any $u\in \R, T\in [0,\infty)$,
\begin{equation}\label{eq:rewritevarID}
\Expe \mathbb{E}^{\mathbb{Q}^{u,T}}_{0}[X_T] = \frac{1}{2} \left(\Var[H_u(T,0)]+uT\right).
\end{equation}
By Proposition~\ref{pr:symmetry}, Theorem~\ref{thm:var} is extended to all $u>0$ with
\[
uT\leq \Var[H_u(T,0)] \leq 2\psi'(2u)+uT.
\]
Substituting the above bounds of $\Var[H_u(T,0)]$ into \eqref{eq:rewritevarID} gives \eqref{eq:intright}.
\end{proof}


\section{Extended critical regime: proof of Theorem~\ref{thm:scpolym}}\label{se:sc}
{In this last section, }we prove \eqref{eq:scpolym}, which provides an upper bound to the polymer endpoint displacement when the boundary parameter $u$ is in the {extended critical} regime: $u=ct^{-\alpha}$ with $\alpha>1/3$ and $c\in\R$. Combining
with the variance identity obtained in Theorem~\ref{thm:varianceID}, we further derive an upper bound for the fluctuation size of the height function in the {extended critical regime}, as stated in \eqref{eq:scfe}.

To prove \eqref{eq:scpolym}, we use a  stochastic monotonicity argument different from Section~\ref{se:uppbd}. The main idea is the following. 

For any realization of $(\xi, W)$ and  any $u_1< u_2$, one may expect that the endpoint distribution of the half-space  continuum directed random polymer   $\mathbb{Q}_0^{u_1,t}$ would be stochastically dominated by that  of the half-space continuum directed random polymer  $\mathbb{Q}_0^{u_2,t}$. Intuitively, this is because that when the parameter $u$ becomes greater, the ``wall attraction'' at $x=0$ imposed on the   polymer would be weakened, while at the same time, the terminal data at time $t$ would have a greater drift. Both the weakened attraction force and the greater drift would drive the polymer endpoint farther away from the wall. Recall that in \eqref{eq:hsQuench0} we defined $\rho^R_{u}(\cdot\viv t,0)$ as the quenched density of polymer endpoints under $\mathbb{Q}_0^{u,t}$. Let $c\in\R$,  $u_1=ct^{-\alpha}$ for some $\alpha>1/3$, and $u_2=(|c|\vee 1)t^{-1/3}$. For any $t\geq 1$,  we expect that $ \rho^{R}_{u_1}(\cdot \viv t,0) \leq_{st} \rho^{R}_{u_2}(\cdot\viv t,0)$ $\Prob$-almost surely.

On the other hand, from the upper bound in \eqref{eq:intright} of Corollary~\ref{co:upbdright}, we know that when  $u_2=ct^{-1/3}$ for some $c>0$, for any $t\geq 1$, the annealed mean of the polymer endpoint is bounded by \[\Expe \mathbb{E}_0^{\mathbb{Q}^{u_2,t}}[X_t] \leq Ct^{2/3}\]
for some $C>0$ depending only on $c$.  Combining  the above estimate and the stochastic dominance of $ \rho^{R}_{u_1}(\cdot \viv t,0) \leq_{st} \rho^{R}_{u_2}(\cdot\viv t,0)$  gives \eqref{eq:scpolym}.

We now prove the aforementioned results. The Lemma~\ref{le:udominance} below is the  main ingredient to prove $ \rho^{R}_{u_1}(\cdot \viv t,0) \leq_{st} \rho^{R}_{u_2}(\cdot\viv t,0)$ $\Prob$-almost surely.  One can compare it with Lemma~\ref{le:ptstdom} above. 
In fact, our approach to the  stochastic monotonicity here is partly inspired by Section~\ref{se:uppbd}. 
As in Section~\ref{se:uppbd},  we first analyze the  determinants in \eqref{eq:bounddeterminant} formed by four copies of Green’s functions, and we will work with a smoothed noise to utilize the Feynman-Kac formula before passing to the limit. The determinant in \eqref{eq:bounddeterminant} being nonnegative further leads to the relation $ \rho^{R}_{u_1}(\cdot \viv t,0) \leq_{st} \rho^{R}_{u_2}(\cdot\viv t,0)$.

For reference, we note that in Section~\ref{se:uppbd}, we proved $ \rho^{R}_{u}(\cdot\viv t,x_1) \leq_{st} \rho^{R}_{u}(\cdot\viv t,x_2)$ for any  $x_1< x_2$. Here we are expecting that $ \rho^{R}_{u_1}(\cdot\viv t,0)\leq_{st}  \rho^{R}_{u_2}(\cdot\viv t,0)$ for any $u_1 < u_2$. 

\begin{lemma}\label{le:udominance}
Fix $t>0$. For any $u_1 < u_2$ and $0\leq y_1<y_2$, we have
\begin{equation}\label{eq:bounddeterminant}
\det[\mathcal{Z}_{u_i-\frac{1}{2}}(t,0\viv 0, y_j) ]_{i,j \in \{1,2\}} \geq 0 \quad \Prob\text{-almost surely}.
\end{equation}
\end{lemma}

Before proving Lemma~\ref{le:udominance}, we first use it to prove Theorem~\ref{thm:scpolym}.

\begin{proof}[Proof of Theorem~\ref{thm:scpolym}]
By \eqref{eq:bounddeterminant}, for any fixed $t>0$ and $u_1 < u_2$, with  probability one,
\begin{equation}\label{eq:compbdpa}
\mathcal{Z}_{u_1-\frac{1}{2}}(t,0\viv 0, y_1)\mathcal{Z}_{u_2-\frac{1}{2}}(t,0\viv 0, y_2)  \geq \mathcal{Z}_{u_1-\frac{1}{2}}(t,0\viv 0, y_2)\mathcal{Z}_{u_2-\frac{1}{2}}(t,0\viv 0, y_1),
\end{equation}
for all $0\leq y_1<y_2$ simultaneously. (As in the proof of Corollary~\ref{co:endptcompdom}, one can first consider rational $y_1,y_2$ then use the continuity of $\mathcal{Z}$.) 

Similar to Proposition~\ref{pr:mombds0}, it is straightforward to show that for any fixed $u_1, u_2\in \R$, we have $\Prob$-almost surely
\[
0<\int_0^\infty \mathcal{Z}_{u_2-\frac{1}{2}}(t,0\viv 0, y) \exp(W_{u_1}(y))\dd y<\infty.
\]
Define \[\mathcal{U}(t, u, \theta):= \frac{ \int_0^\infty y \mathcal{Z}_{u-\frac{1}{2}}(t,0\viv 0, y) \exp(W_{\theta}(y))\dd y }{ \int_0^{\infty} \mathcal{Z}_{u-\frac{1}{2}}(t,0\viv 0, y) \exp(W_{\theta}(y))\dd y},\]
which is the quenched mean of the endpoint displacement of a half-space polymer with boundary parameter $u$ and initial   parameter $\theta$. 
As in the proof of Corollary~\ref{co:endptcompdom}, one can obtain  from \eqref{eq:compbdpa} that for any fixed $t>0$ and $u_1< u_2$,
\[\mathcal{U}(t, u_1, u_1)\leq \mathcal{U}(t, u_2, u_1)  \quad \Prob\text{-almost surely}.
\]
(To compare with the proof of Corollary~\ref{co:endptcompdom}, one only needs to replace $u_i\mapsto x_i, i=1,2$.)

On the other hand, by the Jensen's inequality,
\[
\partial_\theta \mathcal{U}(t, u, \theta) =\frac{ \int_0^\infty y^2 \mathcal{Z}_{u-\frac{1}{2}}(t,0\viv 0, y) \exp(W_{\theta}(y))\dd y }{ \int_0^{\infty} \mathcal{Z}_{u-\frac{1}{2}}(t,0\viv 0, y) \exp(W_{\theta}(y))\dd y} -\frac{ [\int_0^\infty y \mathcal{Z}_{u-\frac{1}{2}}(t,0\viv 0, y) \exp(W_{\theta}(y))\dd y]^2 }{ [\int_0^{\infty} \mathcal{Z}_{u-\frac{1}{2}}(t,0\viv 0, y) \exp(W_{\theta}(y))\dd y]^2}\geq 0,
\]
$\Prob$-almost surely. It thus follows that $\mathcal{U}(t, u_2, u_1)\leq \mathcal{U}(t, u_2, u_2)$ $\Prob$-almost surely.

Combining the above results, for any   $u_1< u_2$, we have $\mathcal{U}(t, u_1, u_1)\leq \mathcal{U}(t, u_2, u_2)$ $\Prob$-almost surely, 
which further implies that 
\begin{equation}\label{eq:EPcomp}
    \Expe \mathbb{E}_0^{\mathbb{Q}^{u_1,t}}[X_t] \leq \Expe \mathbb{E}_0^{\mathbb{Q}^{u_2,t}}[X_t].
\end{equation}

Now for any $u_1=ct^{-\alpha}$ with $\alpha>1/3$ and $c\in \R$, let $u_2 =  (|c|\vee 1)t^{-1/3}$.  For any $t\geq 1$, we have $u_1\leq u_2$, and by \eqref{eq:intright},
\begin{equation}\label{eq:useofCo16}
\Expe \mathbb{E}_0^{\mathbb{Q}^{u_2,t}}[X_t] \leq u_2t+\psi'(2u_2)\leq Ct^{2/3},
\end{equation}
for some constant $C>0$ depending on $c$ only.  \eqref{eq:EPcomp} and \eqref{eq:useofCo16} together prove \eqref{eq:scpolym}, and \eqref{eq:scfe} follows immediately using \eqref{eq:varidkpzhs}.
\end{proof}

To finish the proof, it only remains to prove Lemma~\ref{le:udominance}. 

\begin{proof}[Proof of Lemma~\ref{le:udominance}]
Similar to the proof of Lemma~\ref{le:ptstdom}, we first prove an inequality analogous to \eqref{eq:bounddeterminant} for the Green's functions of half-space SHE with a smoothed noise (\eqref{eq:mollifieducomp} below). With smoothed noise, one can prove such an inequality using the Feynman-Kac representation and the strong Markov property of reflected Brownian motions. We then pass to the limit using Lemma~\ref{le:fk} to obtain \eqref{eq:bounddeterminant} for the Green's functions of half-space SHE with the white noise.

As in the proof of Lemma~\ref{le:ptstdom}, let $\xi_{\delta}(t,x) = \int_\R  p_\delta(x-y)\xi(t, \dd y)$ be  a Gaussian noise that is white in time and smooth in space, with the covariance function
\[
Q_{\delta}(t_1,t_2,x_1,x_2) = \Expe \xi_\delta(t_1,x_1)\xi_\delta(t_2,x_2)=\delta_0(t_1-t_2) p_{2\delta}(x_1-x_2).
\]
Using the same notations as in the proof of Lemma~\ref{le:ptstdom},
for any $\mu \in \R, \delta>0, x,y \in [0,\infty)$,  define
\[\begin{aligned}
    \mathcal{Z}_\mu^{\delta}(t,x\viv 0, y) &:= \kernel^{N}(t,x\viv 0,y)
\mathbb{E}^x_B\left[\exp\left(-\mu L^{0}_{t} +
\int_{0}^{t}\xi_{\delta}(t-r,|B_r|)\dd r- \frac{t}{2}p_{2\delta}(0)\right)\bigg||B_{t}|=y\right]\\
&= \kernel^{N}(t,x\viv 0,y)
\mathbb{E}^{x,y}_{\Upsilon}\left[\exp\left(-\mu L^{0,{\Upsilon}}_{t} +
\int_{0}^{t}\xi_{\delta}(t-r,\Upsilon_r)\dd r- \frac{t}{2}p_{2\delta}(0)\right)\right],
\end{aligned}
\]with the expectation $\mathbb{E}_\Upsilon^{x,y}$ being taken with respect to the reflected Brownian bridges starting from $\Upsilon_0=x$ and ending at $\Upsilon_t=y$, and the local time term $L^{0,{\Upsilon}}_{t}$ is the same as defined in \eqref{eq:ltrBM}.
Furthermore, by the same argument following after \eqref{eq:equalinlawRE} in Appendix~\ref{ap:reversal}, one can always make a time reversal to the paths of reflected Brownian bridges to obtain\[   \mathcal{Z}_\mu^{\delta}(t,x\viv 0, y) =
 \kernel^{N}(t,y\viv 0,x)
\mathbb{E}^{y,x}_{\Upsilon}\left[\exp\left(-\mu L^{0,{\Upsilon}}_{t} +
\int_{0}^{t}\xi_{\delta}(r,\Upsilon_r)\dd r- \frac{t}{2}p_{2\delta}(0)\right)\right].
\]
We use $L^{0,{\Upsilon}}_{[s_1,s_2]}$ to denote the local time at zero for reflected Brownian motions restricted to the time interval $[s_1,s_2] \subset [0,t]$:
\begin{equation}\label{eq:localtimetau}
L^{0,{\Upsilon}}_{[s_1,s_2]} := \lim _{\eps \to 0} \frac{1}{2 \varepsilon} \int_{s_1}^{s_2} \1_{[0, \varepsilon]}\left(\Upsilon_s\right) \dd s.
\end{equation}
One may rewrite $L^{0,{\Upsilon}}_{t}=L^{0,{\Upsilon}}_{[0,t]} $. To ease notations, we also define the following functional on the paths of reflected Brownian motions between any time interval $[s_1,s_2] \subset [0,t]$:
\[
F^{\delta,\mu}_{[s_1,s_2]}(\Upsilon) := \exp\left(-\mu L^{0,{\Upsilon}}_{[s_1,s_2]} +
\int_{s_1}^{s_2}\xi_{\delta}(r,\Upsilon_r)\dd r- \frac{s_2-s_1}{2}p_{2\delta}(0)\right).
\]
For any $0\leq s_1\leq s_2\leq t$, 
$F^{\delta,\mu}_{[s_1,s_2]}(\Upsilon) $ is a $\Prob$-almost surely continuous, strictly positive, and multiplicative functional of $\Upsilon$.  We also note that the functional $F^{\delta,\mu}_{[s_1,s_2]}(\cdot)$ depends on the boundary parameter $\mu$ through the multiplier of the local time term.

Our goal is to prove a $\delta$-level mollified version of \eqref{eq:bounddeterminant}, namely, for any $t>0, \delta>0$, $u_1< u_2$ and $0\leq y_1<y_2$,
\begin{equation}\label{eq:mollifieducomp}
\det\left[\mathcal{Z}^\delta_{u_i-\frac{1}{2}}(t,0\viv 0, y_j) \right]_{i,j \in \{1,2\}}=\det\left[ \kernel^{N}(t,y_j\viv 0,0) \mathbb{E}_\Upsilon^{y_j,0}\left[F^{\delta,u_i-\frac{1}{2}}_{[0,t]}\right]\right]  \geq 0 \quad \Prob\text{-almost surely}.
\end{equation}
One can check that \eqref{eq:bounddeterminant} follows from   \eqref{eq:mollifieducomp} with Lemma~\ref{le:fk}~\eqref{it:approx}.
In fact, by Lemma~\ref{le:fk}~\eqref{it:approx} (with a mollification in space only), for any $t>0, u_1,u_2\in\R, y_1,y_2 \in [0,\infty)$ and $p\in [1,\infty)$,
\[\det\left[\mathcal{Z}^\delta_{u_i-\frac{1}{2}}(t,0\viv 0, y_j) \right]_{i,j \in \{1,2\}} \to  \det\left[\mathcal{Z}_{u_i-\frac{1}{2}}(t,0\viv 0, y_j) \right]_{i,j \in \{1,2\}}\quad \text{in }L^p(\Omega) \quad \text{as }\delta \to 0. \]
Thus the proof is complete once we prove \eqref{eq:mollifieducomp}.

To prove \eqref{eq:mollifieducomp}, we use a path-switching argument. Let  $\mathbb{P}^{\mathbf{y},\mathbf{0}}_{\Upsilon^1,\Upsilon^2 } $ be the measure of a two-dimensional reflected Brownian bridge  $(\Upsilon^1,\Upsilon^2)$ starting from $\mathbf{y}=(y_1,y_2)$ at time $0$ and ending at $\mathbf{0}=(0,0)$ at time $t$. Since $\kernel^{N}(t,y\viv 0,0) $ is positive for any $y\in[0,\infty)$, it suffices to show that  \begin{equation}\label{eq:detpos}
\mathbb{E}^{\mathbf{y},\mathbf{0}}_{\Upsilon^1,\Upsilon^2 }  \left[F^{\delta, u_1-\frac{1}{2}}_{[0,t]}(\Upsilon^1)F^{\delta, u_2-\frac{1}{2}}_{[0,t]}(\Upsilon^2)
- F^{\delta, u_2-\frac{1}{2}}_{[0,t]}(\Upsilon^1)F^{\delta, u_1-\frac{1}{2}}_{[0,t]}(\Upsilon^2)\right] \geq 0  \quad \Prob\text{-almost surely},
\end{equation}
for any $t>0, \delta>0$, $u_1<u_2$ and $0\leq y_1<y_2$. 

Define
the stopping time
\[
\tau = \inf \{ s \in [0,t]: \Upsilon^1(s)=\Upsilon^2(s)\}.
\]
Since $\Upsilon^1(t)=\Upsilon^2(t)=0$ and $\Upsilon^1(0)=y_1<y_2=\Upsilon^2(0)$, we have $0< \tau \leq t$ $\mathbb{P}^{\mathbf{y},\mathbf{0}}_{\Upsilon^1,\Upsilon^2 }\text{-almost surely}$.
By the path continuity of  reflected Brownian motions, $\Upsilon^1(s)\leq \Upsilon^2(s)$ for all $s\in[0,\tau]$. It thus follows from the definition of the local time through approximation
\eqref{eq:localtimetau} that
\begin{equation}\label{eq:localtimecomparison}
L^{0,{\Upsilon^1}}_{[0,\tau]} \geq 
L^{0,{\Upsilon^2}}_{[0,\tau]}  \quad \mathbb{P}^{\mathbf{y},\mathbf{0}}_{\Upsilon^1,\Upsilon^2 }\text{-almost surely}.
\end{equation}

Define
$$
\left(\tilde{\Upsilon}_r^1, \tilde{\Upsilon}_r^2\right)= \begin{cases}\left(\Upsilon_r^1, \Upsilon_r^2\right) & r \leq \tau \\ \left(\Upsilon_r^2, \Upsilon_r^1\right) & r >\tau .\end{cases}
$$
As stated in \cite[Proposition 1]{fpyMarkovianBridge},  the reflected Brownian bridges $\Upsilon_\cdot$ under $\mathbb{P}^{{x},{y}}_{\Upsilon}$ for any $x,y\in[0,\infty)$ are strong Markov processes. (The proof follows from the optional stopping theorem and that the reflected Brownian motion is strong Markov.) Hence $\left(\tilde{\Upsilon}_r^1, \tilde{\Upsilon}_r^2\right)$ is also a two-dimensional reflected Brownian bridge starting from $\mathbf{y}=(y_1,y_2)$ at time $0$ and ending at $\mathbf{0}=(0,0)$ at time $t$, and the left-hand-side of \eqref{eq:detpos} equals to
\[
\mathbb{E}^{\mathbf{y},\mathbf{0}}_{\tilde{\Upsilon}^1,\tilde{\Upsilon}^2 }  \left[F^{\delta, u_1-\frac{1}{2}}_{[0,t]}(\tilde{\Upsilon}^1)F^{\delta, u_2-\frac{1}{2}}_{[0,t]}(\tilde{\Upsilon}^2)
- F^{\delta, u_2-\frac{1}{2}}_{[0,t]}(\tilde{\Upsilon}^1)F^{\delta, u_1-\frac{1}{2}}_{[0,t]}(\tilde{\Upsilon}^2)\right],\]
which can be further expanded as
\[\begin{aligned}
&\mathbb{E}^{\mathbf{y},\mathbf{0}}_{\Upsilon^1,\Upsilon^2 }  \bigg[F^{\delta, u_1-\frac{1}{2}}_{[0,\tau]}(\Upsilon^1)F^{\delta, u_1-\frac{1}{2}}_{[\tau,t]}(\Upsilon^2) F^{\delta, u_2-\frac{1}{2}}_{[0,\tau]}(\Upsilon^2)F^{\delta, u_2-\frac{1}{2}}_{[\tau,t]}(\Upsilon^1)
\\&- F^{\delta, u_2-\frac{1}{2}}_{[0,\tau]}(\Upsilon^1) F^{\delta, u_2-\frac{1}{2}}_{[\tau,t]}(\Upsilon^2) F^{\delta, u_1-\frac{1}{2}}_{[0,\tau]}(\Upsilon^2)F^{\delta, u_1-\frac{1}{2}}_{[\tau,t]}(\Upsilon^1)\bigg]=:\mathcal{S}_1.\end{aligned}
\]Meanwhile, one can also expand the left-hand-side of \eqref{eq:detpos} directly so  that it equals to
\[
\begin{aligned}
&\mathbb{E}^{\mathbf{y},\mathbf{0}}_{\Upsilon^1,\Upsilon^2 }  \bigg[F^{\delta, u_1-\frac{1}{2}}_{[0,\tau]}(\Upsilon^1)F^{\delta, u_1-\frac{1}{2}}_{[\tau,t]}(\Upsilon^1) F^{\delta, u_2-\frac{1}{2}}_{[0,\tau]}(\Upsilon^2)F^{\delta, u_2-\frac{1}{2}}_{[\tau,t]}(\Upsilon^2)
\\&- F^{\delta, u_2-\frac{1}{2}}_{[0,\tau]}(\Upsilon^1) F^{\delta, u_2-\frac{1}{2}}_{[\tau,t]}(\Upsilon^1) F^{\delta, u_1-\frac{1}{2}}_{[0,\tau]}(\Upsilon^2)F^{\delta, u_1-\frac{1}{2}}_{[\tau,t]}(\Upsilon^2)\bigg]=:\mathcal{S}_2.\end{aligned}
\]
Denote the left-hand-side of \eqref{eq:detpos} by $\mathcal{S}$, then by a symmetrization we have
\[
\begin{aligned}
\mathcal{S} &= \mathcal{S}_1 = \mathcal{S}_2
=\frac{1}{2}\mathcal{S}_1+\frac{1}{2}\mathcal{S}_2\\
&=\frac{1}{2}\mathbb{E}^{\mathbf{y},\mathbf{0}}_{\Upsilon^1,\Upsilon^2 }\left[ \left(F^{\delta, u_1-\frac{1}{2}}_{[0,\tau]}(\Upsilon^1) F^{\delta, u_2-\frac{1}{2}}_{[0,\tau]}(\Upsilon^2)-F^{\delta, u_1-\frac{1}{2}}_{[0,\tau]}(\Upsilon^2) F^{\delta, u_2-\frac{1}{2}}_{[0,\tau]}(\Upsilon^1) \right) F^{\delta, u_1-\frac{1}{2}}_{[\tau,t]}(\Upsilon^2) F^{\delta, u_2-\frac{1}{2}}_{[\tau,t]}(\Upsilon^1)
 \right]\\
 &+\frac{1}{2}\mathbb{E}^{\mathbf{y},\mathbf{0}}_{\Upsilon^1,\Upsilon^2 } \left[ \left(F^{\delta, u_1-\frac{1}{2}}_{[0,\tau]}(\Upsilon^1) F^{\delta, u_2-\frac{1}{2}}_{[0,\tau]}(\Upsilon^2)-F^{\delta, u_1-\frac{1}{2}}_{[0,\tau]}(\Upsilon^2) F^{\delta, u_2-\frac{1}{2}}_{[0,\tau]}(\Upsilon^1) \right) F^{\delta, u_1-\frac{1}{2}}_{[\tau,t]}(\Upsilon^1) F^{\delta, u_2-\frac{1}{2}}_{[\tau,t]}(\Upsilon^2)
 \right]\\
 &= \frac{1}{2}\mathbb{E}^{\mathbf{y},\mathbf{0}}_{\Upsilon^1,\Upsilon^2 } \left[ \det \left[F^{\delta, u_i-\frac{1}{2}}_{[0,\tau]}(\Upsilon^j) \right]_{i,j\in\{1,2\}}\left( F^{\delta, u_1-\frac{1}{2}}_{[\tau,t]}(\Upsilon^1) F^{\delta, u_2-\frac{1}{2}}_{[\tau,t]}(\Upsilon^2) + F^{\delta, u_1-\frac{1}{2}}_{[\tau,t]}(\Upsilon^2) F^{\delta, u_2-\frac{1}{2}}_{[\tau,t]}(\Upsilon^1)
 \right)
 \right].
\end{aligned}
\]

By \eqref{eq:localtimecomparison} and the assumption that $u_1< u_2$, we have 
\begin{equation}\label{eq:uandloaltime}
\left(u_1-u_2\right)\left(L^{0,{\Upsilon^1}}_{[0,\tau]} - L^{0,{\Upsilon^2}}_{[0,\tau]}\right) \leq 0  \quad  \mathbb{P}^{\mathbf{y},\mathbf{0}}_{\Upsilon^1,\Upsilon^2 }\text{-almost surely}.\end{equation}
One can directly check that \eqref{eq:uandloaltime} implies $ \det \left[F^{\delta, u_i-\frac{1}{2}}_{[0,\tau]}(\Upsilon^j) \right]_{i,j\in\{1,2\}}\geq 0$,  $\mathbb{P}^{\mathbf{y},\mathbf{0}}_{\Upsilon^1,\Upsilon^2 }$-almost surely.  Since the functionals $F^{\delta,u_i-\frac{1}{2}}_{[\tau,t]}(\cdot)$ are always positive, $\mathcal{S} \geq 0$ is thus proved.
\end{proof}


\appendix

\section{Auxiliary lemmas}\label{ap:auxlemproofs}
\subsection{Proof of Lemma~\ref{le:posMom}}\label{ap:auxlemproofsposMom}
The proof follows a standard argument as in \cite{walsh1986introduction}. We use $\alpha \lesssim \beta$ to denote $\alpha \leq C \beta$ for any constant $C=C(a, b, \mu_0, p,\zeta)$.

 By the chaos expansion \eqref{eq:chaos1}, for any fixed $\mu \in [\mu_0, \infty), s \in [a,b), t \in (s, b]$ and $x \in [0, \infty)$,
\[
Z_\mu(t,x\viv s, \zeta) = \sum_{n=0}^{\infty} z_n(t,x\viv s, \zeta, \mu),
\]
where
\[
\begin{aligned}
z_0(t,x\viv s, \zeta, \mu)&:= \int_0^{\infty} \kernel_{\mu}^R(t,x\viv s,y)\zeta(\dd y),\\
\text{and }\quad z_{n+1}(t,x\viv s, \zeta, \mu) & := \int_s^t\int_0^{\infty}  \kernel_{\mu}^R(t,x\viv r,w) z_n(r,w\viv s, \zeta, \mu) \xi(\dd w\dd r ).
\end{aligned}
\]
Define \[f_n(r) = \sup_{x\in [0,\infty)}\|z_n(r,x\viv s, \zeta, \mu)\|_p^2 [z_0(r,x\viv s, \zeta, \mu)]^{-1}.\] We will show below that this supremum always exists. By Hypothesis~\ref{hy:detic}, for any $\mu \in  [\mu_0, \infty), r \in (s, b]$, \[f_0(r) = \sup_{x\in [0,\infty)}z_0(r,x\viv s, \zeta, \mu) \lesssim (r-s)^{-1/2}.\]
While $z_{n+1}(t,x\viv s, \zeta, \mu) $ is not a martingale in $t$, the process \[M_{n+1}(v) :=  \int_s^v\int_0^{\infty}  \kernel_{\mu}^R(t,x\viv r,w) z_n(r,w \viv s, \zeta, \mu) \xi(\dd w\dd r )\] is a martingale in $v\in[s,t]$.
By applying the Burkholder-Davis-Gundy (to $M_{n+1}(v)$) and the Minkowski inequalities,  \eqref{eq:elem1} and semigroup property of the Robin heat kernel, we derive \[
\begin{aligned}
&\|z_{n+1}(t,x\viv s, \zeta, \mu) \|^2_p  = \|M_{n+1}(t)\|_p^2 \lesssim \left\| \int_s^t\int_0^{\infty} \left[\kernel_{\mu}^R(t,x\viv r,w) z_n(r,w \viv s, \zeta, \mu)\right]^2 \dd r \dd w \right\|_{p/2}\\
& \leq  \int_s^t \int_0^{\infty}  \kernel_{\mu}^R(t,x\viv r,w)^2  \|z_n(r,w \viv s, \zeta, \mu)\|^2_{p} \dd r \dd w \\& \leq   \int_s^t \int_0^{\infty} \kernel_{\mu}^R(t,x\viv r,w)^2z_0(r,w\viv s, \zeta, \mu)f_n(r) \dd r \dd w  \lesssim z_0(t,x\viv s, \zeta, \mu)  \int_r^t f_n(r) (t-r)^{-1/2}\dd r.
\end{aligned}
\]
where we used inequality \eqref{eq:elem1} and semigroup property of the Robin heat kernel
 in the last inequality. Thus \[f_{n+1}(t)\lesssim  \int_s^t f_n(r)(t-r)^{-1/2}\dd r.\]
It is straightforward to check that $f_1(r)$ and $f_2(r)$ are both bounded on $[s,t]$. By iteration, for any $n \geq 2$, (with the convention  $(n/2)!=\Gamma(\tfrac{n}{2}+1)$,) \[
f_{n}(t) \leq C \int_s^tf_{n-2}(r)\dd r \leq   \frac{[C(t-s)]^{(n-1)/2}}{(n/2)!}
\]
for some constant $C =C(a,b,\mu_0, p,\zeta)$.
Now by the Minkowski's inequality, 
\[
\begin{aligned}
& \|Z_\mu(t,x\viv s, \zeta) \|_p \leq \sum_{n=0}^{\infty} \|z_n(t,x\viv s, \zeta, \mu)\|_p \leq \sum_{n=0}^{\infty}[z_0(t,x\viv s, \zeta, \mu) f_n(t)]^{1/2}\\&\lesssim z_0(t,x\viv s, \zeta, \mu)^{1/2}(t-s)^{-1/4}\sum_{k=0}^{\infty}\left[\frac{[C(t-s)]^{n/2}}{(n/2)!}\right]^{1/2} \lesssim z_0(t,x\viv s, \zeta, \mu)^{1/2}(t-s)^{-1/4}.
\end{aligned}
\]
This completes the proof.

\subsection{Feynman-Kac Approximation}\label{ap:fk}
We describe how to approximate the half-space SHE solution $Z_\mu(t,x \viv s, \zeta)$ with the Feynman-Kac type representations. This approximation is used in the proof of Lemma~\ref{le:negmom}. It is similar to full-space Feynman-Kac type approximations, except that the Brownian paths are now restricted to the positive half-plane (become the reflected Brownian motion), and we have a local time term on the boundary.

We consider a mollification of the space-time white noise $\xi(t,x)$.  For the proof of Lemma~\ref{le:negmom} below, we will smooth both time and spatial variables. As in \cite{bertini1995stochastic}, the Feynman-Kac type approximation is also valid when we only mollify the spatial variable. We omit the statements and proofs for the ``spatial mollification only'' approximation, as they are similar.

Let $g\in \mathcal{C}_{c}^\infty(\R,[0,\infty))$ be a  smooth symmetric function with compact support such that $\int_{\R}g(x)\dd x=1$. We define $g_{\eps}(x) := \frac{1}{\eps}g(\frac{x}{\eps})$ and
\begin{equation}\label{eq:smstnoise}
\xi_{\eps, \delta}(t,x) := e^{-\frac{1}{2}\eps|t|^2-\frac{1}{2}\delta|x|^2}\int_\R\int_\R g_{\eps} (t-s) p_\delta(x-y)\xi(\dd y \dd s),
\end{equation}
where $p_t(x)=\frac{1}{\sqrt{2 \pi t}} e^{-x^2 /(2 t)}$ is the standard heat kernel. In the above mollification   the extra factor $ e^{-\frac{1}{2}\eps|t|^2-\frac{1}{2}\delta|x|^2}$  will be used later for the negative moments bound. With ``$*$'' denoting the convolution,
the covariance of $\xi_{\eps,\delta}$ is a nonnegative function
\begin{equation}\label{eq:qmollified}
\begin{aligned}
Q_{\eps,\delta}(t_1,t_2,x_1,x_2)&:= \Expe [\xi_{\eps, \delta}(t_1,x_1)\xi_{\eps, \delta}(t_2,x_2)] \\
&=  e^{-\frac{1}{2}\eps (|t_1|^2+|t_2|^2)} (g_{\eps}*g_{\eps})(t_1-t_2)e^{-\frac{1}{2}\delta (|x_1|^2+|x_2|^2)} p_{2\delta}(x_1-x_2),
\end{aligned}
\end{equation}
and $\xi_{\eps,\delta}(\cdot,\cdot) \in \mathcal{C}^{\infty}(\R^2,\R)$. One can check that for any $\eps,\delta>0$,
\[
\Expe \int_\R\int_\R |\xi_{\eps,\delta}(t,x)|^2\dd t\dd x = \int_\R\int_\R Q_{\eps,\delta}(t,t,x,x) \dd t\dd x <\infty,
\] so the noise $\xi_{\eps,\delta}(\cdot,\cdot)\in L^2(\R^2)$ $\Prob$-almost surely.

For any $\eps,\delta>0$ and $\zeta(\cdot)$ under Hypothesis~\ref{hy:detic}, define
\begin{equation}\label{eq:approxdef}
\begin{aligned}
&Z_\mu^{\eps, \delta}(t,x\viv s, \zeta) \\&:=\int_0^{\infty}\kernel^{N}(t,x\viv s,y)
\mathbb{E}^x_B\left[\exp\left(-\mu L^{0}_{t-s}\right):\exp: \left\{\int_{0}^{t-s}\xi_{\eps,\delta}(t-r,|B_r|)\dd r\right\}\bigg\mid |B_{t-s}|=y\right] \zeta(\dd y),
\end{aligned}
\end{equation}
with 
\[
\begin{aligned}
&:\exp: \left\{\int_{0}^{t-s}\xi_{\eps,\delta}(t-r,|B_r|)\dd r\right\}\\&\quad := \exp\left(
\int_{0}^{t-s}\xi_{\eps,\delta}(t-r,|B_r|)\dd r- \frac{1}{2}\int_0^{t-s}\int_0^{t-s} Q_{\eps,\delta}(t-r_1,t-r_2,|B_{r_1}|,|B_{r_2}|) \dd r_1 \dd r_2\right).
\end{aligned}
\]

It may be possible to interpret \eqref{eq:approxdef} as the random field solution to the following equation in the Skorohod sense, with $\diamond$ denoting the Wick product. For equation of Skorohod type on the full-space, we refer to \cite{hu2015stochastic}. Here we do not pursue this interpretation for our half-space problem.
\begin{equation*}
\begin{aligned}
\partial_t Z_\mu^{\eps, \delta}(t,x \viv s, \zeta) & =\frac{1}{2} \partial_x^2 Z_\mu^{\eps, \delta}(t,x \viv s, \zeta)+Z_\mu^{\eps, \delta}(t,x \viv s, \zeta) \diamond \xi_{\eps, \delta}(t, x),\\ \partial_x Z_\mu^{\eps, \delta}(t,x \viv s, \zeta)\big|_{x=0} & =\mu Z_\mu^{\eps, \delta}(t,0 \viv s, \zeta), \\
Z_\mu^{\eps, \delta}(s, \cdot\viv s, \zeta)& =\zeta(\cdot).
\end{aligned}
\end{equation*}

Let $\mathscr{L}^a_t(X)$ denote the (symmetric) local time at level $a$ and at time $t$ for any continuous semimartingale $X$, which means that,  with $\langle X \rangle_\cdot$ denoting the quadratic variation of $X$, we define
\[\mathscr{L}^a_t(X):= \lim _{\varepsilon \rightarrow 0} \frac{1}{2 \varepsilon} \int_0^t \1_{[a-\varepsilon, a+\varepsilon]}(X_s) \dd\langle X\rangle_s.\]
We are using the usual local time definition so that it is consistent with the local time of Brownian motions defined after \eqref{eq:probrep1}. When $X$ is the difference of two independent standard Brownian motions, this definition of local time is two times the local time  used in \cite{bertini1995stochastic}.

The following lemma is on the approximation of the solution $Z_\mu(t,x \viv s, \zeta)$. The expression in \eqref{eq:formalFK} can now be understood as the $L^p(\Omega)$ limit of $Z_\mu^{\eps,\delta}(t,x\viv s,\zeta)$ as $\eps,\delta \to 0$.
\begin{lemma}[Feynman-Kac approximation]\label{le:fk}
Let $\zeta(\cdot)$ be an initial condition satisfying Hypothesis~\ref{hy:detic}. The following holds for any $\mu,s,t \in \R$ with $s<t$, and $x\in [0,\infty)$:
\begin{enumerate}[(i).]
\item\label{it:approx} {By first taking $\delta \to 0$ then taking $\eps \to 0$,} the approximated solution $Z_\mu^{\eps,\delta}(t,x\viv s,\zeta)$ in \eqref{eq:approxdef} converges to $Z_\mu(t,x\viv s,\zeta)$ in $L^p(\Omega,\filt,\Prob)$ for any   $p \in[1,\infty)$, where $Z_\mu(t,x\viv s,\zeta)$ is the unique mild solution to \eqref{eq:hsSHEv2}. The convergence is uniform for $x \in [0,\infty)$, and for $t$ in any
compact subset of $(s,\infty)$.
\item\label{it:apprmoment} For any integer $p \in [2,\infty)$, let $B_t^{1},\dots,B_t^{p}$ be $p$ independent Brownian motions starting at $x$ and independent of $\xi$, with $L^{0,j}_{t}:=\mathscr{L}^0_t(B^j)$ for any $B^{j}$, $j=1,\cdots,p$, $t\geq 0$.
 For any $x \in [0,\infty)$,
\begin{equation}\label{eq:estmomentFK}
\begin{aligned}
&\Expe[Z_\mu^{\eps, \delta}(t,x \viv s, \zeta) ^p] =\int_{[0,\infty)}\dots \int_{[0,\infty)}\zeta(\dd y_1)\dots \zeta(\dd y_p)
\prod_{j=1}^p \kernel^{N}(t,x\viv s,y_j)\\
&\mathbb{E}^x_{B^{1},\dots,B^{p}}\left[e^{-\mu \sum_{j=1}^p L^{0,j}_{t-s}+\sum_{1\leq i<j \leq p} \int_0^{t-s}\int_0^{t-s}Q_{\eps,\delta} (t-r_1,t-r_2, |B^{i}_{r_1}|,|B^{j}_{r_2}|) \dd r_1 \dd r_2}\Bigg\mid |B^j_{t-s}|=y_j{\text{ for all } j}\right].
\end{aligned}
\end{equation} 
\item\label{it:momentsFK}
For any integer $p \in [2,\infty)$,
\begin{equation}\label{eq:fkmomentexpression}
\begin{aligned}
&\Expe[Z_\mu(t,x \viv s, \zeta)^p] =\int_{[0,\infty)}\dots \int_{[0,\infty)}\zeta(\dd y_1)\dots \zeta(\dd y_p)
\prod_{j=1}^p \kernel^{N}(t,x\viv s,y_j)\\
&\mathbb{E}^x_{B^{1},\dots,B^{p}}\left[ \exp(-\mu \sum_{j=1}^p L^{0,j}_{t-s})
\exp\left({\frac{1}{2}}\sum_{1\leq i<j \leq p}\mathscr{L}^0_{t-s}(|B^{i}|-|B^{j}|)\right)\Bigg\mid |B^j_{t-s}|=y_j{\text{ for all }j}
\right].\end{aligned}
\end{equation}
In particular, for any fixed 
$p \in [2,\infty)$, $a,b \in \R$ with $a<b$, $\mu_0\in\R$, there exists a constant $C=C(p,\mu_0,a,b)$ such that for any $x\in[0,\infty)$, $a\leq s< t\leq b$, and $\mu \in [\mu_0,\infty)$,
\begin{equation}\label{eq:momentbdsratio}
    \Expe[Z_\mu (t,x \viv s, \zeta) ^p] \leq C \left[ \int_0^{\infty} \kernel^{N}(t,x\viv s,y) \zeta(\dd y)\right]^p.
\end{equation}
\end{enumerate}
\end{lemma}
We first provide the next two useful results.
\begin{lemma}\label{le:bbmax}
    Let $B_{\cdot}$ be a standard Brownian motion. For any $\tau>0, \Lambda>0$, there exists a constant $C(\tau, \Lambda)>0$ such that for each $t\in [0,\tau], \lambda\in [0,\Lambda]$,
\begin{equation}\label{eq:bcLTbd}
\sup_{x,y,a \in \R}\mathbb{E}^x_{B}\left[\exp(\lambda \mathscr{L}^a_t(B))\mid B_t=y\right] \leq C(\tau,\Lambda).
\end{equation}
\end{lemma}
\begin{proof}
There are several ways to prove this property of local times. One way is to directly compute from the distribution of local times of a Brownian bridge (see \cite{pitman1999distribution}). An alternative is the arguments used in the proof of
\cite[Lemma 3.2]{bertini1995stochastic}. By approximating the local time with some positive definite functions, an argument similar to the one used in the last inequality in \eqref{eq:posdefdom} below can give another proof.
\end{proof}
\begin{lemma}\label{le:uniformbdFK}
    For any $\Lambda>0$ and $a<b$, there exists some uniform constant $C=C(\Lambda,a,b)$ such that for any two independent Brownian motions $B^1, B^2$ starting at $x\in[0,\infty)$,
\begin{equation}\label{eq:uniformbdFK}
    \sup_{\eps, \delta \in (0,1)}
    \mathbb{E}^x_{B^1,B^2}\left[\exp\left(\lambda\int_0^{t-s}\int_0^{t-s} Q_{\eps,\delta}(t-r_1,t-r_2,|B^{1}_{r_1}|,|B^{2}_{r_2}|) \dd r_1 \dd r_2 \right)\Bigg\mid |B^1_{t-s}|=y_1, |B^2_{t-s}|=y_2 \right]\leq C,
\end{equation}
for any $x, y_1, y_2 \in [0,\infty)$, $a\leq s<t\leq b$, and $\lambda \in [0,\Lambda]$.
\end{lemma}
\begin{proof}
We explain how to relate \eqref{eq:uniformbdFK} to similar estimates for the full-space SHE    and use the uniform bounds derived there.
We first note that, for each fixed $y_1,y_2 \in [0,\infty)$, the left-hand side of \eqref{eq:uniformbdFK} can be bounded above by a sum over four expectations conditioning on the events $\{B^i_{t-s}=\pm y_i, i=1,2\}$ respectively, i.e., 
\[
\begin{aligned}
&\mathbb{E}^x_{B^1,B^2}\left[\exp\left(\lambda\int_0^{t-s}\int_0^{t-s} Q_{\eps,\delta}(t-r_1,t-r_2,|B^{1}_{r_1}|,|B^{2}_{r_2}|) \dd r_1 \dd r_2 \right)\Bigg\mid |B^1_{t-s}|=y_1, |B^2_{t-s}|=y_2 \right] \\&\leq \sum_{i_1,i_2 =\pm 1} \mathbb{E}^x_{B^1,B^2}\left[\exp\left(\lambda\int_0^{t-s}\int_0^{t-s} Q_{\eps,\delta}(t-r_1,t-r_2,|B^{1}_{r_1}|,|B^{2}_{r_2}|) \dd r_1 \dd r_2 \right)\Bigg\mid B^1_{t-s}=i_1y_1, B^2_{t-s}=i_2y_2 \right].
\end{aligned}
\]

To ease notations, we use $\mathbb{E}^{x_1, x_2, y_1,y_2}_{B^1,B^2}$ as the expectation over independent Brownian bridges $B^1, B^2$ with $B^1_{0}=x_1, B^2_{0}=x_2, B^1_{t-s}=y_1$ and $B^2_{t-s}=y_2$ for any $x_1,x_2,y_1,y_2 \in \R$.

A straightforward observation is that, for any realization of $B^1,B^2$ and any $r_1,r_2\in [0,t-s]$,
\begin{equation*}
\begin{aligned}
&Q_{\eps,\delta}(t-r_1,t-r_2,|B^{1}_{r_1}|,|B^{2}_{r_2}|) \leq Q_{\eps,\delta}(t-r_1,t-r_2,|B^{1}_{r_1}|,|B^{2}_{r_2}|)\sum_{i_1,i_2=\pm 1}\1_{\{i_1B_{r_1} \geq 0, i_2B_{r_2} \geq 0 \}}(r_1,r_2)\\&
\leq \sum_{i_1,i_2 =\pm 1}Q_{\eps,\delta}(t-r_1,t-r_2,i_1B^{1}_{r_1},i_2B^{2}_{r_2}).
\end{aligned}
\end{equation*}
Thus by Cauchy-Schwarz inequality, for any $\eps,\delta>0$ and    $x_1,x_2,y_1,y_2 \in \R$,
    \begin{equation*}
        \begin{aligned}
        &\mathbb{E}^{x,x,y_1,y_2}_{B^1,B^2}\left[\exp\left(\lambda\int_0^{t-s}\int_0^{t-s} Q_{\eps,\delta}(t-r_1,t-r_2,|B^{1}_{r_1}|,|B^{2}_{r_2}|) \dd r_1 \dd r_2 \right) \right]\\
            &\leq \prod_{i_1, i_2 = \pm 1}\left\{\mathbb{E}^{x,x,y_1,y_2}_{B^1,B^2}\left[\exp\left(4\lambda\int_0^{t-s}\int_0^{t-s} Q_{\eps,\delta}(t-r_1,t-r_2,i_1B^{1}_{r_1},i_2B^{2}_{r_2}) \dd r_1 \dd r_2 \right) \right]\right\}^{1/4},
        \end{aligned}
    \end{equation*}
and \eqref{eq:uniformbdFK} will be proved if we can show that there exists $C=C(\Lambda,a,b)$, such that
\[
    \sup_{\eps, \delta \in (0,1)}\sup_{x_1,x_2,y_1,y_2 \in \R}
\mathbb{E}^{x_1,x_2,y_1,y_2}_{B^1,B^2}\left[\exp\left(4\lambda\int_0^{t-s}\int_0^{t-s} Q_{\eps,\delta}(t-r_1,t-r_2,B^{1}_{r_1},B^{2}_{r_2}) \dd r_1 \dd r_2 \right) \right]\leq C,
\]
    for any $a\leq s<t\leq b$, and $\lambda \in [0,\Lambda]$. 

To give a uniform bound over different starting points and endpoints, we note that for any $\eps,\delta>0$, $x_1,x_2,y_1,y_2\in\R$,
    \begin{equation}\label{eq:posdefdom}
    \begin{aligned}
&\mathbb{E}^{x_1,x_2,y_1,y_2}_{B^1,B^2}\left[\exp\left(4\lambda\int_0^{t-s}\int_0^{t-s} Q_{\eps,\delta}(t-r_1,t-r_2,B^{1}_{r_1},B^{2}_{r_2}) \dd r_1 \dd r_2 \right) \right]\\
        &\leq \mathbb{E}^{x_1,x_2,y_1,y_2}_{B^1,B^2}\left[\exp\left(4\lambda\int_0^{t-s}\int_0^{t-s} (g_{\eps}*g_{\eps})(r_1-r_2)  p_{2\delta}(B^{1}_{r_1}-B^{2}_{r_2}) \dd r_1 \dd r_2 \right) \right]\\
        & = \mathbb{E}^{0,0,0,0}_{B^1,B^2}\bigg[\exp\bigg(4\lambda\int_0^{t-s}\int_0^{t-s} (g_{\eps}*g_{\eps})(r_1-r_2) \\& p_{2\delta}\left(B^{1}_{r_1}-B^{2}_{r_2}+x_1-x_2+\frac{r_1(y_1-x_1)-r_2(y_2-x_2)}{t-s}\right) \dd r_1 \dd r_2 \bigg) \bigg]\\
        &\leq \mathbb{E}^{0,0,0,0}_{B^1,B^2}\left[\exp\left(4\lambda\int_0^{t-s}\int_0^{t-s} (g_{\eps}*g_{\eps})(r_1-r_2) p_{2\delta}\left(B^{1}_{r_1}-B^{2}_{r_2}\right) \dd r_1 \dd r_2 \right) \right],
    \end{aligned}
        \end{equation}
    where the equation follows from rescaling the Brownian bridges, and the last inequality follows from a generalization of \cite[Lemma 4.1]{huang2017large} (see also \cite[(4.5)]{huang2017large2}). Roughly speaking, the proof of the last inequality can be done by expanding the exponential, applying Fourier transforms to the Gaussians, and using the fact that $p_{2\delta}$ is a positive definite function for any $\delta>0$. Heuristically, this is to say that an approximated ``local time at level $0$'' of Brownian bridges starting from $x_1-x_2$ and ending at $y_1-y_2$ would reach its maximum when $x_1-x_2=0$ and $y_1-y_2=0$. 

Combining the results above, we have reduced the problem of showing \eqref{eq:uniformbdFK} to showing that for any $a\leq s<t\leq b$, and $\lambda \in [0,\Lambda]$,
\[
\sup_{\eps, \delta \in (0,1)}\mathbb{E}^{0,0,0,0}_{B^1,B^2}\left[\exp\left(4\lambda\int_0^{t-s}\int_0^{t-s} (g_{\eps}*g_{\eps})(r_1-r_2) p_{2\delta}\left(B^{1}_{r_1}-B^{2}_{r_2}\right) \dd r_1 \dd r_2 \right) \right] \leq C(\Lambda, a,b),
\]
where the Brownian bridges no longer have reflections. This uniform bound would follow from the same arguments as for a similar result \cite[(3.17)]{hu2015stochastic}. Compared with \cite[(3.17)]{hu2015stochastic}, a few adaptions are needed for the noise $\xi$ being white in time (instead of colored as in \cite[(3.17)]{hu2015stochastic}) and $B^1, B^2$ being Brownian bridges (instead of Brownian motions as in \cite[(3.17)]{hu2015stochastic}). The adaptation for time independence is straightforward from the computations. For the adaption to Brownian bridges, we refer to \cite[Proposition 4.2]{huang2017large}, where essential estimates have been proved.
\end{proof}

We now prove Lemma~\ref{le:fk}.

\begin{proof}[Proof of Lemma~\ref{le:fk}] A direct Gaussian calculation gives \eqref{it:apprmoment}. To prove \eqref{it:approx} and \eqref{eq:fkmomentexpression}, one can follow the same arguments as used in \cite[Theorem 3.6]{hu2015stochastic} for the analogous full-space SHEs. In particular, the proof mainly needs two estimates: (i) passing to the limit: for each $p \in [1,\infty)$,
\begin{equation}\label{eq:fklimitcon}
\begin{aligned}
        \lim _{\eps,\delta \to 0} \mathbb{E}^x_{B^1,B^2}\Bigg[&\Bigg|\int_0^{t-s}\int_0^{t-s} Q_{\eps,\delta}(t-r_1,t-r_2,|B^{1}_{r_1}|,|B^{2}_{r_2}|) \dd r_1 \dd r_2 \\& -\frac{1}{2} \mathscr{L}_{t-s}^0\left(\left|B^1\right|-\left|B^2\right|\right)\Bigg|^p\Bigg\mid |B^1_{t-s}|=y_1, |B^2_{t-s}|=y_2 \Bigg]=0,
\end{aligned}
\end{equation}
uniformly for $x,y_1,y_2\in [0,\infty)$ and $t$ in compact subset of $(s,\infty)$; (ii) a uniform bound: for any $p \in[1,\infty)$,
\begin{equation}\label{eq:keyestapprox}
\begin{aligned}
     &\sup_{\eps,\delta\in(0,1)} \sup_{x, y_1,\dots,y_p \in [0,\infty)}\mathbb{E}^x_{B^{1},\dots,B^{p}}\Bigg[\exp({-\mu \sum_{j=1}^p L^{0,j}_{t-s}})\\& \quad
\exp\left(\sum_{1\leq i<j \leq p} \int_0^{t-s}\int_0^{t-s}Q_{\eps,\delta} (t-r_1,t-r_2, |B^{i}_{r_1}|,|B^{j}_{r_2}|) \dd r_1 \dd r_2\right)\bigg\mid |B^j_{t-s}|=y_j\text{ for all }j\Bigg]
\leq C(p,\mu,s,t). 
\end{aligned}
\end{equation}
The limit \eqref{eq:fklimitcon} follows from the continuity of the semimartingale $|B^1|-|B^2|$, the density of occupation time formula \cite[Corollary 9.7]{le2016brownian}, and the fact that the local time process $\mathscr{L}^a_t{(|B^1|-|B^2|)}$ is continuous in $a\in \R$ almost surely (which can be verified using e.g. \cite[Theorem 9.4]{le2016brownian}).  The uniform bound \eqref{eq:keyestapprox} follows from H\"older's inequality applied to the conditional expectation and applying Lemma~\ref{le:bbmax} and Lemma~\ref{le:uniformbdFK}.

The proof also shows that the right-hand side of \eqref{eq:fkmomentexpression} is the limit of the right-hand side of \eqref{eq:estmomentFK} as $\eps,\delta \to 0$.
To prove \eqref{eq:momentbdsratio}, we apply H\"older's inequality to the conditional expectation in \eqref{eq:estmomentFK}, and use Lemma~\ref{le:bbmax} and Lemma~\ref{le:uniformbdFK} with the convergence. From the uniform bounds in Lemma~\ref{le:bbmax} and Lemma~\ref{le:uniformbdFK}, it is not hard to see that the constant in \eqref{eq:momentbdsratio} is uniform for $x\in[0,\infty)$, $\mu \in [\mu_0,\infty)$ and $a\leq s<t\leq b$.
\end{proof}

\subsection{Proof of Lemma~\ref{le:negmom}}\label{ap:negmom}
In this section, we provide a proof of the negative moments bounds {for the SHEs starting from constant or Dirac-$\delta$ initial data}  following the argument in \cite{hu2022asymptotics}. 


Using the notations as in Appendix~\ref{ap:fk}, {we define} 
\begin{equation}\label{eq:shemuchaosexpand}
\begin{aligned}
&\Phi^{\eps,\delta}_{\mu}(t,x\viv 0,\zeta):= \int_0^{\infty}\int_0^{\infty}\kernel^{N}(t,x\viv 0,y_1) \kernel^{N}(t,x\viv 0,y_2) \mathbb{E}^x_{B^1,B^2}\Bigg[ \exp(-\mu L^{0,1}_{t}-\mu L^{0,2}_{t}) \\& \exp\left(\int_0^t\int_0^tQ_{\eps,\delta} (t-s,t-r, |B^{1}_{r}|,|B^{2}_{s}|)\dd s \dd r+\frac{1}{2}\mathscr{L}_t^0(|B^{1}|-|B^{2}|)\right) \bigg\mid |B^1_{t}|=y_1, |B^2_{t}|=y_2  \Bigg]  \zeta(\dd y_1)\zeta(\dd y_2).
\end{aligned}
\end{equation}
Recall that $z_\mu(t,x\viv 0,\zeta):= \int_0^{\infty}\kernel_\mu^R(t,x\viv0,y)\zeta(\dd y).$
We first prove the following auxiliary result:
\begin{proposition}\label{pr:negmomcondition}
    Let $\zeta$ be the constant initial condition  or the Dirac-$\delta$ initial condition $\zeta=\delta_y$ for some $y\in [0,\infty)$. For any $\tau>0$, there exists some positive constant $C=C(\mu,\tau)$ such that
\begin{equation}\label{eq:condwithepsdelta}
    \sup_{x\in[0,\infty)} \sup_{\eps,\delta \in (0,1)} \frac{\Expe[Z_\mu^{\eps, \delta}(t,x \viv 0, \zeta) ^2]}{z_\mu(t,x\viv0,\zeta)^2} \leq C(\mu,\tau),
\end{equation}
and
\begin{equation}\label{eq:bdforPhi}
     \sup_{x\in[0,\infty)} \sup_{\eps,\delta \in (0,1)} \frac{\Phi^{\eps,\delta}_{\mu}(t,x\viv 0,\zeta)}{z_\mu(t,x\viv0,\zeta)^2}\leq C(\mu,\tau),
\end{equation}
for any $t\in (0,\tau]$. 
The constant $C(\mu,\tau)$ can   be chosen uniformly for $\mu$ in any compact subset of $\R$.

\end{proposition}
\begin{proof}
    The proof is essentially applying H\"older's inequality to the conditional expectations and using the uniform bounds in Lemma~\ref{le:bbmax} and Lemma~\ref{le:uniformbdFK}. These estimates would give us \eqref{eq:condwithepsdelta} and \eqref{eq:bdforPhi} with $z_\mu(t,x\viv0,\zeta)^2$ replaced by $z_0(t,x\viv 0, \zeta)^2$ on the denominators. Note that when $\zeta=\delta_y$, $z_0(t,x\viv 0, \zeta)=\kernel^N(t,x\viv 0, y)$. Another estimate useful here is that following Lemma~\ref{le:bbmax}, for any $\tau>0, \mu_0 \in \R$, there exists a constant $C(\mu_0, \tau)>0$ such that
    \begin{equation}\label{eq:ratiokernels}
    \sup_{t\in(0,\tau]}\sup_{\mu\leq\mu_0} 
    \frac{z_0(t,x\viv 0,\delta_y)}{z_\mu(t,x\viv 0,\delta_y)} = \sup_{t\in(0,\tau]}\sup_{\mu\leq \mu_0} \frac{1}{\mathbb{E}^{x}_{B}[\exp(-\mu L_t^0)\mid |B_t|=y]} \leq C(\mu_0, \tau),
    \end{equation}
    for all $x,y \in [0,\infty)$. Similarly, we also have that
    \[
\sup_{t\in(0,\tau]}\sup_{\mu\leq\mu_0} 
    \frac{z_0(t,x\viv 0,\1)}{z_\mu(t,x\viv 0,\1)} = \sup_{t\in(0,\tau]}\sup_{\mu\leq \mu_0} \frac{1}{\mathbb{E}^{x}_{B}[\exp(-\mu L_t^0)]} \leq C'(\mu_0, \tau),
    \]
    for all $x \in [0,\infty)$ with some constant $C'(\mu_0, \tau)>0$.
\end{proof}


Following  \cite{hu2022asymptotics}, we now present the proof of Lemma~\ref{le:negmom}. 
{For clarity of  presentation, we give the argument only in the case where the initial condition $\zeta$ admits a continuous density $f(\cdot)$ with respect to the Lebesgue measure, since in this setting the Feynman–Kac approximations take a simpler form.
Under the assumption on the initial data in Lemma~\ref{le:negmom}, this corresponds to the case where 
 $\zeta$ is  constant (that is $f(\cdot)=\mathbf{1}(\cdot)$).} 

{However, since the Feynman-Kac approximation in Lemma~\ref{le:fk} is also valid for Dirac-$\delta$ initial conditions, one can easily adapt the following proof, requiring only minor modifications to the Feynman-Kac expressions.
To make the argument fully rigorous in the Dirac cases, these expressions need to be rewritten to reflect the singular nature of the initial data, but the underlying steps of the proof carry over.
 Formally, one can simply check by  taking 
$f (\cdot) = \delta_y$
for any 
$y\geq 0$. } 

{We also note that Lemma~\ref{le:negativelemma1} below is independent of the choice of initial condition and holds for any 
$\zeta$ satisfying Hypothesis~\ref{hy:detic}.}

{Now suppose $\zeta(\dd y)= f(y) \dd y $ for some deterministic function $f\in \mathcal{C}([0,\infty), [0,\infty))$.} For any fixed $\mu \in \R, t> 0$ and any $L^2(\R^2)$-valued sample path $\xi_{\eps,\delta}$ of the smoothed space-time noise \eqref{eq:smstnoise}, define \[
\Theta(B_t, \xi_{\eps,\delta}) := f(|B_t|)\exp(-\mu L^0_t):\exp: \left\{\int_{0}^{t}\xi_{\eps,\delta}(t-r,|B_r|)\dd r\right\}.\]
$\Theta(B_t, \xi_{\eps,\delta})$
 is a continuous functional of the Brownian motion $B_t$ starting from $x \in [0,\infty)$. For any bounded functional $F$ of the Brownian motion $B$, we define its weighted expectation as
\[
\mathbb{E}^B_{x,\xi_{\eps,\delta}}[F] =  \frac{\mathbb{E}^x_{B}\left[F(\cdot) \Theta(B_t, \xi_{\eps,\delta})\right]}{\mathbb{E}^x_{B}\left[ \Theta(B_t, \xi_{\eps,\delta})\right]} = \frac{\mathbb{E}^x_B\left[F(\cdot) f(|B_t|)\exp(-\mu L^0_t):\exp: \left\{\int_{0}^{t}\xi_{\eps,\delta}(t-r,|B_r|)\dd r\right\}\right]}{\mathbb{E}^x_B\left[f(|B_t|)\exp(-\mu L^0_t):\exp: \left\{\int_{0}^{t}\xi_{\eps,\delta}(t-r,|B_r|)\dd r\right\}\right]}.
\]
For any bounded functional $F$ of two paths $B^1,B^2$, we define 
\[
\begin{aligned}
&\mathbb{E}^{B^1,B^2}_{x,\xi_{\eps,\delta}}[F] = \frac{\mathbb{E}^x_{B^1,B^2}\left[F(\cdot) \Theta(B^{1}_t, \xi_{\eps,\delta})\Theta(B^{2}_t, \xi_{\eps,\delta}) \right]}{\mathbb{E}^x_{B^1,B^2}\left[\Theta(B^{1}_t, \xi_{\eps,\delta})\Theta(B^{2}_t, \xi_{\eps,\delta}) \right]}.
\end{aligned}
\]
To ease notations, we write
\[Z_\mu(t,x,\xi_{\eps,\delta}):=\mathbb{E}^x_B\left[f(|B_t|)\exp\left(-\mu L^0_t\right):\exp: \left\{\int_{0}^{t}\xi_{\eps,\delta}(t-r,|B_r|)\dd r\right\} \right]
\]
to emphasize the dependence of $Z^{\eps,\delta}_\mu(t,x\viv0,\zeta)$ on the noise $\xi_{\eps,\delta}$. 
For any fixed Brownian path $B_t$,
\[
\Expe\left[:\exp: \left\{\int_{0}^{t}\xi_{\eps,\delta}(t-r,|B_r|)\dd r\right\}\right]=1.
\]
Tonelli's theorem gives $
\Expe\left[Z_\mu(t,x, \xi_{\eps,\delta})\right] =z_\mu(t,x\viv0,\zeta)$. As above, we use $\mathscr{L}^a_t(\cdot)$ to denote the local time at level $a \in \R$ up to time $t\geq 0$ for continuous semimartingales.

Now for each $\lambda>0, (t,x)\in(0,\infty)\times[0,\infty)$, define the event
\begin{equation}\label{eq:alambda}
\begin{aligned}
A_{\lambda}(t,x) &= \Bigg\{\xi'_{\eps,\delta}: Z_\mu(t,x, \xi'_{\eps,\delta})> \frac{1}{2}z_\mu(t,x\viv0,\zeta),  \quad   \mathbb{E}^{B^1,B^2}_{x,\xi'_{\eps,\delta}}\left[{\frac{1}{2}}\mathscr{L}_t^0(|B^{1}|-|B^{2}|)\right] \leq \lambda \Bigg\},
\end{aligned}
\end{equation}
and a distance
\[
d(\xi_{\eps,\delta}, A_{\lambda}(t,x)):= \inf\{ \|\xi_{\eps,\delta}-\xi'_{\eps,\delta}\|_{L^2(\R^2)}: \xi'_{\eps,\delta} \in A_{\lambda}(t,x) \}.
\]
We first prove the following lemmas.

\begin{lemma}\label{le:negativelemma1}
For any sample path $\xi_{\eps,\delta}$ and any $\lambda>0$,
\[
Z_\mu(t,x, \xi_{\eps,\delta}) \geq \frac{1}{2}z_\mu(t,x\viv0,\zeta)\exp\left[ - \sqrt{\lambda}d(\xi_{\eps,\delta}, A_{\lambda}(t,x))\right].
\]
\end{lemma}
\begin{proof}
Fix any  $\xi'_{\eps,\delta} \in A_{\lambda}(t,x)$. By   definitions and Jensen's inequality, we have
\begin{equation}\label{eq:negmomJensen}
\begin{aligned}
Z_\mu(t,x, \xi_{\eps,\delta}) &= Z_\mu(t,x, \xi'_{\eps,\delta}) \mathbb{E}^{B}_{x,\xi'_{\eps,\delta}}\left[\frac{:\exp: \left\{\int_{0}^{t}\xi_{\eps,\delta}(t-r,|B_r|)\dd r\right\}}{:\exp: \left\{\int_{0}^{t}\xi'_{\eps,\delta}(t-r,|B_r|)\dd r\right\}} \right]\\
&=Z_\mu(t,x, \xi'_{\eps,\delta}) \mathbb{E}^{B}_{x,\xi'_{\eps,\delta}}\left\{\exp\left[ \int_{0}^{t}\xi_{\eps,\delta}(t-r,|B_r|)\dd r - \int_{0}^{t}\xi'_{\eps,\delta}(t-r,|B_r|)\dd r  \right] \right\}&\\
 &\geq  Z_\mu(t,x, \xi'_{\eps,\delta}) \exp\left\{\mathbb{E}^{B}_{x,\xi'_{\eps,\delta}}\left[ \int_{0}^{t}\xi_{\eps,\delta}(t-r,|B_r|)\dd r - \int_{0}^{t}\xi'_{\eps,\delta}(t-r,|B_r|)\dd r  \right]\right\}.
\end{aligned}
\end{equation}
The exponent can be estimated as
\begin{equation}\label{eq:pathintegralbd}
\begin{aligned}
& \left|\mathbb{E}^{B}_{x,\xi'_{\eps,\delta}}\left[ \int_{0}^{t}\xi_{\eps,\delta}(t-r,|B_r|)\dd r - \int_{0}^{t}\xi'_{\eps,\delta}(t-r,|B_r|)\dd r  \right]\right|\\
&\leq \mathbb{E}^{B}_{x,\xi'_{\eps,\delta}}\left[ \int_{0}^{t}\left|\xi_{\eps,\delta}(t-r,|B_r|)- \xi'_{\eps,\delta}(t-r,|B_r|)\right|\dd r  \right] \\&\leq \left[ {\frac{1}{2}}\mathbb{E}^{B^1,B^2}_{x,\xi'_{\eps,\delta}} \left[\mathscr{L}_t^0( |B^{1}|-|B^{2}|) \right]\right]^{1/2} \|\xi_{\eps,\delta}-\xi'_{\eps,\delta}\|_{L^2(\R^2)}.\end{aligned}\end{equation}
The last ``$\leq$'' will be proved at the end. Since $\xi'_{\eps,\delta} \in A_{\lambda}(t,x)$, by the second inequality in \eqref{eq:alambda},  the above expression  is  further bounded by $\sqrt{\lambda} \|\xi_{\eps,\delta}-\xi'_{\eps,\delta}\|_{L^2(\R^2)}$. Therefore, we have
\[Z_\mu(t,x, \xi_{\eps,\delta}) \geq Z_\mu(t,x, \xi'_{\eps,\delta}) \exp(-\sqrt{\lambda} \|\xi_{\eps,\delta}-\xi'_{\eps,\delta}\|_{L^2(\R^2)}).
\]
The proof will be complete when we combine the above with the first inequality in \eqref{eq:alambda} and take the supremum over all $\xi'_{\eps,\delta} \in A_\lambda(t,x)$ on the right-hand side.

It remains to prove the last inequality in \eqref{eq:pathintegralbd}. Formally, this inequality follows from the density of occupation time formula (see e.g. \cite[Corollary 9.7]{le2016brownian}) and H\"older's inequality on $\R^2$. However, since the local time process $\frac{1}{2}\mathscr{L}^a_t$ is not differentiable in $t$, we will prove it by approximation.

Set $\Delta \xi (t,x) :=|\xi'_{\eps,\delta}(t,x)- \xi_{\eps,\delta}(t,x)|$ for $(t,x)\in\R^2$. Since $\Delta\xi$ is deterministic, continuous and in $L^2(\R^2)$, the second line of \eqref{eq:pathintegralbd} can be approximated by a family of convolutions
\[\mathbb{E}^{B}_{x,\xi'_{\eps,\delta}}\left[ \int_{0}^{t} \Delta \xi (t-r,|B_r|)\dd r\right]  = \lim_{\kappa\to0} \mathbb{E}^{B}_{x,\xi'_{\eps,\delta}}\left[ \int_{0}^{t} \int_\R \Delta \xi (t-r,y)p_\kappa(|B_r|-y)\dd y \dd r\right].
\]
For each $\kappa>0$, by Tonelli's theorem and H\"older's inequality,
\[
\begin{aligned}
& \mathbb{E}^{B}_{x,\xi'_{\eps,\delta}}\left[ \int_{0}^{t} \int_\R \Delta \xi (t-r,y)p_\kappa(|B_r|-y)\dd y \dd r\right]=\int_{0}^{t} \int_\R \Delta \xi (t-r,y) \mathbb{E}^{B}_{x,\xi'_{\eps,\delta}}\left[p_\kappa(|B_r|-y)\right]\dd y \dd r \\&\leq \left[\int_0^t\int_\R \left(\mathbb{E}^{B}_{x,\xi'_{\eps,\delta}} \left[p_\kappa(|B_r|-y)\right]\right)^2\dd y \dd r \right]^{1/2} \|\Delta\xi\|_{L^2(\R^2)}
\\&=  \left[\int_0^t\int_\R \left(\mathbb{E}^{B^1, B^2}_{x,\xi'_{\eps,\delta}} \left[p_\kappa(|B^1_r|-y)p_\kappa(|B^2_r|-y)\right]\right)\dd y \dd r \right]^{1/2} \|\Delta\xi\|_{L^2(\R^2)}\\
& =  \left[\mathbb{E}^{B^1, B^2}_{x,\xi'_{\eps,\delta}} \left( \int_0^t p_{2\kappa}(|B^1_r|-|B^2_r|)\dd r \right) \right]^{1/2} \|\Delta\xi\|_{L^2(\R^2)}.
\end{aligned}\]
By the density of occupation time formula, {since the quadratic variation of the semimartingale $|B^1_r|-|B^2_r|$ equals to $2r$,}
\[\int_0^t p_{2\kappa}(|B^1_r|-|B^2_r|)\dd r ={\frac{1}{2} }\int_\R p_{2\kappa}(a)\mathscr{L}_t^a(|B^1|-|B^2|) \dd a.
\] The proof is complete since \[\int_\R p_{2\kappa}(a) \mathbb{E}^{B^1, B^2}_{x,\xi'_{\eps,\delta}} \left[\frac{1}{2}\mathscr{L}_t^a(|B^1|-|B^2|)\right] \dd a \to \mathbb{E}^{B^1, B^2}_{x,\xi'_{\eps,\delta}} \left[\frac{1}{2}\mathscr{L}_t^0(|B^1|-|B^2|)\right], \quad \text{as } \kappa \to 0.\]
\end{proof}

We next show that the probability of  $\xi'_{\eps,\delta} \in A_{\lambda}(t,x)$ is always positive.
\begin{lemma}\label{le:probofAset} For any $\mu \in \R$, $\tau>0$,
there exists some constant $\lambda=\lambda(\mu,\tau)$ such that for any $x \in [0,\infty)$, $t \in (0,\tau]$ and $\zeta$ addressed in Proposition~\ref{pr:negmomcondition},
\[
\Prob(A_{\lambda}(t,x)) \geq \frac{1}{8}\frac{1}{C(\mu,\tau)}=:b(\mu,\tau),\]
where $C(\mu,\tau)$ is the constant in \eqref{eq:condwithepsdelta} and \eqref{eq:bdforPhi}.
\end{lemma}
\begin{proof}
From \eqref{eq:alambda}, we see that the set of noise
\[
\begin{aligned}
& \Bigg\{\xi'_{\eps,\delta}: Z_\mu(t,x, \xi'_{\eps,\delta})> \frac{1}{2} z_\mu(t,x\viv0,\zeta),\\& \quad 
\mathbb{E}^x_{B^1,B^2}\left[ \frac{1}{2}\mathscr{L}_t^0(|B^{1}|-|B^{2}|) \Theta(B^{1}_t, \xi'_{\eps,\delta})\Theta(B^{2}_t, \xi'_{\eps,\delta})  \right] \leq  \frac{\lambda}{4}z_\mu(t,x\viv0,\zeta)^2 \Bigg\}
\end{aligned}
\]
is a subset of $A_{\lambda}(t,x)$. Thus we have that 
\begin{equation}\label{eq:probofaset}
\Prob(A_{\lambda}(t,x)) \geq \Prob\left( Z_\mu(t,x, \xi'_{\eps,\delta})> \frac{1}{2}z_\mu(t,x\viv0,\zeta) \right) - \Prob(B_{\lambda}(t,x)),
\end{equation}
where
\[
\begin{aligned}
B_{\lambda}(t,x) = \Bigg\{\xi'_{\eps,\delta}:   \mathbb{E}^x_{B^1,B^2}\left[ \frac{1}{2}\mathscr{L}_t^0(|B^{1}|-|B^{2}|) \Theta(B^{1}_t, \xi'_{\eps,\delta})\Theta(B^{2}_t, \xi'_{\eps,\delta})  \right] >  \frac{\lambda}{4}z_\mu(t,x\viv0,\zeta)^2    \Bigg\}.
\end{aligned}
\]
{By the Paley–Zygmund inequality and \eqref{eq:condwithepsdelta}, when $\eps,\delta \in (0,1)$, for any $t \in (0,\tau]$, $x \in [0,\infty)$,
\begin{equation}\label{eq:paleyzygmund}
\begin{aligned}
\Prob\left( Z_\mu(t,x, \xi'_{\eps,\delta})> \frac{1}{2}z_\mu(t,x\viv0,\zeta)\right) \geq \frac{1}{4} \frac{z_\mu(t,x\viv0,\zeta)^2 }{\Expe\left[Z_\mu(t,x, \xi'_{\eps,\delta})^2\right]} \geq \frac{1}{4} \frac{1}{C(\mu,\tau)},
\end{aligned}
 \end{equation}
 where $C(\mu,\tau)$ is the constant in Proposition~\ref{pr:negmomcondition}.}
To estimate $\Prob(B_{\lambda}(t,x))$, by Tonelli's theorem and Lemma~\ref{le:fk}~\eqref{it:apprmoment},
 \[
 \begin{aligned} & \Expe \mathbb{E}^x_{B^1,B^2}\left[ \frac{1}{2}\mathscr{L}_t^0(|B^{1}|-|B^{2}|) \Theta(B^{1}_t, \xi'_{\eps,\delta})\Theta(B^{2}_t, \xi'_{\eps,\delta})  \right]  \\
 &=  \mathbb{E}^x_{B^1,B^2}\left[ \frac{1}{2}\mathscr{L}_t^0(|B^{1}|-|B^{2}|)\Expe[ \Theta(B^{1}_t, \xi'_{\eps,\delta})\Theta(B^{2}_t, \xi'_{\eps,\delta}) ] \right]  \\
 & = \mathbb{E}^x_{B^1,B^2}\bigg[ \frac{1}{2}\mathscr{L}_t^0(|B^{1}|-|B^{2}|)f\left(|B^{1}_t|\right) f\left(|B^{2}_t|\right) 
 \\&\quad\exp(-\mu L^{0,1}_{t}-\mu L^{0,2}_{t})
\exp\left(\int_0^t\int_0^tQ_{\eps,\delta} (t-s,t-r, |B^{1}_{r}|,|B^{2}_{s}|) \dd s \dd r\right)\bigg].
 \end{aligned}
 \]

Using a trivial bound $a \leq \exp(a)$ for all $a\geq 0$, the last expression is bounded above by the value $\Phi^{\eps,\delta}_{\mu}(t,x\viv 0,\zeta)$ defined in \eqref{eq:shemuchaosexpand}.
By Chebyshev inequality and \eqref{eq:bdforPhi}, when $\eps,\delta \in (0,1)$, for any $t \in (0,\tau]$, $x \in [0,\infty)$,
\[
\Prob(B_{\lambda}(t,x)) \leq \frac{4 \Phi^{\eps,\delta}_{\mu}(t,x\viv 0,\zeta)}{\lambda z_\mu(t,x\viv0,\zeta)^2} \leq \frac{4C(\mu,\tau)}{\lambda}.
\]
Together with \eqref{eq:probofaset} and \eqref{eq:paleyzygmund}, we have that for any $x \in [0,\infty)$, $t \in (0,\tau]$,
\[
\Prob(A_{\lambda}(t,x)) \geq \frac{1}{4}\frac{1}{C(\mu,\tau)} - \frac{4C(\mu,\tau)}{\lambda}.
\]
Let
\[
\lambda_0(\mu,\tau) := 32  C(\mu,\tau)^2.
\]
By Proposition~\ref{pr:negmomcondition}, $\lambda_0(\mu,\tau)<\infty$. Choose some $\lambda =\lambda(\mu,\tau)> \lambda_0(\mu,\tau)$, then for any $x \in [0,\infty)$, $t \in (0,\tau]$, \[
\Prob(A_{\lambda}(t,x))  \geq \frac{1}{8}\frac{1}{C(\mu,\tau)}>0.
\]\end{proof}

With Lemma~\ref{le:probofAset}, we can further bound the probability of any sample path $\xi_{\eps,\delta}$ deviating from the set $A_{\lambda}(t,x)$: for any $t\in (0,\tau], x \in [0,\infty)$ and $a>0$,
\begin{equation}\label{eq:distanceprobbound}
\Prob \left(d(\xi_{\eps,\delta}, A_{\lambda}(t,x))\geq a+2\sqrt{\log \frac{2}{b(\mu,\tau)}} \right) \leq 2e^{-a^2/4}.
\end{equation}
This is derived from Talagrand's concentration inequality. The proof only depends on estimates of the space-time noise $\xi_{\eps,\delta}$ and the lower bound on the probability of event $A_{\lambda}(t,x)$ from Lemma~\ref{le:probofAset}. A technical (though not difficult) part here is that one needs to use an $L^2(\R^2)$ approximation $\xi_{\eps,\delta,n}$ for $\xi_{\eps,\delta}$ and prove concentration at each $n$. Since we are using the same $\xi_{\eps,\delta}$ as in the (1+1) full-space SHEs, we refer to \cite[Lemma 4.5]{hu2022asymptotics} for the proof of \eqref{eq:distanceprobbound}. 

We are now ready to prove Lemma~\ref{le:negmom}.

\begin{proof}[Proof of Lemma~\ref{le:negmom}]
By Lemma~\ref{le:negativelemma1}, for any fixed $\mu \in \R, t\in (0,\tau], x \in [0,\infty)$, we have
\[
Z_\mu(t,x, \xi_{\eps,\delta}) \geq \frac{1}{2} z_\mu(t,x\viv0,\zeta)\exp\left[ - \sqrt{\lambda}d(\xi_{\eps,\delta}, A_{\lambda}(t,x))\right].
\]
Applying \eqref{eq:distanceprobbound}, for any $\lambda>\lambda_0(\mu,\tau)$, $a>0$, \[
\Prob \left(\frac{Z_\mu(t,x, \xi_{\eps,\delta}) }{z_\mu(t,x\viv0,\zeta)}\leq \frac{1}{2}\exp\left(-\sqrt{\lambda}\left[a+2\sqrt{\log \frac{2}{b(\mu,\tau)}}\right]\right) \right) \leq 2e^{-a^2/4}.
\]
As $\lambda,b$ are independent of $\eps, \delta$, by Lemma~\ref{le:fk}~\eqref{it:approx}, by sending $\eps, \delta \to 0$, we have
\[
\Prob \left(F(t,x)\leq \exp\left(-\sqrt{\lambda}a\right) \right) \leq 2e^{-a^2/4}, \quad \text{for }
F(t,x):=\frac{Z_\mu(t,x\viv 0, \zeta) }{z_\mu(t,x\viv0,\zeta)}2 \exp\left(2\sqrt{\lambda\log\frac{2}{b(\mu,\tau)}} \right).
\]
As in \cite[Corollary 4.8]{hu2022asymptotics}, for any $p\geq0$, we compute
\[
\begin{aligned}
&\Expe F^{-p} = p\int_0^{\infty}r^{-p} \Prob(F<r)\frac{\dd r}{r} \leq p\int_0^1r^{-p} \Prob(F<r)\frac{\dd r}{r} +  p\int_1^{\infty}r^{-p}\frac{\dd r}{r} \\
&= p \sqrt{\lambda}\int_0^{\infty} \exp(p\sqrt{\lambda}a)\Prob \left(F<\exp(-\sqrt{\lambda}a) \right)\dd a +1 \leq p \sqrt{\lambda}\int_0^{\infty} 2\exp(p\sqrt{\lambda}a-\frac{a^2}{4})\dd a +1 \\ &\leq 4\sqrt{\pi p^2\lambda}\exp(p^2\lambda)+1.
\end{aligned}
\]
Thus we have that for any $t\in (0,\tau], x \in [0,\infty)$,
\[
\Expe \left[\left(Z_\mu(t,x\viv 0, \zeta)  \right)^{-p}\right] \leq  z_\mu(t,x\viv0,\zeta)^{-p} 2^{p} \exp\left(2p\sqrt{\lambda\log\frac{2}{b(\mu,\tau)}} \right)\left(4\sqrt{\pi p^2\lambda}\exp(p^2\lambda)+1\right),
\]
which implies \eqref{eq:negativemoments}.

Additionally, Proposition~\ref{pr:negmomcondition} has assured that the constant $C(\mu,\tau)$ in \eqref{eq:condwithepsdelta} and \eqref{eq:bdforPhi} can be chosen uniformly over $\mu$ in compact subsets of $\R$. It thus implies that $\lambda_0(\mu,\tau)$ and $b(\mu,\tau)^{-1}$ can also be uniform in the corresponding sets of $\mu$.
\end{proof}

\subsection{Proof of Proposition~\ref{pr:greens}~\eqref{it:trev}}\label{ap:reversal}
Let $\xi_{\delta}(t,x) = \int_\R  p_\delta(x-y)\xi(t, \dd y)$, which is a Gaussian noise white in time and smooth in space, with the covariance function
\[
Q_{\delta}(t_1,t_2,x_1,x_2) = \delta_0(t_1-t_2) p_{2\delta}(x_1-x_2).
\]
For any $\delta>0$, $\mu, s,t \in \R, t>s$, $x,y \in [0,\infty)$, define \[
    \mathcal{Z}_\mu^{\delta}(t,x\viv s, y) := \kernel^{N}(t,x\viv s,y)
\mathbb{E}^x_B\left[\exp\left(-\mu L^{0}_{t-s} +
\int_{0}^{t-s}\xi_{\delta}(t-r,|B_r|)\dd r- \frac{1}{2}(t-s)p_{2\delta}(0)\right)\bigg||B_{t-s}|=y\right]\]
Let $\Upsilon_\cdot = |B_\cdot|$ be the reflected Brownian motion starting from $\Upsilon_0=|B_0|=x$. We use $L^{0,\Upsilon}_{t}$ to denote the local time
\[L^{0,\Upsilon}_{t}:=\lim _{\eps \to 0} \frac{1}{2 \varepsilon} \int_0^t \1_{[0, \varepsilon]}\left(\Upsilon_s\right) \dd s.
\]
As discussed after \eqref{eq:probrep3}, we can rewrite the above as
\[\mathcal{Z}_\mu^{\delta}(t,x\viv s, y) = 
\kernel^{N}(t,x\viv s,y)
\mathbb{E}^x_{\Upsilon}\left[\exp\left(-\mu L^{0,\Upsilon}_{t-s} +
\int_{0}^{t-s}\xi_{\delta}(t-r,\Upsilon_r)\dd r- \frac{1}{2}(t-s)p_{2\delta}(0)\right)\bigg|\Upsilon_{t-s}=y\right],
\]
where the conditional expectation is taken over the reflected Brownian bridges started from $\Upsilon_0=x$ and ending at $\Upsilon_{t-s}=y$.
Since 
\[
\{\xi_{\delta}(s+r,x)\}_{r \in [0,t-s], x \in [0, \infty)} \stackrel{\text { law }}{=} \{\xi_{\delta}(t-r,x)\}_{r \in [0,t-s], x \in [0, \infty)},\]
we have
\begin{equation}\label{eq:equalinlawRE}
\begin{aligned}
&\left\{\mathbb{E}_{\Upsilon}\left[\exp\left(-\mu L^{0,\Upsilon}_{t-s} +
\int_{0}^{t-s}\xi_{\delta}(t-r,\Upsilon_r)\dd r\right)\bigg\mid \Upsilon_0=x, \Upsilon_{t-s}=y\right]\right\}_{x,y \in [0,\infty)}\\
\stackrel{\text { law }}{=} &
\left\{\mathbb{E}_{\Upsilon}\left[\exp\left(-\mu L^{0,\Upsilon}_{t-s} +
\int_{0}^{t-s}\xi_{\delta}(s+r,\Upsilon_r)\dd r\right)\bigg\mid \Upsilon_0=x, \Upsilon_{t-s}=y\right]\right\}_{x,y \in [0,\infty)}.
\end{aligned}\end{equation} Let $r'=t-s-r$, $\tilde{\Upsilon}_{r'}:=\Upsilon_{t-s-{r'}}$ for $0\leq r'\leq t-s$ and ${L}_{t-s}^{0,\tilde{\Upsilon}}$ be the local time at zero associated with $\tilde{\Upsilon}$. The
tuples $((\Upsilon_{t-s-r})_{r \in [0,t-s]},L_{t-s}^{0,\Upsilon})$ and $((\tilde{\Upsilon}_r)_{r \in [0,t-s]},{L}_{t-s}^{0,\tilde{\Upsilon}})$ are identical $\omega$-by-$\omega$, and the process $(\tilde{\Upsilon}_{r'})_{0\leq r'\leq t-s}$ is a reflected Brownian bridge started from $\tilde{\Upsilon}_{0}=y$ and ending at $\tilde{\Upsilon}_{t-s}=x$. It then follows that the right-hand side of \eqref{eq:equalinlawRE} equals to
\[\left\{\mathbb{E}_{\tilde{\Upsilon}}\left[\exp\left(-\mu {L}^{0,\tilde{\Upsilon}}_{t-s} +
\int_{0}^{t-s}\xi_{\delta}(t-r',\tilde{\Upsilon}_{r'})\dd r'\right)\bigg\mid \tilde{\Upsilon}_0=y, \tilde{\Upsilon}_{t-s}=x\right]\right\}_{x,y \in [0,\infty)}.\]
Since the deterministic function $\kernel^{N}(t,x\viv s,y)$ is also symmetric in $(x,y)$-variables, we have proved that for any $\delta>0$,
\[
\{\mathcal{Z}^\delta_\mu(t,x\viv s,y)\}_{x, y \in [0, \infty)} \stackrel{\text { law }}{=}\{\mathcal{Z}^\delta_\mu(t,y\viv s,x)\}_{x, y \in [0, \infty)}.\]
{By  Lemma~\ref{le:fk}~\eqref{it:approx} (with a mollification in space only),} $\mathcal{Z}_\mu^{\delta}(t,x\viv s, y)$ converges to $\mathcal{Z}_\mu(t,x\viv s, y)$ in $L^p(\Omega)$ as $\delta \to 0$. We   thus derive \eqref{eq:trev}.

\subsection{Proof of Proposition~\ref{pr:mombds0}}\label{ap:momentsproof}
For any $a,u \in \R$, $p \in [1,\infty)$, and $x \in [0,\infty)$, we have identity
\begin{equation}\label{eq:expBMbd}
 \left\|\exp \left(aW_{u}(x)\right)\right\|_p = \exp\left(aux+\frac{1}{2}a^2px\right).
 \end{equation}
By \eqref{eq:sheGreensform}, Minkowski and Cauchy-Schwarz inequalities, for any $t>0$,
\[\begin{aligned}
\|&Z_{u-\frac{1}{2}}(t,x)\|_p \leq \int_0^{\infty} \|\mathcal{Z}_{u-\frac{1}{2}}(t,x\viv 0,y)\exp(W_u(y))\|_p\dd y \\&\leq  \int_0^{\infty} \|\mathcal{Z}_{u-\frac{1}{2}}(t,x\viv 0,y)\|_{2p}\|\exp(W_u(y))\|_{2p}\dd y.\\
\end{aligned}
\]
Using Lemma~\ref{le:fk}~\eqref{it:momentsFK} and \eqref{eq:expBMbd}
proves \eqref{eq:posCx}. 

For the bound of negative moments, we use Jensen's, Minkowski and Cauchy-Schwarz inequalities to obtain
\[
\begin{aligned}
&\Expe \left[ Z_{u-\frac{1}{2}}(t,x)^{-p}\right] = \Expe\left[ \left( \int_0^{\infty}\mathcal{Z}_{u-\frac{1}{2}}(t,x\viv 0,y)\exp (W_u(y)) \dd y\right)^{-p}\right]\\
&\leq \Expe \left[\left(\int_0^{\infty}\mathcal{Z}_{u-\frac{1}{2}}(t,x\viv 0,y) \dd y\right)^{-p-1} \int_0^{\infty}\mathcal{Z}_{u-\frac{1}{2}}(t,x\viv 0,y)\exp (-pW_{u}(y)) \dd y \right]\\
& \leq \left\| \left(\int_0^{\infty}\mathcal{Z}_{u-\frac{1}{2}}(t,x\viv 0,y) \dd y\right)^{-1} \right\|^{p+1}_{2p+2} \left\| \int_0^{\infty}\mathcal{Z}_{u-\frac{1}{2}}(t,x\viv 0,y)\exp (-pW_{u}(y)) \dd y\right\|_2\\
& \leq  \left\| \left(\int_0^{\infty}\mathcal{Z}_{u-\frac{1}{2}}(t,x\viv 0,y) \dd y\right)^{-1} \right\|^{p+1}_{2p+2}  \int_0^{\infty}\left\|\mathcal{Z}_{u-\frac{1}{2}}(t,x\viv 0,y)\right\|_4 \left\|\exp (-pW_{u}(y))\right\|_4 \dd y.
\end{aligned}
\]
By Lemma~\ref{le:negmom},
the first term in the last product is uniformly bounded on $x \in [0,\infty)$. The bound of the second term follows from a similar computation.

The constants in Lemma~\ref{le:fk}~\eqref{it:momentsFK} or Lemma~\ref{le:negmom} do not depend on $t$ when $t\in (0,\tau]$ for any fixed $\tau>0$ and can be chosen uniformly for $u$ in any compact subset of $\R$.  Thus it is not hard to check that the bounds of $\|Z_{u-\frac{1}{2}}(t,x)\|_p$ and $\|Z_{u-\frac{1}{2}}(t,x)^{-1}\|_p$ can also be chosen uniformly for $(u,t)$ in any compact subset of $\R\times {[0,\infty)}$.

{Using the uniform negative moment bounds, by Kolmogorov continuity criteria and the uniqueness of the mild solution, $Z_{u-\frac{1}{2}}(t,x)^{-1}$ is also continuous on $(t,x)\in \R^2_{\geq 0}$, which implies that $Z_{u-\frac{1}{2}}(t,x)>0$ for all $(t,x)\in \R^2_{\geq 0}$.}

To prove \eqref{eq:HbdCx}, we use an elementary inequality: for any $a \in\R$, $|a| \leq e^a+e^{-a}$. By the definition of $H_u(t,x)$,  
$ |H_u(t,x)| \leq Z_{u-\frac{1}{2}}(t,x)+Z_{u-\frac{1}{2}}(t,x)^{-1}$,
and \eqref{eq:HbdCx} follows from the triangle inequality.

\section{Proof of \eqref{eq:covlimit}}\label{ap:covlimit}

Recall that the goal was to show 
\[
\Cov[\mathcal{H}_u(t,x), \mathcal{H}_u(t,0)]  \to 0, \quad \text{as } x \to \infty.
\]
The above convergence  arises from the spatial decorrelation of the half-space KPZ equation \eqref{eq:hsKPZ}, for fixed $t>0$. Similar to the analogous result for the full-space KPZ equation (see  \cite[Proposition 5.2]{balazs2011fluctuation} and \cite[Eq. (2.10)]{gu2023another} for two different proofs), this decorrelation is independent of the KPZ behavior or the boundary condition. 

Using the same notation as above, $\mathscr{W}$ is the spatial white noise associated with the Brownian motion $W$ and $\malD$ is the Malliavin derivative operator on the Gaussian space generated by $\mathscr{W}$. We further define $\mathscr{D}$ to be the Malliavin derivative operator on the Gaussian space generated by the noise $\xi$. For any $t\geq 0, x\geq 0$, $\Expe H_u (t,x) = \mathbf{E}_W\mathbf{E}_\xi H_u(t,x)$. The law of total covariance gives that
\[
\begin{aligned}
&\Cov[\mathcal{H}_u(t,x), \mathcal{H}_u(t,0)] =  \mathbf{E}_W\mathbf{E}_\xi \left( \left[\mathcal{H}_u(t,x) - \mathbf{E}_\xi \mathcal{H}_u(t,x)\right] \left[\mathcal{H}_u(t,0) - \mathbf{E}_\xi \mathcal{H}_u(t,0)\right] \right) \\&+ \mathbf{E}_W\left( \left[\mathbf{E}_\xi \mathcal{H}_u(t,x) - \Expe \mathcal{H}_u(t,x) \right] \left[\mathbf{E}_\xi \mathcal{H}_u(t,0) - \Expe \mathcal{H}_u(t,0) \right] \right).
\end{aligned}
 \]
 For each realization of $W$, by the Clark-Ocone formula   \cite[Proposition 6.3]{chen2021spatial}, we have \[
\mathcal{H}_u(t,x) - \mathbf{E}_\xi \mathcal{H}_u(t,x) = \int_0^t \int_{\mathbb{R}} \mathbf{E}_{\xi}\left[\mathscr{D}_{s, z} \mathcal{H}_u(t, x) \mid \mathcal{F}_s\right] \xi(\dd s \dd z),
\] 
with $\{\mathcal{F}_s\}_{s\geq 0}$ being the natural filtration in time generated by $\xi$. Similarly, 
\[
\mathbf{E}_\xi \mathcal{H}_u(t,x) - \Expe \mathcal{H}_u(t,x) = \int_0^{\infty} \mathbf{E}_W[\malD_r\mathbf{E}_\xi \mathcal{H}_u(t,x)\mid \tilde{\mathcal{F}}_r ]\dd W(r),
\]
with $\{\tilde{\mathcal{F}}_r\}_{r\geq 0}$ being the natural filtration generated by $\mathscr{W}$.
By It\^o isometry, Cauchy-Schwarz inequality and Jensen's inequality,
\[
\begin{aligned}
&\Expe\left( \left[\mathcal{H}_u(t,x) - \mathbf{E}_\xi \mathcal{H}_u(t,x)\right] \left[\mathcal{H}_u(t,0) - \mathbf{E}_\xi \mathcal{H}_u(t,0)\right] \right)\\
& = \int_0^t\int_\R \Expe\left[ \mathbf{E}_{\xi}\left[\mathscr{D}_{s, z} \mathcal{H}_u(t, x) \mid \mathcal{F}_s\right]  \mathbf{E}_{\xi}\left[\mathscr{D}_{s, z} \mathcal{H}_u(t, 0) \mid \mathcal{F}_s\right] \right] \dd z \dd s\\
&\leq \int_0^t\int_\R  \|\mathscr{D}_{s, z} \mathcal{H}_u(t, x)\|_2  \|\mathscr{D}_{s, z} \mathcal{H}_u(t, 0)\|_2 \dd z \dd s=:I_1(t,x),
\end{aligned}
\]
and 
\[\begin{aligned}
&\mathbf{E}_W\left( \left[\mathbf{E}_\xi \mathcal{H}_u(t,x) - \Expe \mathcal{H}_u(t,x) \right] \left[\mathbf{E}_\xi \mathcal{H}_u(t,0) - \Expe \mathcal{H}_u(t,0) \right] \right)\\
&=\mathbf{E}_W\left[  \int_0^{\infty} \mathbf{E}_W[\malD_r\mathbf{E}_\xi \mathcal{H}_u(t,x)\mid \tilde{\mathcal{F}}_r] \mathbf{E}_W[\malD_r\mathbf{E}_\xi \mathcal{H}_u(t,0)\mid \tilde{\mathcal{F}}_r ]\dd r\right]\\
&\leq \int_0^{\infty}\|\malD_r \mathcal{H}_u(t,x)\|_2\|\malD_r \mathcal{H}_u(t,0)\|_2 \dd r=:I_2(t,x).
\end{aligned}
\]
The proof of \eqref{eq:covlimit} reduces to showing that for any fixed $t>0$, $I_1(t,x)+I_2(t,x) \to 0$ as $x \to \infty$.

We first analyze $I_1$. Similar to the result of full-space SHE in \cite[Proposition 5.1]{chen2021regularity}, for $(s,z)\in [0,t)\times \R$ and $x\geq 0$, we have
\[
\begin{aligned}
\mathscr{D}_{s, z} {Z}_{u-\frac{1}{2}}(t, x) &=\1_{[0, \infty)}(z) \kernel_{u-\frac{1}{2}}^R(t,x\viv s,z) {Z}_{u-\frac{1}{2}}(s, z) \\
&+ \int_s^t\int_0^{\infty} \kernel_{u-\frac{1}{2}}^R(t,x\viv r,y) \mathscr{D}_{s,z} {Z}_{u-\frac{1}{2}}(r, y) \xi(r,y) \dd y\dd r,
\end{aligned}
\]
with $\mathscr{D}_{s, z} {Z}_{u-\frac{1}{2}}(t, x) \equiv 0$ when $z<0$ or $s\geq t$. When $s<t, z \geq 0$, since ${Z}_{u-\frac{1}{2}}(s, z)$ is strictly positive, we can multiply both sides by ${Z}_{u-\frac{1}{2}}(s, z)^{-1} $. The uniqueness statement in Proposition~\ref{pr:hsSHEprop}~(2) then implies that \[
\mathscr{D}_{s, z} {Z}_{u-\frac{1}{2}}(t, x) {Z}_{u-\frac{1}{2}}(s, z)^{-1} = \mathcal{Z}_{u-\frac{1}{2}}(t,x\viv s,z),
\]
where we recall that $\mathcal{Z}$ is the Green's function.
It follows that for any $s\in[0,t), z \in [0, \infty)$,
\[
\begin{aligned}
\mathscr{D}_{s, z} \mathcal{H}_{u}(t, x) &= \mathscr{D}_{s, z} \left[\log{Z}_{u-\frac{1}{2}}(t, x)\right]
 = \frac{\mathscr{D}_{s, z} {Z}_{u-\frac{1}{2}}(t, x)}{Z_{u-\frac{1}{2}}(t,x)}=\frac{ \mathcal{Z}_{u-\frac{1}{2}}(t,x\viv s,z) {Z}_{u-\frac{1}{2}}(s, z)}{Z_{u-\frac{1}{2}}(t,x)}.
\end{aligned}
\]
By H\"older's inequality, \eqref{eq:greensPosMom}, \eqref{eq:posCx} and \eqref{eq:negCx}, there exists some positive constant $C=C(t,u)$ such that for any $x \in [0,\infty)$,
\[
\begin{aligned}
\|\mathscr{D}_{s, z} \mathcal{H}_{u}(t, x) \|_2 &\leq \|\mathcal{Z}_{u-\frac{1}{2}}(t,x\viv s,z)\|_6  \|{{Z}}_{u-\frac{1}{2}}(s,z)\|_6  \|{{Z}}_{u-\frac{1}{2}}(t,x)^{-1}\|_6\\
&\leq C(t-s)^{-1/2}\exp(-(x-z)^2/4(t-s)) \exp(C(x+z)).
\end{aligned}
\]
Thus we have
\[
I_1(t,x) \leq  C^2 \int_0^t\int_{0}^{\infty}  (t-s)^{-1}\exp\left(-\frac{(x-z)^2}{4(t-s)} - \frac{z^2}{4(t-s)}\right)  \exp(C(x+2z)) \dd z \dd s.
\]
By a straightforward computation, the right-hand side of the above converges to 0 as $x \to \infty$.

For $I_2$, we note that 
\[
\begin{aligned}
\malD_r \mathcal{H}_u(t,x) &= \frac{\int_r^{\infty}\mathcal{Z}_{u-\frac{1}{2}}(t,x\viv 0,y)\exp (W_{u}(y)) \dd y}{\int_0^{\infty}\mathcal{Z}_{u-\frac{1}{2}}(t,x\viv 0,y)\exp (W_{u}(y)) \dd y} - \1_{[0,x]}(r) \\
&= \1_{[x,\infty)}(r)\int_r^{\infty}\rho^{R}_{u}(y \viv t,x) \dd y - \1_{[0,x]}(r)\int_0^{r}\rho^{R}_{u}(y \viv t,x) \dd y,
\end{aligned}
\]
which further implies 
\[
\|\malD_r \mathcal{H}_u(t,x)\|_2 \leq \1_{[0,x]}(r)\int_0^{r} \|\rho^{R}_{u}(y \viv t,x)\|_2 \dd y + \1_{[x,\infty)}(r)\int_r^{\infty}\|\rho^{R}_{u}(y \viv t,x)\|_2 \dd y.
\]
By a similar estimate as in \eqref{eq:quedenmom}, we can prove that there exists a constant $C=C(t,u)$ such that
\[
\|\rho^{R}_{u}(y \viv t,x)\|_2 \leq  C\exp\left(-{(x-y)^2/C}+Cx\right), \quad \text{for all } x \in [0,\infty).\]
As a result, we can bound $I_2$ by 
\[
\begin{aligned}
&I_2(t,x) \leq \exp({Cx})\int_0^x\left( \int_0^{r} C\exp\left(-{(x-y)^2}/{C}\right) \dd y \right)\left(\int_r^{\infty} C\exp\left(-{y^2}/{C}\right)  \dd y\right) \dd r \\
 &+ \exp({Cx})\int_x^{\infty}\left( \int_r^{\infty} C\exp\left(-{(x-y)^2}/{C}\right) \dd y \right)\left(\int_r^{\infty} C\exp\left(-{y^2}/{C}\right)  \dd y\right) \dd r\\
  & \leq \exp({Cx}) \int_0^x \phi(2x-r)\phi(r)\dd r + \exp({Cx}) \int_x^{\infty} \phi(r-x)\phi(r) \dd r,
\end{aligned}
\]
with $\phi(x):= \int_x^{\infty} C\exp\left(-{y^2}/{C}\right)  \dd y$. By the Gaussian tail of the function $\phi$, the last line above converges to 0 as $x \to \infty$.

\bibliographystyle{alpha}
\bibliography{hskpz.bib}

\end{document}